\documentclass[a4paper, 10pt, twoside, notitlepage]{amsart}

\usepackage{amsmath,amscd}
\usepackage{amssymb}
\usepackage{amsthm}
\usepackage{comment}
\usepackage{graphicx, xcolor}
\usepackage{booktabs}
\usepackage{makecell}
\usepackage{mathrsfs}
\usepackage{float}
\usepackage[ocgcolorlinks, linkcolor=blue]{hyperref}

\usepackage{bm}
\usepackage{bbm}
\usepackage{url}

\usepackage[utf8]{inputenc}
\usepackage{mathtools,amssymb}
\usepackage{esint}
\usepackage{tikz}
\usepackage{dsfont}
\usepackage{relsize}
\usepackage{url}
\usepackage{xcolor}
\usepackage{graphicx}
\usepackage{mathrsfs}
\usepackage[shortlabels]{enumitem}
\usepackage{lineno}
\usepackage{amsmath}
\usepackage{enumitem}
\usepackage{amsthm} 
\usepackage{verbatim}
\usepackage{dsfont}
\numberwithin{equation}{section}

\usepackage{slashed}

\allowdisplaybreaks

\mathtoolsset{showonlyrefs}

\graphicspath{{images}}

\newtheorem{theorem}{Theorem}[section]
\newtheorem{lemma}[theorem]{Lemma}
\newtheorem{definition}[theorem]{Definition}
\newtheorem{corollary}[theorem]{Corollary}

\newtheorem{claim}[theorem]{Claim}
\newtheorem{proposition}[theorem]{Proposition}
\newtheorem{example}[theorem]{Example}

\newtheorem{remark}[theorem]{Remark}

\title[Third order nonlocal wave equations]{The Calder\'on problem for third order nonlocal wave equations with time-dependent nonlinearities and potentials}

\author[S. Fu]{Song-Ren Fu}
\address{School of Mathematics and Statistics, Northwestern Polytechnical University, Xi'an 710129, People's Republic of China.}
\email{songrenfu@nwpu.edu.cn}

\author[Y. Yu]{Yongyi Yu}
\address{School of Mathematics, Jilin University, Changchun 130012, China}
\email{yuyy122@jlu.edu.cn}

\author[P. Zimmermann]{Philipp Zimmermann}
\address{Corresponding author: Departament de Matem\`atiques i Inform\`atica, Universitat de Barcelona, Barcelona, Spain}
\email{philipp.zimmermann@ub.edu}


\newcommand{\R}{{\mathbb R}}
\newcommand{\Z}{{\mathbb Z}}
\newcommand{\N}{{\mathbb N}}

\newcommand{\eps}{\varepsilon}

\newcommand{\tempered}{\mathscr{S}^{\prime}}

\newcommand{\fourier}{\mathcal{F}}
\newcommand{\ifourier}{\mathcal{F}^{-1}}

\newcommand{\distr}{\mathscr{D}^{\prime}}%distribution
\DeclareMathOperator{\Div}{div} %divergence
\DeclareMathOperator{\loc}{loc} %distance

\newcommand{\weak}{\rightharpoonup}%weak convergence 
\newcommand{\weakstar}{\overset{\ast}{\rightharpoonup}}%weak star convergence 

\def\R{\mathbb{R}}

\def\<{\langle}
\def\>{\rangle}

\def\Om{\Omega}

\def\var{\varphi}

\def\det{\mbox{det\,}}

\begin{document}

	\begin{abstract}
    In this article, we study the Calder\'on problem for nonlocal generalizations of the semilinear Moore--Gibson--Thompson (MGT) equation and the Jordan--Moore--Gibson--Thompson (JMGT) equation of Westervelt-type. These partial differential equations are third order wave equations that appear in nonlinear acoustics, describe the propagation of high-intensity sound waves and exhibit finite speed of propagation. For semilinear MGT equations with nonlinearity $g$ and potential $q$, we show the following uniqueness properties of the Dirichlet to Neumann (DN) map $\Lambda_{q,g}$:
    \begin{enumerate}[(i)]
        \item If $g$ is a polynomial-type nonlinearity whose $m$-th order derivative is bounded, then $\Lambda_{q,g}$ uniquely determines $q$ and $(\partial^{\ell}_\tau g(x,t,0))_{2\leq \ell \leq m}$.
        \item If $g$ is a polyhomogeneous nonlinearity of finite order $L$, then $\Lambda_{q,g}$ uniquely determines $q$ and $g$.
    \end{enumerate}
    The uniqueness proof for polynomial-type nonlinearities is based on a higher order linearization scheme, while the proof for polyhomogeneous nonlinearities only uses a first order linearization. Finally, we demonstrate that a first linearization suffices to uniquely determine Westervelt-type nonlinearities from the related DN maps. We also remark that all the unknowns, which we wish to recover from the DN data, are allowed to depend on time.
    \end{abstract}
\maketitle    
	\tableofcontents
    
	\section{Introduction}
    \label{sec: intro}
In the present article, we study the Calder\'on problem for nonlocal generalizations of some partial differential equations (PDEs) that appear in nonlinear acoustics. For the purpose of readers' convenience in reviewing nonlinear acoustic, we introduce in the next subsection some background of the related classical acoustic models.

\subsection{Background of acoustic waves}    
\label{subsec: background acoustic waves}

In physical terms, nonlinear acoustics describe the propagation of large-amplitude sound waves in fluids or solids and often fall into the class of higher order nonlinear damped wave equations. A major difficulty in these models is the fact that the large amplitudes of the waves prohibit the usage of linearization methods. To understand particular aspects of sound propagation, as dissipation or shock formation, different models have been proposed. As discussed in the scientific literature, some subclasses of these models can be put into a hierarchical order. For example, the widely studied Westervelt and Kuznetsov equation can be obtained via certain limit processes from the more general \emph{Brunnhuber--Jordan--Kuznetsov (BJK) equation} that reads
    \begin{equation}
    \label{eq: BJK equation}
    \begin{split}
        &(\partial_t^3-(\beta+\nu \Lambda)\Delta \partial_t^2+\beta\nu \Lambda \Delta^2\partial_t-c^2\Delta\partial_t+\beta c^2\Delta^2)\psi\\
        &=\partial_t^2\bigg(\frac{B}{2A c^2}(\partial_t \psi)^2+|\nabla \psi|^2\bigg).
    \end{split}
    \end{equation}
    Here, we put $\beta=a(1+B/A)$ and the parameters have the following physical interpretation:\\
    
\begin{tabular}{ c l }
$\psi$ & velocity (or acoustic) potential   \\ 
$c$ & speed of sound  \\  
$\nu$ & kinematic viscosity \\
$\Lambda$ & measures ratio of bulk viscosity $\mu_B$ and dynamic viscosity $\mu$\\
$a=\nu/\text{Pr}$  & thermal conductivity, $\text{Pr}$ denotes the Prandtl number\\
$B/A$ & parameter of nonlinearity
\end{tabular}\\

We point out that $-\nabla \psi$ corresponds to the irrotational part in the Helmholtz decomposition of the (relative) particle velocity field $v$, whereas the incompressible part is the rotation of the vector field $\mathcal{A}$. 

In the limit $a\to 0$, while keeping $\nu$ constant, the PDE \eqref{eq: BJK equation} converges to the \emph{Kuznetsov equation}
\begin{equation}
\label{eq: Kuznetsov}
    (\partial_t^2-\nu\Lambda\Delta \partial_t-c^2\Delta)\psi=\partial_t\left(\frac{B}{2A c^2}(\partial_t\psi)^2+|\nabla\psi|^2\right)
\end{equation}
(see \cite[Sections 2 and 4]{kaltenbacher2018fundamental}). This PDE has been first reported in \cite{Kuznetsov} and we refer the interested reader also to \cite{jordan2004analytical,dekkers2020models}. Furthermore, the PDE \eqref{eq: BJK equation} reduces to the \emph{Westervelt equation}
\begin{equation}
\label{eq: Westervelt}
    (\partial_t^2-\nu\Lambda\Delta \partial_t-c^2\Delta)\psi=\frac{1}{2c^2}\left(2+\frac{B}{A}\right)\partial_t(\partial_t\psi)^2,
\end{equation}
 when one takes the limit $a\to 0$ and neglects local nonlinear effects (see \cite[Section 4]{kaltenbacher2018fundamental}). The latter approximation is often physically justified when the propagation distance is large compared to the wavelength. In those cases, one may assume that
 \begin{equation}
 \label{small wavelength approx}
    |\nabla \psi|^2\approx \frac{1}{c^2}(\partial_t\psi)^2,
 \end{equation}
 which explains the constant on the right hand side of \eqref{eq: Westervelt}. Next, we want to express the Westervelt equation in terms of the acoustic pressure 
 \begin{equation}
 \label{eq: acoustic pressure}
    p=\rho\partial_t \psi.
 \end{equation}
 If we formally differentiate \eqref{eq: Westervelt} in time and use the relation \eqref{eq: acoustic pressure}, then we find
 \begin{equation}
 \label{eq: Westervelt pressure formulation}
     (\partial_t^2-\nu\Lambda\Delta \partial_t-c^2\Delta)p= \partial_t^2(\beta p^2),
 \end{equation}
 where we put
 \begin{equation}
     \beta=\frac{1}{4\rho c^2}\left(2+\frac{B}{A}\right).
 \end{equation}

 The BJK equation \eqref{eq: BJK equation} can be derived from the conservation of mass, momentum, energy and a thermodynamical equation of state, which relates the mass density $\rho$, the acoustic pressure $p$ and the temperature $T$ (see, e.g., \cite[Appendix]{kaltenbacher2018fundamental}). On the one hand, the conservation laws for mass and momentum are nothing more than the continuity equation and the compressible Navier--Stokes equations with uniform viscosities. On the other hand, the conservation law for the energy is used in the following form:
\begin{equation}
\label{eq: energy balance eq}
    \rho(\partial_t E+v\cdot \nabla E)+p\nabla \cdot v=a\Delta T+(\mu_B-\frac{2}{3}\mu)(\nabla\cdot v)^2+\frac{\mu}{2}|\nabla v+(\nabla v)^T|^2.
\end{equation}
Here, $E$ denotes the internal energy per unit mass and henceforth it is assumed that the heat flux $q$ obeys \emph{Fourier's law} of heat conduction,
\begin{equation}
\label{eq: Fourier's law}
    q=-a\nabla T.
\end{equation}
It is a well-known fact that this can lead to an infinite signal speed paradox in waves (see, for example, \cite{chandrasekharaiah1986thermoelasticity}). To retain a finite speed of propagation, other relations have been proposed, such as \emph{Maxwell--Cattaneo law}
\begin{equation}
\label{eq: Maxwell Cattaneo}
    \tau\partial_t q+q=-a\nabla T,
\end{equation}
where $\tau$ is a small constant that describes the relaxation time of the heat flux. We remark that, physically speaking, \eqref{eq: Maxwell Cattaneo} models a time lag between the temperature gradient $\nabla T$ and the heat flux. Using this constitutive law in the derivation of the energy balance, one gets \emph{Jordan--Moore--Gibson--Thompson (JMGT) equation}
\begin{equation}
\label{eq: JMGT eq}
(\tau\partial_t^3+\partial_t^2-b\Delta\partial_t-c^2\Delta)\psi=\partial_t\left(\frac{1}{c^2}\frac{B}{2A}(\partial_t\psi)^2+|\nabla\psi|^2\right),
\end{equation}
where $b=\delta+\tau c^2$ and $\delta$ is the diffusity of sound. Due to the similarity of \eqref{eq: JMGT eq} to \eqref{eq: BJK equation}, we also say \eqref{eq: JMGT eq} is of \emph{Kuznetsov-type}. If one neglects local nonlinear effects, then one finds the following \emph{Westervelt-type JMGT equations} 
\begin{equation}
\label{eq: westervelt JMGT eq}
(\tau\partial_t^3+\partial_t^2-b\Delta\partial_t-c^2\Delta)\psi=\partial_t(\kappa(\partial_t\psi)^2)
\end{equation}
and
 \begin{equation}
 \label{eq: Westervelt JMGT eq pressure formulation}
     (\tau\partial_t^3+\partial_t^2-b\Delta \partial_t-c^2\Delta)p= \partial_t^2(\beta p^2),
 \end{equation}
where we set $\kappa=2\rho\beta$ (cf.~\eqref{eq: Westervelt} and \eqref{eq: Westervelt pressure formulation}). Furthermore, if one neglects all quadratic terms in \eqref{eq: JMGT eq}, then one ends up with the so-called \emph{Moore--Gibson--Thompson (MGT) equation}
\begin{equation}
\label{eq: MGT eq}
(\tau\partial_t^3+\partial_t^2-b\Delta\partial_t-c^2\Delta)\psi=0.
\end{equation}
In \cite{Kalten}, it has been shown that in the singular limit $\tau \to 0$ one obtains from \eqref{eq: JMGT eq} the Kuznetsov equation \eqref{eq: Kuznetsov}.

Physically, the (J)MGT equations describe the propagation of high-intensity sound waves that have a wide range of applications, for example, thermotherapy, ultrasound cleaning and sonochemistry, see, e.g., \cite{OVA,MKHL}. We also refer to the monographs \cite{BOECM,OVSI} offering a more physical or mathematical point of view. In recent years, many researchers have studied the well-posedness, stability, controllability and inverse problems for the linear MGT and the nonlinear JMGT equation (see, for example, \cite{CWH,SRPFYY,Kalten,BKILR,BKVN,LiuT,CLSZ} and the references therein).

The MGT equation \eqref{eq: MGT eq} differs from the classical second-order (in time) wave equation in several fundamental aspects, including the analysis of well-posedness, long-time dynamics, and controllability. More precisely, this equation exhibits a rich variety of dynamical behaviors that crucially depend on the physical parameters appearing in the model. For instance, the diffusivity parameter $b$ has a decisive impact on the stability of solutions to the MGT equation, and well-posedness may even fail in the simplest case $b=0$; see \cite[Theorem 1.1]{BKILR}. Moreover, the construction of geometric optics solutions, which play a central role in inverse problems for local MGT equations, is considerably more involved due to the presence of higher-order time derivatives. By contrast, for second-order wave equations, the presence of structural/viscous damping does not affect well-posedness. 

The controllability properties of the (local) MGT equation, in comparison with those of the classical wave equation, are significantly more delicate. In particular, it was shown in \cite{lizama2023boundary} that even in one spatial dimension the MGT equation is neither exactly nor null controllable by a boundary-supported control. This poor controllability is closely related to the presence of the damping term $(-\Delta)\partial_t u$, which, in the local setting, generates accumulation points in the spectrum; see \cite{CLSZ}. As a consequence, if a Carleman estimate were to hold for the MGT equation, it would imply boundary controllability, leading to a contradiction with the aforementioned negative controllability results (see \cite{arancibia2020inverse} for further discussion). Moreover, to the best of our knowledge, there are no unique continuation property results in the literature for local MGT equations from solutions vanishing either on the boundary or on a subset of the interior domain. By contrast, due to the nonlocal nature of the fractional Laplacian, unique continuation holds for the nonlocal MGT equation (see, e.g., \eqref{eq: control PDE}).
	
Apart from existing works on the above (J)MGT equations, some works considered time fractional modifications of the (J)MGT equation (see, e.g. \cite{EACLM,BKVN0,BKVN1,CLMZ,MM}). Furthermore, the article \cite{CLMZ} considers the controllability of the following \emph{nonlocal MGT equation}
    \begin{equation}
    \label{eq: control PDE}
    \begin{cases}
    (\tau\partial_t^3+\alpha\partial_t^2+b(-\Delta)^s\partial_t+c^2(-\Delta)^s)\psi=0, &{\rm in}\ \Omega_T, \\
    \psi=g\chi_{\mathcal O} &{\rm in}\ (\Omega_e)_T,\\
    \psi(0)=\psi_0, \partial_t\psi(0)=\psi_1, \partial_t^2\psi(0)=\psi_2 &{\rm in}\ \Omega,
    \end{cases}
    \end{equation}
where $\Omega\subset\R^n$ is a bounded domain with exterior $\Omega_e\vcentcolon = \mathbb R^n\backslash\overline{\Omega}$, $T>0$ a finite time horizon, $\mathcal O\subset \Omega_e$ a control set and $g\in H^1(0,T;L^2(\mathcal O))$ a control function. Furthermore, $(-\Delta)^s$ denotes the \emph{fractional Laplacian} of order $s>0$ that we discuss in more detail in Section \ref{sec: preliminaries}. Moreover, here and throughout this article, if $A\subset\R^n$ is any set and $t>0$, then we write $\chi_A$ for its characteristic function and $A_t$ for the associated space time cylinder $A\times (0,t)$. 
Under some conditions on the parameters, an approximate controllability property of the nonlocal MGT equation has been established. 
    
We point out that nonlocal wave equations arise in peridynamics, which studies the dynamics of materials with discontinuities, and we refer to \cite{YMP,SAS,ZLB} for a detailed account of this theory.  It is worthwhile to mention that nonlocal operators, like the fractional Laplacian or more general integro-differential operators, appear in various branches of science, including geometry, financial mathematics, fluid dynamics and biology (cf., e.g.,~\cite{CS-extension-problem-fractional-laplacian,SI-regularity-obstacle-problem,colombo2020generalized,krieger2017small,caselli2023yau} and the references therein).

	\subsection{Nonlocal MGT and JMGT equations}
	
    The damped nonlocal wave equations studied in this article are of the form
    \begin{equation}
    \label{eq: main nonlocal PDE}
	\begin{cases}
		(\partial_t^3+\alpha \partial_t^2+b(-\Delta)^s\partial_t+c(-\Delta)^s+q)u+G(u,\partial_t u,\partial_t^2u)=0 & {\rm in}\ \Omega_T, \\
		u=\varphi & {\rm in}\ (\Omega_e)_T,\\
		u(0)=\partial_tu(0)=\partial_t^2u(0)=0 &  {\rm in}\ \Omega,
	\end{cases}
    \end{equation}
    which represents a nonlocal generalization of the equations \eqref{eq: JMGT eq}--\eqref{eq: MGT eq}. Here the linear coefficient $q$ denotes some potential and $G$ is a suitable nonlinearity that will be specified below. For simplicity, we have rescaled the quantities appearing in the PDEs \eqref{eq: JMGT eq}--\eqref{eq: MGT eq} such that $\tau=1$. Moreover, throughout this work, we assume for the sake of generality that the constants in \eqref{eq: main nonlocal PDE} obey
    \[
        b>0\text{ and }\alpha,c\in\R,
    \]
    where $c$ corresponds to $c^2$ in the MGT equations from the previous paragraph. Note that the fractional Laplacian in the PDE \eqref{eq: main nonlocal PDE} leads to infinite speed of propagation, whereas the speed of propagation for solutions to the MGT equations \eqref{eq: JMGT eq}--\eqref{eq: MGT eq} is finite. This is one motivation for studying nonlocal generalizations of the MGT equations and not of the BJK equation \eqref{eq: BJK equation} as in the former case the infinite speed of propagation is entirely from the fractional Laplacian.
    
    In particular, we study the \emph{semilinear nonlocal MGT equation}
    \begin{equation}
        \label{sysmgt}
        \begin{cases}
		(\partial_t^3+\alpha \partial_t^2+b(-\Delta)^s\partial_t+c(-\Delta)^s+q)u+g(u)=0 & {\rm in}\ \Omega_T, \\
		u=\varphi & {\rm in}\ (\Omega_e)_T,\\
		u(0)=\partial_tu(0)=\partial_t^2u(0)=0 &  {\rm in}\ \Omega,
	\end{cases}
    \end{equation}
	and the \emph{Westerverlt-type nonlocal JMGT equations}
    \begin{equation}
        \label{sysmgtWtype}
	\begin{cases}
		( \partial_t^3+\alpha \partial_t^2+b(-\Delta)^s\partial_t+c(-\Delta)^s)u=\mathcal{F}(u,\partial_tu,\partial_t^2u) & {\rm in}\ \Omega_T, \\
		u=\varphi& {\rm in} \ (\Omega_e)_T,\\
		u(0)=\partial_tu(0)=\partial_t^2u(0)=0 & {\rm in}\ \Omega,
	\end{cases}
    \end{equation}
	where $\mathcal{F}$ is given by
    \begin{equation}
    \label{eq: Westervelt nonlinearities}
    \mathcal{F}(u,\partial_tu,\partial_t^2u)=\partial_t(\kappa(\partial_tu)^2)\text{ or }\mathcal{F}(u,\partial_tu,\partial_t^2u)=\partial_t^2(\beta u^2)
    \end{equation}
    corresponding to the nonlinearity appearing in the potential formulation \eqref{eq: westervelt JMGT eq} or the pressure formulation \eqref{eq: Westervelt JMGT eq pressure formulation} of the Westervelt equation, respectively.
	
	For the moment, let us assume that the PDEs \eqref{sysmgt} and \eqref{sysmgtWtype} are well-posed for all sufficiently nice exterior conditions $\varphi$. 
    Then, we may (formally) introduce the \emph{Dirichlet to Neumann (DN) map} $\Lambda_{q,g}$ related to \eqref{sysmgt} by
\begin{equation}\label{DtN0}
\begin{split}
	\Lambda_{q,g}&\colon C_c^\infty((\Omega_e)_T)\to L^2(0,T;H^{-s}(\Omega_e)),\\
    &\varphi\mapsto(b(-\Delta)^s\partial_tu+c(-\Delta)^su)|_{(\Omega_e)_T},
\end{split}
    \end{equation} 
where $u$ denotes the solution to \eqref{sysmgt} with exterior value $\varphi$. The precise definitions of the fractional Laplacian $(-\Delta)^s$ and the $L^2$-based Sobolev spaces
will be given in Section 2.1.
Similarly, for the PDE \eqref{sysmgtWtype}, the corresponding DN maps $\Lambda_\beta,\Lambda_\kappa$ are given by
\begin{equation}
\label{DtN3 intro}
\begin{split}
\Lambda_\beta\varphi&=(b(-\Delta)^s\partial_tu+c(-\Delta)^su)|_{(\Omega_e)_T},\\
\Lambda_\kappa\varphi&=(b(-\Delta)^s\partial_tu+c(-\Delta)^su)|_{(\Omega_e)_T}
\end{split}
\end{equation} 
with the same notation as above. 
These DN maps are rigorously introduced in Definition \ref{def: DN maps}. The nonlocal acoustic models, together with the associated DN maps, considered in this article are collected in Table~\ref{table1}. Therein, we use the following notation:
\begin{equation}
\label{eq: third order nonlocal wave operator}
L_{\alpha,b,c}\vcentcolon = \partial_t^3+\alpha\partial_t^2+b(-\Delta)^s \partial_t +c(-\Delta)^s.
\end{equation}
\begin{table}[ht]
\centering
\caption{Mathematical models, nonlocal equations, and DN maps}
\label{table1}
\begin{tabular}{c @{\quad} c @{\quad} c}
\toprule
\textbf{Mathematical model} 
& \textbf{Nonlocal equation} 
& \textbf{DN map} \\
\midrule
\makecell[c]{Semilinear nonlocal MGT equation\\ with polynomial-type nonlinearity}
& \makecell[c]{$(L_{\alpha,b,c}+q)u + g(u)=0$,\\ $g(x,t,\tau)$ satisfying \ref{A1}--\ref{A3}}
& $\Lambda_{q,g}\varphi$ \\
\midrule
\makecell[c]{Semilinear nonlocal MGT equation\\ with polyhomogeneous nonlinearity}
& \makecell[c]{$(L_{\alpha,b,c}+q)u + g(u)=0,$\\
$g(x,t,\tau)=\sum_{k=1}^L \alpha_k(x,t) |\tau|^{r_k}\tau$}
& $\Lambda_{q,g}\varphi$ \\
\midrule
\makecell[c]{Westervelt-type nonlocal\\ MGT equations}
& \makecell[c]{$L_{\alpha,b,c}u = \partial_t\!\big(\kappa(\partial_t u)^2\big)$\\
$L_{\alpha,b,c}u = \partial_t^2\!\big(\beta u^2\big)$}
& \makecell[c]{$\Lambda_\kappa\varphi$\\
$\Lambda_\beta\varphi$} \\
\bottomrule
\end{tabular}
\end{table}
Motivated by recent developments on the Calderón problem for nonlinear nonlocal wave equations--see, for instance, \cite{KMSK,LTZ1,LTZ2,PZ1}, we investigate in this article the Calderón problem for semilinear nonlocal MGT equations and for Westervelt-type JMGT equations. A detailed account of the related literature is provided in Section~\ref{sec: related literature}. Specifically, we aim to answer the following questions.
	\begin{enumerate}[\bf({IP-}1)]
	    \item\label{IP1} Does $\Lambda_{q,g}$ uniquely determine the potential $q$ and what information can be extracted about the nonlinearity $g$?
        \item\label{IP2} Does $\Lambda_\beta$ (resp. $\Lambda_\kappa$) uniquely determine $\beta$ (resp. $\kappa$)?
	\end{enumerate}

   \begin{remark}
    In the main theorems below (see Section~\ref{sec: assumptions}), we consider Questions~\ref{IP1} and~\ref{IP2} in the partial data setting. 
    For instance, in Question~\ref{IP1} this means that, for any measurement sets $W_1, W_2 \subset \Omega_e$, 
    we ask whether the knowledge of
    \[
        \left. \Lambda_{q,g}\varphi \right|_{(W_2)_T}\quad \text{for all admissible}\quad \varphi \in C_c^\infty\big((W_1)_T\big)
    \]
    uniquely determines the potential $q$ and the nonlinearity $g$.

Thus, throughout this article, we do not assume access to the full Dirichlet-to-Neumann map $\Lambda_{q,g}$, but only to its restriction to exterior Dirichlet data supported in $(W_1)_T$ and Neumann data taken in $(W_2)_T$.
    \end{remark}

Let us note that, in the inverse problem \ref{IP1}, one may also consider the measurement map
\begin{equation}
\widetilde{\Lambda}_{q,g}\varphi \vcentcolon= (-\Delta)^s u_\varphi \big|_{(\Omega_e)_T},
\end{equation}
in place of the DN map $\Lambda_{q,g}$. These two quantities are connected through the ODE relation
\begin{equation}
\label{DN-new}
    \Lambda_{q,g}\varphi = b\partial_t\big(\widetilde{\Lambda}_{q,g}\varphi\big) + c \widetilde{\Lambda}_{q,g}\varphi,
\end{equation}
for fixed $\varphi\in C_c^{\infty}((\Omega_e)_T)$. Consequently, if $(q_j,g_j)$, $j=1,2$, denote two sets of unknown potentials and nonlinearities, then
\begin{equation}
    \label{eq: equivalence DN maps}
    \widetilde{\Lambda}_{q_1,g_1}\varphi=\widetilde{\Lambda}_{q_2,g_2}\varphi \,\Leftrightarrow\, \Lambda_{q_1,g_1}\varphi=\Lambda_{q_2,g_2}\varphi.
\end{equation}
The implication $``\Rightarrow"$ follows directly from \eqref{DN-new}. For the reverse implication, note that $\Lambda_{\varphi}=\Lambda_{q_1,g_1}\varphi-\Lambda_{q_2,g_2}\varphi$ satisfies
\begin{equation}
\label{eq: equivalence of measurement operators}
    \begin{cases}
        \partial_t\Lambda+\frac{c}{b}\Lambda=0 & \text{ in }(\Omega_e)_T,\\
        \Lambda(0)=0 & \text{ in }\Omega_e.
    \end{cases}
\end{equation}
Here, we use that $\varphi\in C_c^{\infty}((\Omega_e)_T)$ implies $\widetilde\Lambda_{q,g}\varphi|_{t=0}=(-\Delta)^su_\varphi|_{t=0}=0$. By the fundamental theorem for ODEs, it follows that $\Lambda=0$ in $(\Omega_e)_T$, completing the argument. 

In the present article, despite the equivalence \eqref{eq: equivalence DN maps}, we choose to work with the DN map $\Lambda_{q,g}$ \eqref{DtN0}, as it is the quantity naturally associated with the nonlocal MGT equation and leads to more concise integration-by-parts identities (see Lemma~\ref{integral}).
An analogous remark applies to \ref{IP2} and the DN maps $\Lambda_{\beta}$ and $\Lambda_{\kappa}$.

	\subsection{Assumptions and main results}
    \label{sec: assumptions}

    In this paragraph, we specify the smoothness and decay properties which the potential $q$, the nonlinearity $g$ and the Westervelt parameters $(\beta,\kappa)$, appearing in \eqref{sysmgt}--\eqref{sysmgtWtype}, shall satisfy throughout this work. These are very similar to those imposed in the articles \cite{PZ1,LTZ1,LTZ2}. We illustrate the assumptions with concrete examples and present our main findings for the questions \ref{IP1}--\ref{IP2}.

	We suppose that the potential $q$ and the Westervelt parameters $(\beta,\kappa)$ satisfy
    \begin{equation}
    \label{eq: regularity pot westervelt par}
        (q,\beta,\kappa)\in L^\infty(0,T;L^p(\Omega))\times W^{2,\infty}(0,T;L^{\infty}(\Omega))\times W^{1,\infty}(0,T;L^{\infty}(\Omega)),
    \end{equation}
    where the integrability exponent $2\leq p\leq \infty$ obeys
    \begin{equation}
        \label{conp}
	\begin{cases}
		\frac ns\leq p\le\infty, & \text{ if } 2s<n,\\
		2<p\leq \infty, & \text{ if }\ 2s=n,\\
		2\leq p\leq \infty, & \text{ if } 2s> n.
	\end{cases}
    \end{equation}	
	Furthermore, we assume that the nonlinearity $g\colon \Omega_T\times \mathbb R\to \mathbb{R}$ satisfies:
    \begin{enumerate}[({A.}1)]
        \item\label{A1} There exists $M\in\N_{\geq 2}$ such that $g$ together with all partial derivatives $\partial^{\ell}_\tau g$, $1\leq \ell\leq M$, are Carath\'eodory functions.
        \item\label{A2} The partial derivative $\partial_\tau g$ satisfies the growth condition\footnote{We write the symbol $\lesssim$ to signify that the inequality holds up to a positive constant.}
        \begin{equation}
        \label{eq: growth first der}
            |\partial_{\tau} g(x,t,\tau)|\lesssim |\tau|^\gamma+|\tau|^r
        \end{equation} 
        for a.e.~$(x,t)\in\Omega_T$ and all $\tau\in\mathbb R$, where the exponents fulfill the conditions
    \begin{equation}
        \label{conr}
	0<\gamma\leq r\text{ and }\begin{cases}
		0< r\leq \frac{2s}{n-2s}, & \text{ when}\ 2s<n,\\
		0<r<\infty, & \text{ when } 2s\geq  n.
	\end{cases}
    \end{equation}
    \item\label{A3} In the range $2s\geq n$, we suppose that the partial derivatives $(\partial_{\tau}^{\ell}g)_{2\leq \ell\leq M}$ satisfy
          \begin{equation}
        \label{condg super}
            |\partial_{\tau}^{\ell} g(x,t,\tau)|\lesssim 1+|\tau|^k
        \end{equation}
        for a.e.~$(x,t)\in\Omega_T$, all $\tau\in\mathbb R$ and some $k\geq 0$. In the range $2s<n$, we replace the constraint \eqref{condg super} by the requirement that
          \begin{equation}
        \label{condg}
            |\partial_{\tau}^{\ell} g(x,t,\tau)|\lesssim \begin{cases}
                1+|\tau|^k,\text{ when }2\leq \ell\leq m-1,\\
                1,\text{ when }\ell=m
            \end{cases}
        \end{equation}
        for a.e.~$(x,t)\in\Omega_T$, all $\tau\in\mathbb R$ and some parameters $2\leq m\leq M$, $k\geq 0$ satisfying the conditions 
        \[  
            m\leq \frac{n}{n-2s},\text{ or equivalently }1-\frac{1}{m}\leq\frac{2s}{n},
        \]
        and
        \[
            k\leq \min\left(\frac{n}{(n-2s)m},\frac{2s}{n-2s}\right).
        \]
    \end{enumerate}
Let us emphasize that the conditions \ref{A1}–\ref{A2} are required for the well-posedness theory of the semilinear nonlocal MGT equation \eqref{sysmgt}, whereas \ref{A3} plays a crucial role in the uniqueness proof for the inverse problem \ref{IP1}.

To illustrate the scope of our results, we now present several example nonlinearities $g$ that satisfy assumptions \ref{A1}–\ref{A3}. One of our main results, Theorem~\ref{thmIP1}, ensures that each of these nonlinearities can be uniquely recovered from the DN map $\Lambda_{0,g}$, or more generally from $\Lambda_{q,g}$ for sufficiently regular potentials $q$. For clarity, we formulate our examples in two space dimensions, although they extend straightforwardly to higher dimensions.
	\begin{example}[Polynomial-type nonlinearities]\label{exl1}
		Let $n=2$ and suppose that $s=\frac{\mathfrak{m}-1}{\mathfrak{m}}\in (0,1)$ for some $\mathfrak{m}\in\N_{\geq 2}$. Then, we have
        \begin{equation}
        \label{eq: endpoint rational exponents}
           \frac{n}{n-2s}=\frac{1}{1-s}=\mathfrak{m}\text{ and }\frac{2s}{n-2s}=\frac{s}{1-s}=\mathfrak{m}-1.
        \end{equation}
     \begin{enumerate}[{(E.1.}1)]
         \item\label{monomials} We may check that monomials of the form
        \begin{equation}
            g(x,t,\tau)=\alpha(x,t)\tau^{R+1},
        \end{equation}
        where $1\leq R\leq \min(\mathfrak{m-1},\sqrt{\mathfrak{m}+1})$ is an integer and $\alpha\in L^{\infty}(\Omega_T)$. This corresponds to the case $1\leq \gamma=r=R\leq \mathfrak{m}-1=\frac{2s}{n-2s}$, $2\leq m=R+1\leq \mathfrak{m}=\frac{n}{n-2s}$ and $0\leq k=R-1\leq \min(\mathfrak{m}/(R+1),\mathfrak{m}-1)$. Note that $R\leq \sqrt{\mathfrak{m}+1}$ guarantees $R-1\leq \mathfrak{m}/(R+1)$.
        \item\label{polynomials} If $\beta\vcentcolon =\mathfrak{m}/\mathfrak{r}\in\N$ for some integer $1\leq \mathfrak{r}\leq \mathfrak{m}-1$, then we may choose all polynomials of the form
		\begin{equation}
		g(x,t,\tau)=\sum_{j=1}^{\beta}\alpha_{j}(x,t)\tau^{1+ j R}, 
		\end{equation}
         where $1\leq R\leq \min(\mathfrak{m}-1,\sqrt{\mathfrak{m}+1})/\beta$ is an integer and $(\alpha_j)_{1\leq j\leq \mathfrak{m}}\subset L^\infty(\Omega_T)$. This corresponds to the case $1\leq \gamma=R$, $r=\beta R\leq \mathfrak{m}-1=\frac{2s}{n-2s}$, $1+\beta\leq m=1+\beta R\leq \mathfrak{m}=\frac{n}{n-2s}$ and $0\leq k=\beta R-1\leq \min(\mathfrak{m}/(\beta R+1),\mathfrak{m}-1)$. Condition $k\leq \mathfrak{m}/(\beta R+1)$ follows again from $R\leq \sqrt{\mathfrak{m}+1}/\beta$.
         \item\label{Gaussian and super Gaussian} We can also multiply the examples in \ref{monomials}--\ref{polynomials} by a Gaussian function $e^{-\tau^2}$ or a super-Gaussian function $e^{-\tau^{2 \mathcal{K}}}$, where $\mathcal{K}\geq 2$ is an integer. The latter appear in the modeling of high-intensity light beams \cite{parent1992propagation} and the exponent $\mathcal{K}$ determines how fast the intensity decays away from the center of the beam.
     \end{enumerate}
	\end{example}

%\textcolor{blue}{Henceforth, we use the notation $\Lambda_{q,g}^{X\to Y}$ (resp. $\Lambda_\kappa^{X\to Y}$ and $\Lambda_\beta^{X\to Y}$) to denote the DN map $\Lambda_{q,g}$ (resp. $\Lambda_\kappa$ and $\Lambda_\beta$) from domain $C_c^\infty(X_T)$ to codomain $L^2(0,T;H^{-s}(Y))$, where $X,Y\subset\Omega_e$ are both nonempty open sets. For instance, the DN map $\Lambda_{q,g}^{W_1\to W_2}$ acting on a smooth function $\varphi$ supported on $(W_1)_T$  is given by
%\begin{equation}
%\Lambda_{q,g}^{W_1\to W_2}\varphi\vcentcolon=(b(-\Delta)^s\partial_tu+c(-\Delta)^su)|_{(W_2)_T},
%\end{equation}
%where $u$ is the solution to the semilinear nonlocal MGT equation \eqref{sysmgt} with the exterior Dirichlet data $\varphi\in C_c^\infty(W_1)_T)$. Particularly, when $X=Y=\Omega_e$, we drop the superscript $\Omega_e\to\Omega_e$ appearing in the map $\Lambda_{q,g}^{\Omega_e\to\Omega_e}$ (resp. $\Lambda_\kappa^{\Omega_e\to \Omega_e}$ and $\Lambda_\beta^{\Omega_e\to \Omega_e}$)   and still write it by $\Lambda_{q,g}$ (resp. $\Lambda_\kappa$ and $\Lambda_\beta$) for simplicity.
%}

Next, we present our findings concerning question \ref{IP1}. In Section~\ref{sec: proof of theorem 1.5}, we show that the method of higher-order linearization allows one to determine the potential and the derivatives of the polynomial-type nonlinearity $g$ at $\tau=0$ up to order $m$. More specifically, we prove the following result.

    \begin{theorem}[Unique determination of $q$ and $\{\partial_\tau^{\ell}g(0)\}_{\ell=1}^m$]
    \label{thmIP1}
		Let $\Om\subset\R^n$ be a bounded Lipschitz domain, $T>0$ and $s\in \R_+\setminus\N$. 
        Suppose that the potentials $q_1,q_2\in L^{\infty}(0,T;L^p(\Omega))$ are real valued, $p$ satisfies the condition \eqref{conp} and $q_2$ is time-reversal invariant. Furthermore, let the nonlinearities $g_1,g_2$ obey the assumptions \ref{A1}--\ref{A3} with the same constants $M,\gamma,r,m$ and $k$. If $W_1,W_2\subset\Omega_e$ are two nonempty open sets such that
        \begin{equation}
        \label{eq: equality DN map main thm}
   \left.\Lambda_{q_1,g_1}\varphi\right|_{(W_2)_T}=
   \left.\Lambda_{q_2,g_2}\varphi\right|_{(W_2)_T}
        \end{equation}
		for all sufficiently small exterior values $\varphi\in C_c^{\infty}((W_1)_T)$, then we have
        \begin{equation}
        \label{eq: conclusion main thm 1}
            q_1=q_2\text{ and }\partial_{\tau}^{\ell}g_1(0)=\partial_{\tau}^{\ell}g_2(0)\text{ for all }1\leq \ell\leq m
        \end{equation}
        a.e. in $\Omega_T$.
	\end{theorem}

Apart from polynomial-type nonlinearities, as illustrated in Example~\ref{exl1}, we also address question \ref{IP1} for polyhomogeneous nonlinearities of finite order $L\in \N_{\geq 1}$. That is, we consider nonlinearities $g$ satisfying the following structural assumption:
\begin{enumerate}[({A.}4)]
    \item\label{A4} The nonlinearity $g$ has the form
\begin{equation}
    \label{power}
g(x,t,\tau)=\sum_{k=1}^{L}\alpha_k(x,t)|\tau|^{r_k}\tau,
\end{equation}
where $\alpha_1,\ldots,\alpha_L\in L^{\infty}(\Omega_T)$, $0<r_1<\ldots<r_L$ and $0<r_L\leq 1$ satisfy the growth condition \eqref{conr}. 
\end{enumerate} 

In Section~\ref{sec: polyhom}, we employ the method of first-order linearization to obtain the following uniqueness result for polyhomogeneous nonlinearities $g$ of finite order $L$.

    \begin{theorem}[Unique determination of $q$ and $g$ (or $\{\alpha_k\}_{k=1}^L$)]
    \label{thm: polyhom nonlinearities}
        Let $\Om\subset\R^n$ be a bounded Lipschitz domain, $T>0$, $s\in \R_+\setminus\N$ and assume that $W_1,W_2\subset\Omega_e$ are two nonempty open sets. Furthermore, 
        Suppose that the potentials $q_1,q_2\in L^{\infty}(0,T;L^p(\Omega))$ satisfy the assumption of Theorem  \ref{thmIP1} and we have given two polyhomogeneous nonlinearities $g_1,g_2$ of order $L$, as defined in \ref{A4}. If the related DN maps $\Lambda_{q_j,g_j}$, $j=1,2$, obey
         \begin{equation}
        \label{eq: equal DN maps poly}
\left.\Lambda_{q_1,g_1}\right|_{(W_2)_T}\varphi=\left.\Lambda_{q_2,g_2}\varphi\right|_{(W_2)_T}
        \end{equation}
		for all sufficiently small exterior values $\varphi\in C_c^{\infty}((W_1)_T)$,
        then it holds
        \begin{equation}
          q_1=q_2\text{ and }   g_1=g_2\text{ for a.e. }(x,t)\in\Omega_T\text{ and all }\tau\in\R.
        \end{equation}
    \end{theorem}
\begin{remark}
    We note that the time-reversal invariance assumption on the potential $q_2$ in Theorems~\ref{thmIP1} and~\ref{thm: polyhom nonlinearities} is imposed in order to apply the integral identity \eqref{eq: integral identity potential} (see Lemma~\ref{integral}) for the linear nonlocal MGT equation, which is a key step in establishing the unique determination of the potential. However, this assumption is in fact not essential and can be removed by a straightforward modification of the argument in \cite[Section~4.2]{PZ1}.
\end{remark}

The next result resolves the question \ref{IP2}.
\begin{theorem}[Unique determination of $\kappa$ or $\beta$]
\label{thmIP2}
Let $\Om\subset\R^n$ be a bounded Lipschitz domain, $T>0$, and 
suppose that $s\in \R_+\setminus\N$ satisfies $s>n/2$, and  $(\beta_j,\kappa_j)\in W^{2,\infty}(0,T;L^\infty(\Omega))\times W^{1,\infty}(0,T;L^\infty(\Omega))$ for $j=1,2$. Let $W_1,W_2\subset\Omega_e$ be two nonempty open sets.
\begin{enumerate}[(i)]
\item\label{Westervelt} If there holds
\begin{equation}
\label{eq: equal DN Westervelt 1}
\left.\Lambda_{\beta_1}\varphi\right|_{(W_2)_T}=
\left.\Lambda_{\beta_2}\varphi\right|_{(W_2)_T}
\end{equation}
for all sufficiently small exterior values $\varphi\in C_c^{\infty}((W_1)_T)$, then we have $\beta_1=\beta_2$ in $\Omega_T$.
\item\label{Westervelt 2} If there holds
\begin{equation}
\label{eq: equal DN Westervelt 2}
\left.\Lambda_{\kappa_1}\varphi\right|_{(W_2)_T}=\left.\Lambda_{\kappa_2}\varphi\right|_{(W_2)_T}
\end{equation}
for all sufficiently small exterior values $\varphi\in C_c^{\infty}((W_1)_T)$, then we have $\kappa_1=\kappa_2$ in $\Omega_T$.
\end{enumerate} 
\end{theorem}

\begin{remark}
    We note that the smallness condition imposed on the Dirichlet data $\varphi\in C_c^{\infty}((W_1)_T)$ in Theorems~\ref{thmIP1}, \ref{thm: polyhom nonlinearities}, and \ref{thmIP2} is dictated by the well-posedness theory for the corresponding forward problems \eqref{sysmgt} and \eqref{sysmgtWtype}. This issue is discussed in more detail in Section~\ref{sec: well-posedness forward}; see in particular Remarks~\ref{rem: dependence of energy constant}, \ref{rem: smallness in semilinear} and \ref{rem: smallness in westervelt}. Roughly speaking, the smallness of the exterior data is required in order to carry out the fixed-point arguments.
\end{remark}

For the reader’s convenience, we summarize in Table~\ref{tableII} the inverse problems considered in this article and, in particular, highlight the parameters recovered in Theorems~1.3, 1.4, and~1.6.

{\small
\begin{table}[h]
\centering
\small
\caption{Nonlocal equations and associated inverse problems}
\label{tableII}
\begin{tabular}{c @{\quad} c  @{\quad} c}
\toprule
\textbf{Nonlocal equation} 
& \textbf{Recovered parameters} 
& \textbf{Measurement data} \\
\midrule
\makecell[c]{$(L_{\alpha,b,c}+q)u + g(u)=0$,\\ $g(x,t,\tau)$ satisfying \ref{A1}--\ref{A3}}
& $q,\{\partial_\tau^{\ell} g(0)\}_{\ell=1,\ldots,m}$
& $\Lambda_{q,g}\varphi \big|_{(W_2)_T}$ \\
\midrule
 \makecell[c]{$(L_{\alpha,b,c}+q)u + g(u)=0,$\\
$g(x,t,\tau)=\sum_{k=1}^L \alpha_k(x,t) |\tau|^{r_k}\tau$}
& $q,\, \{\alpha_k\}_{k=1,\ldots,L}$
& $\Lambda_{q,g}\varphi \big|_{(W_2)_T}$ \\
\midrule
$L_{\alpha,b,c}u = \partial_t\!\big(\kappa(\partial_t u)^2\big)$
& $\kappa$
& $\Lambda_{\kappa}\varphi \big|_{(W_2)_T}$ \\
\midrule
$L_{\alpha,b,c}u = \partial_t^2\!\big(\beta u^2\big)$
& $\beta$
& $\Lambda_{\beta}\varphi \big|_{(W_2)_T}$ \\
\bottomrule
\end{tabular}\\
\vspace{0.3cm}
\parbox{0.9\linewidth}{\scriptsize
\textit{Legend.} $L_{\alpha,b,c}$ is the third order nonlocal wave operator defined in equation \eqref{eq: third order nonlocal wave operator}; $W_1,W_2\subset\Omega_e$ are measurement sets;
$\varphi \in C_c^\infty((W_1)_T)$ denotes the exterior Dirichlet datum supported in $(W_1)_T$;
$\Lambda_{q,g},\Lambda_{\kappa}$, and $\Lambda_{\beta}$ denote the DN maps.}
\end{table}}

Finally, we point out that the methods employed in this article to solve the inverse problems for the nonlocal (J)MGT equations, summarized in Table~\ref{tableII}, differ substantially from those used for inverse problems associated with local PDEs; in particular, from those developed for \emph{local} (J)MGT equations, which were studied, for example, by the first two authors in \cite{SRPFYY,fu2026partial}.

Indeed, our arguments rely heavily on the \emph{unique continuation principle (UCP)} for the fractional Laplacian $(-\Delta)^s$, a property that holds only for $s\in\R_+\setminus\N$. The UCP for the fractional Laplacian can be phrased as follows (see, for instance, \cite{TGMS,KRZ}): 
\vspace{0.2cm}
\begin{center}\emph{If $V\subset\R^n$ is a nonempty open set and $u\in H^r(\mathbb R^n)$, $r\in\mathbb R$, satisfies\\ $u=(-\Delta)^su=0$ in $V$, then $u = 0$ in $\mathbb R^n$.}
\end{center}
\vspace{0.2cm}
Let us emphasize that this property cannot hold for integer orders $s\in\N$ and stands in sharp contrast to the finite speed of propagation exhibited by the (J)MGT equation, as discussed in Section~\ref{subsec: background acoustic waves}.

In contrast to the nonlocal inverse problems treated in the present work, the resolution of inverse problems associated with \emph{local} (J)MGT equations relies on the construction of geometric optics solutions and on a reduction to an injective geometric inverse problem; see \cite{SRPFYY,fu2026partial}. We also refer the interested reader to \cite{kian2019recovery,belishev1987approach,kumar2025holderstabilityestimatesdetermination,bellassoued2022stable,buchgeim1981uniqueness,rakesh1988uniqueness,isakov1991inverse,eskin2006new,eskin2007new,cristofol2015determining,kian2019determination,kurylev2018inverse,lassas2017determination,lassas2018inverse,lassas2022uniqueness} and the references therein, which address inverse problems associated with other hyperbolic equations. These works present additional classical techniques in the theory of inverse problems for hyperbolic equations, including the boundary control (BC) method and the use of Carleman estimates.

\subsection{Comparison of our main results with the literature}
\label{sec: related literature}

In the following sections, we first provide a brief overview of inverse problems associated with elliptic and parabolic equations, then discuss inverse problems for nonlocal wave equations. Finally, we compare our main results with the existing literature.

\subsubsection{Inverse problems for nonlocal elliptic and parabolic equations}
\label{subsec: elliptic nonlocal IP}

In the influential article \cite{TGMS}, the authors studied the so-called Calder\'on problem for the \emph{fractional Schrödinger equation}
\begin{equation}
\label{eq: frac schroed}
    ((-\Delta)^s + q)u = 0 \quad \text{in } \Omega,
\end{equation}
where $\Omega \subset \mathbb{R}^n$ is a bounded Lipschitz domain, $0 < s < 1$, and $q$ is a bounded potential. In particular, they established the \emph{unique continuation property (UCP)} of $(-\Delta)^s$ for functions in $H^r(\mathbb{R}^n)$, $r\in\mathbb{R}$, and, by a Hahn--Banach argument, the $L^2(\Omega)$ Runge approximation property for solutions of \eqref{eq: frac schroed}. Combining this approximation with an Alessandrini-type identity, the authors proved that the potential $q$ can be uniquely recovered from the DN map $\Lambda_q$. Later, a $\widetilde{H}^s(\Omega)$ Runge approximation was obtained in \cite{ARMS}, allowing the recovery of certain singular potentials (Sobolev multipliers). These uniqueness and approximation results were generalized in \cite{JRPZ} to unbounded domains, while \cite{KRZ} studied the UCP in the more general scale of Bessel potential spaces $H^{t,p}(\mathbb{R}^n)$.

Moreover,
\cite[Proposition~5.1]{TRSU} establishes a measurable unique continuation property (MUCP) for the fractional Schrödinger equation \eqref{eq: frac schroed} when $q \in L^\infty(\Omega)$ and $1/4 \le s < 1$: if $u \in H^s(\mathbb{R}^n)$ solves \eqref{eq: frac schroed} and vanishes on a set of positive measure, then $u$ must vanish identically. Furthermore, according to \cite[Remark~5.6]{TRSU}, this MUCP extends to the full range $0 < s < 1$ provided that $q \in C^1(\overline{\Omega})$ (see also \cite{fall2014unique}). In contrast to the UCP, which depends only on $(-\Delta)^s$, the MUCP fundamentally relies on the fractional Schrödinger equation.

In the regime $1/4 \le s < 1$, the MUCP yields a single-measurement uniqueness result for the inverse problem. Indeed, since MUCP ensures that the solution $u$ of \eqref{eq: frac schroed} with nonzero exterior data satisfies $u \neq 0$ a.e.\ in $\Omega$, we may write
\begin{equation}
\label{eq: reconstruction formula schroedinger}
    q(x) = -\frac{(-\Delta)^s u(x)}{u(x)} 
\end{equation}
for a.e.\ $x \in \Omega$. By \cite[Theorem~2]{TRSU} and the linearity of $(-\Delta)^s$, the right-hand side is uniquely determined by $\Lambda_q \varphi$ for a single nonzero exterior datum $\varphi$.

For $0 < s < 1/4$, one can still recover continuous potentials, even though the MUCP is unavailable for continuous potentials in this range. Using the UCP for $(-\Delta)^s$, one shows that for any $x_0 \in \Omega$ there exists a sequence $(x_k) \subset \Omega$ with $x_k \to x_0$ and $u(x_k) \neq 0$ for all $k$. By continuity of $q$, this yields
\begin{equation}
\label{eq: limit reconstruction formula schroedinger}
    q(x_0)
    = \lim_{k\to\infty} q(x_k)
    = -\lim_{k\to\infty} \frac{(-\Delta)^s u(x_k)}{u(x_k)}.
\end{equation}
Again, by \cite[Theorem~2]{TRSU}, the right-hand side is determined by $\Lambda_q \varphi$ for a single nonzero exterior datum $\varphi$. Consequently, for $0 < s < 1/4$, any continuous potential $q \in C(\overline{\Omega})$ can still be uniquely recovered from a single measurement.

MUCP results also extend beyond the fractional Schrödinger operator. In \cite[Proposition~7.9]{KRZ}, the last author together with M. Kar and J. Railo proved a MUCP for the fractional $p$-biharmonic operator
\[
    (-\Delta)^s_p u \vcentcolon= (-\Delta)^{s/2}\big(|(-\Delta)^{s/2} u|^{p-2} (-\Delta)^{s/2} u\big),
\]
showing that, under suitable assumptions on $(s,p)$, if a sufficiently regular function $u$ satisfies
\[
    (-\Delta)^s_p u = 0 \quad \text{in } \Omega,
    \qquad
    (-\Delta)^{s/2}u = 0 \quad \text{in } \Omega',
\]
where $\Omega' \subset \Omega$ has positive measure, then $u\equiv 0$. This yields a MUCP for a highly nonlinear and nonlocal operator.

As in the fractional Schrödinger case, this MUCP implies a single-measurement uniqueness result for the corresponding inverse problem. Owing to the nonlinear structure of $(-\Delta)^s_p$, the argument requires a natural \emph{monotonicity assumption} on the coefficients, which is standard in nonlinear inverse problems. Under this assumption, uniformly elliptic bounded coefficients can be uniquely determined without any continuity requirements, whereas continuity is needed when uniqueness is obtained instead via the UCP for $(-\Delta)^s_p$ (see \cite[Theorems~2.4--2.5]{KRZ}).

Taken together, these works demonstrate that the MUCP and the UCP are powerful tools for establishing single-measurement uniqueness results in inverse problems for both linear and nonlinear nonlocal PDEs. This motivates the broader program of extending the MUCP to other, possibly nonlinear, nonlocal equations--such as semilinear fractional Schrödinger equations--where such results remain largely open.
We also refer the interested reader to \cite{JRPZ}, which analyzes the role of UCP in nonlocal inverse problems and, in particular, removes the boundedness assumption on the domain $\Omega$ in the fractional Calderón problem studied in \cite{TGMS} and related works.

In Section~\ref{sec: concluding remarks}, we discuss the problem of uniquely determining potentials and nonlinearities from fewer measurements than those required in our main theorems (Theorems~\ref{thmIP1}, \ref{thm: polyhom nonlinearities}, and \ref{thmIP2}).

For further studies on inverse problems for linear nonlocal PDEs, we refer the interested reader to 
\cite{MYLA,CKJG,CRTZ,10.4310/jdg/1757353909,TRSU,LiLi,YHGNPZ,LRZ-nonlocal-diffusion,ARMS,JRPZ} 
and the references therein, where the underlying PDEs are of elliptic or parabolic type.
In addition, the determination of nonlinearities in nonlocal PDEs has recently become a very active area of research. 
For instance, the works \cite{LaLin,KMSK,PZJN} derived unique determination results for semilinearities in elliptic and parabolic nonlocal PDEs.

In general, the determination of nonlinear terms in PDEs dates back at least to the works \cite{VI,VI1} of Isakov, where inverse problems for nonlinear elliptic and parabolic equations were studied using first- and higher-order linearization techniques. Notably, \cite{VI1} appears to be the first work to employ a higher-order linearization method, which has since become a powerful tool in the analysis of inverse problems for nonlinear PDEs. Over time, these techniques have been successfully applied to a variety of models, including elliptic nonlocal equations.

\subsubsection{Inverse problems for nonlocal wave equations}
\label{subsec: hyperbolic nonlocal IP}
In the recent work \cite{LTZ1}, the last author together with Y.-H.~Lin and T.~Tyni established a Runge approximation in 
$L^2(0,T;\widetilde H^s(\Omega))$ for solutions of linear nonlocal wave equations
\[
    (\partial_t^2 + (-\Delta)^s+q)u = 0 \quad \text{in } \Omega_T,
\]
thereby extending the uniqueness results for Calder\'on problems of linear and nonlinear nonlocal wave equations obtained in their earlier works \cite{KLW,LTZ2}. 
This result provides an affirmative answer to a question raised in \cite{PZ1}, where the Calder\'on problem for viscous nonlocal wave equations
\[
    (\partial_t^2 + (-\Delta)^s \partial_t + (-\Delta)^s)u + f(u) = 0 \quad \text{in } \Omega_T
\]
with linear and nonlinear perturbations was investigated. A key insight of \cite{LTZ1} is the extension of the theory of very weak solutions for classical wave equations to the nonlocal setting. 
Using this approach, the authors recovered two classes of polyhomogeneous nonlinearities from the associated DN map.

Furthermore, the work \cite{PZ2} generalizes the Runge approximation of \cite{LTZ1} to damped nonlocal wave equations
\[
    (\partial_t^2 + \gamma \partial_t + (-\Delta)^s+q)u = 0 \quad \text{in } \Omega_T,
\]
and establishes uniqueness for the simultaneous recovery of H\"older continuous damping coefficients $\gamma$ together with either time-dependent potentials $q$ or homogeneous nonlinearities $f$.

We also mention the work \cite{LLZY}, which studies the recovery of a Westervelt-type nonlinearity in a fractionally damped wave equation from the \emph{source-to-solution map}. More precisely, they consider the PDE
\begin{equation}
\label{eq: fractionally damped}
    \partial_t^2(u - \vartheta u^2) + (-\Delta_{\Omega})^s \partial_t u - \Delta_{\Omega} u = F,
\end{equation}
where $\Delta_{\Omega}$ denotes the Dirichlet Laplacian on $\Omega$ and $(-\Delta_{\Omega})^s$ is its spectral fractional power. In addition, the article \cite{KMSK} studies the recovery of semilinearities in one-dimensional nonlocal wave equations (and in nonlocal parabolic equations) from DN data using the higher-order linearization method.

\subsubsection{Discussion of main results}

In this section, we compare the main results obtained in Theorems~\ref{thmIP1}, \ref{thm: polyhom nonlinearities}, and \ref{thmIP2} with the existing literature. 

\medskip
\noindent
\textit{Determination of polynomial-type nonlinearities (Theorem~{\rm \ref{thmIP1}}):}
All results in this manuscript concern inverse problems without restrictions on the number of measurements. For example, in Theorem~\ref{thmIP1}, we assume the knowledge of $\Lambda_{q,g}\varphi\big|_{(W_2)_T}$
for all (sufficiently small) exterior conditions $\varphi \in C^\infty_c((W_1)_T)$.  
In analogy with the single–measurement results for elliptic nonlocal PDEs discussed in Section~\ref{subsec: elliptic nonlocal IP} (see \cite{TRSU,KRZ}), the works \cite{KMSK,PZJN,li2023inverse} investigate the unique recovery of essentially polyhomogeneous nonlinearities from \emph{finite-dimensional} measurement data.  
For instance, \cite[Theorem~1.4]{KMSK} proves that it suffices to measure
\begin{equation}
\label{eq: finite dim meas}
    \bigl\{ \Lambda^{1D}_{0,g}(\varepsilon_1 \varphi_1 + \cdots + \varepsilon_m \varphi_m )\big|_{(W_2)_T} 
    \,;\, \varepsilon_1,\dots,\varepsilon_m \in \mathbb{R} \text{ close to } 0 \bigr\},
\end{equation}
where $\Lambda^{1D}_{0,g}$ is the DN map for the one-dimensional, supercritical nonlocal wave equation with time-independent semilinearity $g$, and $\varphi_1,\ldots,\varphi_m \in C_c^\infty((W_1)_T)$ are fixed nontrivial exterior conditions.  
From these finite-dimensional measurements, one can uniquely determine partial derivatives of $g$ up to order $m$ at $\tau = 0$.

To clarify the differences with our assumptions, we list the essential properties imposed in \cite{KMSK} on the Carathéodory function $g \colon \Omega_T \times \mathbb{R} \to \mathbb{R}$:
\begin{enumerate}[{(P}1)]
    \item\label{P1} For each $(x,t)\in \Omega_T$, the map $\tau\mapsto g(x,t,\tau)$ belongs to $C^{m+1}((-\delta,\delta))$ for some $\delta>0$ and $m \in \mathbb{N}$.
    \item\label{P2} $g(x,t,0)=0$ for all $(x,t)\in\Omega_T$.
    \item\label{P3} For $k=1,\ldots,m+1$, define  
    \[
        \Phi_k(\varepsilon) \vcentcolon = 
        \sup_{|\tau|\le \varepsilon}\sup_{(x,t)\in\Omega_T} 
        |\partial_\tau^k g(x,t,\tau)|, \qquad 0<\varepsilon\le \delta.
    \]
    Then it is assumed that
    \begin{equation}
    \label{eq: vanishing at origin}
        \Phi_1(\varepsilon) \to 0 \quad \text{as }\varepsilon \to 0,
    \end{equation}
    and for every $2 \le k \le m+1$ there exists $M_k>0$ such that $\Phi_k(\varepsilon) \le M_k$ for all $0 < \varepsilon \le \delta$.
    \item\label{P4} For $k=1,\dots,m$, the functions $\partial_\tau^k g(x,0)$ lie in $C(\overline{\Omega})$ and are independent of $t$.
\end{enumerate}
Moreover, \cite[Theorem~1.3]{KMSK} shows that the same information about the time-dependent nonlinearity $g(x,t,\tau)$ can be recovered under assumptions \ref{P1}–\ref{P3} \emph{when} no restriction on the number of measurements, analogous to Theorem~\ref{thmIP1}, are imposed.

Note that choosing $\varphi_1=\cdots=\varphi_m$ in \eqref{eq: finite dim meas} yields that $\{\partial_\tau^k g(x,0)\}_{k=0}^m$ can be determined from a single one-dimensional measurement under the additional condition \ref{P4}.  
In particular, if $g(x,\tau)$ is analytic in $\tau$, then $g$ can be fully recovered from one-dimensional measurements.

We now highlight several differences between \cite{KMSK} and Theorem~\ref{thmIP1}:
\begin{enumerate}[{(D}1)]
    \item\label{D1} \textit{Criticality.}  
    Our results in Theorems~\ref{thmIP1} and~\ref{thm: polyhom nonlinearities} impose no restriction on the space dimension $n$ or the fractional order $s \in \mathbb{R}_+ \setminus \mathbb{N}$.  
    In contrast, \cite{KMSK} focuses exclusively on the case $s\in (0,1)$ and the supercritical 1D regime, where the Sobolev embedding $H^s(\mathbb{R}) \hookrightarrow C^{0,\alpha}(\mathbb{R})$ plays a crucial role, allowing $g$ to be controlled only near the origin.
    \item\label{D2} \textit{Linear perturbations.}  
    Condition \eqref{eq: vanishing at origin} implies that $g$ contains no linear term in $\tau$.  
    Our Theorems~\ref{thmIP1} and~\ref{thm: polyhom nonlinearities} allow for the presence of linear perturbations.
    \item\label{D3} \textit{Time dependence and low regularity.}  
    A major focus of the present work is the recovery of \emph{time-dependent, low-regularity} potentials and nonlinearities.  
    For example, we assume only $q \in L^\infty(0,T;L^p(\Omega))$, with $p$ satisfying \eqref{conp}, and  
    \[
        \partial_\tau^\ell g(x,t,\tau) \in L^\infty(\Omega_T), \qquad 1\le \ell \le m,\,\tau\in\R,
    \]
    together with appropriate growth conditions in $\tau$ (see \ref{A3}). Additionally, in \cite{KMSK} the authors assume that $\tau\mapsto g(x,\tau)$ is $C^{m+1}$ regular in a neighborhood of $\tau=0$ to recover $\{\partial_{\tau}^k g(0)\}_{0\leq k\leq m}$, whereas we only require $\tau\mapsto g(x,t,\tau)$ to be of class $C^m$ (see \ref{P1} and \ref{A1}).
\end{enumerate}
In view of \ref{D3}, we omit the detailed results and proofs for the finite-dimensional measurement problems from the main body of this article; the interested reader can find a comprehensive discussion in Section~\ref{IP-one dimensional measurement}.

\medskip
\noindent
\textit{Determination of polyhomogeneous nonlinearities (Theorem~{\rm\ref{thm: polyhom nonlinearities}}):}
Assumption \ref{A4}, and hence Theorem~\ref{thm: polyhom nonlinearities}, requires the restriction $r_L \le 1$.  
This condition ensures that the proof relies only on the $L^2(\Omega_T)$ Runge approximation property (Proposition~\ref{Runge}).  
The same restriction appears in \cite{LTZ2}, but it was recently removed in \cite{LTZ1} by establishing Runge approximation in the smaller space $L^2(0,T;\widetilde H^s(\Omega))$.  
For brevity, we confine ourselves to the $L^2(\Omega_T)$ result here; the improved Runge approximation will be developed in forthcoming work.

Another difference from \cite[Section~5.2]{LTZ1} is that our polyhomogeneous nonlinearities involve only a \emph{finite} number of homogeneous terms.  
This allows us to iteratively reduce the order of the nonlinearity, thereby reducing the $L$ unknown coefficients $\alpha_1,\dots,\alpha_L$ to a single unknown $\alpha_L$.  
The coefficient $\alpha_L$ is then recovered using the Runge approximation theorem, which has been carried out in \cite[Section~4.2]{LTZ2} for second-order nonlocal wave equation.  
We remark that Theorem~\ref{thm: polyhom nonlinearities} remains valid for polyhomogeneous nonlinearities of infinite order, as follows from the approach in \cite[Section~5.2]{LTZ1}. Interestingly, this approach relies solely on the UCP for the fractional Laplacian, an asymptotic analysis of solutions for small exterior data $\eps\varphi$, and the Runge approximation theorem.

\medskip
\noindent
\textit{Determination of Westervelt-type nonlinearities (Theorem~{\rm\ref{thmIP2}}):}
As noted in Section~\ref{subsec: hyperbolic nonlocal IP}, the work \cite{LLZY} investigates the recovery of Westervelt-type nonlinearities in a fractionally damped wave equation, using second-order linearization and the UCP for the spectral fractional Laplacian to recover smooth coefficients $\vartheta \in C^\infty(\overline{\Omega_T})$.  
Our corresponding result (statement \ref{Westervelt} in Theorem~\ref{thmIP2}) requires only that $\beta$ belongs to $W^{2,\infty}(0,T;L^\infty(\Omega))$. As in the case of polyhomogeneous nonlinearities, our proof does not involve explicit differentiation of the PDE with respect to a small parameter; instead, we rely on an asymptotic expansion for small exterior conditions $\varepsilon\varphi$.

In conclusion, our work builds upon the linearization approach developed in earlier works on nonlocal wave equations (see, e.g., \cite{PZJN,KMSK,PZ1,LTZ1,LTZ2,PZ2}). While these techniques are classical (see, e.g., \cite{VI1,VI}), our contribution is to extend them to the third-order, time-dependent MGT and JMGT models studied here, together with low regularity coefficients. In particular, we emphasize that our work does not contain any restriction on the dimension $n$ or the fractional exponent $s$. Moreover, such models have not previously been treated in Calderón-type inverse problems.

\subsection{Organization of the article}
The rest of this article is organized as follows. In Section \ref{sec: preliminaries}, we recall several properties of the fractional Laplacian and introduce some fundamental functions spaces. In Section \ref{sec: well-posedness forward}, we establish well-posedness results for the linear and nonlinear nonlocal MGT equations. Section \ref{sec: recovery of lin perturb} discusses the unique determination of linear perturbations. The unique recovery of polynomial-type nonlinearities, polyhomogeneous nonlinearities and Westervelt-type nonlinearities is discussed in Sections~\ref{sec: proof of theorem 1.5}, \ref{sec: polyhom} and \ref{sec: recovery beta and kappa}, respectively. Finally, in Section \ref{sec: concluding remarks} we make some concluding remarks on the results obtained in this article and in Appendix \ref{appendix: parabolic regularization} we present the method of parabolic regularization as well as an integration by parts formula for solutions to third order wave equations.

\subsection*{Notation} 
Throughout this work, $\Omega \subset \mathbb{R}^n$ denotes a bounded Lipschitz domain and 
$\Omega_e := \mathbb{R}^n \setminus \overline{\Omega}$ its exterior. 
We let $W_1, W_2 \subset \Omega_e$ be nonempty open measurement sets. 
For any $T>0$ and any set $A \subset \mathbb{R}^n$, we write $A_T := A \times (0,T)$.
We use the notation $M_1 \lesssim M_2$ to indicate that there exists a constant $C>0$ such that $M_1 \leq C M_2$.

The operators $\mathcal{F}$ and $\mathcal{F}^{-1}$ denote the Fourier transform and its inverse, respectively.
The fractional Laplacian and the Bessel potential operator of order $s$ are defined by
\[
(-\Delta)^s u = \mathcal{F}^{-1} \big( |\xi|^{2s} \mathcal Fu \big),
\qquad
\langle D \rangle^s u = \mathcal{F}^{-1} \big( (1+|\xi|^2)^{s/2} \mathcal Fu \big).
\]
The $L^2$-based fractional Sobolev spaces are given by
\[
H^s(\mathbb{R}^n) = \{ u \in L^2(\mathbb{R}^n) \,;\, \langle D \rangle^s u \in L^2(\mathbb{R}^n) \},
\qquad
\widetilde{H}^s(\Omega) = \overline{C_c^\infty(\Omega)}^{\, H^s(\mathbb{R}^n)}
\]
for any $s\geq 0$. In particular, for notational convenience, we set $\widetilde{L}^2(\Omega)=\widetilde{H}^0(\Omega)$.

Finally, the Dirichlet-to-Neumann maps associated with the semilinear nonlocal MGT equation \eqref{sysmgt} and the
Westervelt-type nonlocal MGT equations \eqref{sysmgtWtype} are denoted by $\Lambda_{q,g}$, $\Lambda_\kappa$, and $\Lambda_\beta$,
respectively.

\section{Preliminaries}\label{sec: preliminaries}

In this section we collect some basic material about the fractional Laplacian and the function spaces relevant for our study.
    
\subsection{Fractional Sobolev spaces}
\label{subsec:frac-space}
Let us start by recalling the definition of the $L^2$ based fractional Sobolev spaces and the fractional Laplacian $(-\Delta)^s$. For any $s\in \R$, we denote by $H^s(\R^n)$ the \emph{fractional Sobolev space} of order $s$, which can be characterized as 
\[
H^s(\R^n)\vcentcolon = \{u\in\tempered(\R^n)\,;\,\langle D\rangle^s u\in L^2(\R^n)\}.
\]
Here, $\tempered(\R^n)$ stands for the space of tempered distributions and $\langle D\rangle^s$ is the Bessel potential operator of order $s$, which is the Fourier multiplier with symbol $\langle \xi\rangle^s=(1+|\xi|^2)^{s/2}$. The scale $H^s(\R^n)$, $s\in\R$, are Hilbert spaces, when we endow them with the natural inner product induced by the norm
\begin{equation}
\label{def: hs norm}
\|u\|_{H^s(\R^n)}\vcentcolon = \|\langle D\rangle^s u\|_{L^2(\R^n)}.
\end{equation}
The \emph{fractional Laplacian} $(-\Delta)^s$, $s\geq 0$, can be defined for $u\in\tempered(\R^n)$ via
\begin{equation}
\label{eq: frac lap}
(-\Delta)^s u=\ifourier(|\xi|^{2s}\fourier u),
\end{equation}
whenever the right hand side makes sense. Clearly, formula \eqref{eq: frac lap} shows that the fractional Laplacian $(-\Delta)^{s/2}$ is the homogeneous counterpart of the Bessel potential operator $\langle D\rangle^s$. By a simple estimate and Plancherel's theorem, we see that an equivalent norm on $H^s(\R^n)$ is given by
\begin{equation}
\label{eq: equivalent norm Hs}
\|u\|^{\ast}_{H^s(\R^n)}=\Big(\|u\|^2_{L^2(\R^n)}+\|(-\Delta)^{s/2}u\|^2_{L^2(\R^n)}\Big)^{1/2}.
\end{equation}

Next, let us introduce some useful local versions of the global fractional Sobolev spaces $H^s(\R^n)$. For any open set $\Omega\subset\R^n$ and $s\in \R$, we introduce the Hilbert spaces
\[
H^s(\Omega) \vcentcolon = \{U|_{\Omega}\,;\, U\in H^s(\R^n)\}.
\]
%\begin{split}
%\widetilde{H}^s(\Omega)&\vcentcolon =\|\cdot\|_{H^s(\R^n)}-\text{clos}(C_c^{\infty}(\Omega)),\\

%\end{split}

As usual, we endow $H^s(\Omega)$ with the corresponding quotient norm. %Throughout this work, we will write $\widetilde{L}^2(\Omega)$ instead of $\widetilde{H}^0(\Omega)$. 
Furthermore, if $\Omega$ is a bounded Lipschitz domain, then the dual space of $\widetilde{H}^s(\Omega)$ can be identified with $H^{-s}(\Omega)$.
    
The most important properties of the fractional Laplacian are collected in the next proposition. 
\begin{proposition}[Properties of fractional Laplacian]
    \label{prop: basic facts on frac Lap}
        Let $0\leq t\leq s$ and $r\in\R$. Then the following assertions hold:
        \begin{enumerate}[(i)]
            \item\label{mapping properties} The fractional Laplacian $(-\Delta)^s$ is a continuous linear map from $H^r(\R^n)$ to $H^{r-2s}(\R^n)$.
            \item\label{Poincare inequality} For any bounded domain $\Omega\subset\R^n$, there exists a constant $C>0$ only depending on $\Omega,n,t$ and $s$ such that 
            \begin{equation}
            \label{eq: poincare inequality}
                \|(-\Delta)^{t/2}u\|_{L^2(\R^n)}\leq C \|(-\Delta)^{s/2}u\|_{L^2(\R^n)}
            \end{equation}
            for all $u\in \widetilde{H}^s(\Omega)$.
            \item\label{UCP} Let $s\notin \N_0$ and suppose that $u\in H^r(\R^n)$ satisfies $u=(-\Delta)^s u=0$ in an open set $V\subset\R^n$, then $u=0$ in $\R^n$.
        \end{enumerate}
    \end{proposition}
    The property \ref{mapping properties} follows from Plancherel's theorem and \eqref{def: hs norm}--\eqref{eq: frac lap}. Furthermore, using Mikhlin's multiplier theorem, the assertion \ref{mapping properties} can be extended to the broader scale of \emph{Bessel potential spaces} $H^{s,p}(\R^n)$ ($s\in\R$, $1<p<\infty$), which consists of all tempered distributions $u$ such that $\langle D\rangle^s u\in L^p(\R^n)$. These spaces are closely related to the \emph{Slobodeckij spaces} $W^{s,p}(\R^n)$. For $1<p<\infty$ and $s=k+\sigma$ with $k\in\N_0$ and $0<\sigma<1$, the Slobodeckij space $W^{s,p}(\R^n)$ consists of all functions $u\colon\R^n\to\R$ in the Sobolev space $W^{k,p}(\R^n)$ whose Gagliardo seminorm
    \[
        [\nabla^{k}u]_{W^{\sigma,p}(\R^n)}=\left(\int_{\R^{2n}}\frac{|\nabla^{k}u (x) -\nabla^{k}u (y)|^p}{|x-y|^{n+\sigma p}}\right)^{1/p}
    \]
    is finite. In the case $p=2$ and $s\geq 0$, all of the above spaces coincide, that is, $H^s(\R^n)=H^{s,2}(\R^n)=W^{s,2}(\R^n)$. For a more detailed discussion of these spaces, and in particular their relationship to the \emph{Triebel--Lizorkin spaces} $F^s_{p,q}(\R^n)$, we refer the interested reader to the classical books \cite{Triebel-Theory-of-function-spaces,Interpolation-spaces-Bergh-Lofstrom} and to the articles \cite{JRPZ,RZ-low-reg,CRTZ} as well as the references therein. The estimate \eqref{eq: poincare inequality} is called \emph{(fractional) Poincar\'e inequality}. Similarly as in the classical case, the Poincar\'e inequality remains valid when $\Omega\subset\R^n$ is only bounded in one direction and also holds in the $L^p$ setting (see \cite{JRPZ}). The proof is similar to the one of \ref{mapping properties}. The property \ref{UCP} is known as the \emph{unique continuity property (UCP)} of the fractional Laplacian and the first proof of this result goes back to M. Riesz \cite[Section 11 and 26]{riesz1988integrales}. His argument was based on the Kelvin transform and uniqueness for the moment problem. We also refer the interested reader to \cite{TGMS,ruland2015unique} and \cite{KRZ} for a generalization to the $L^p$ setting. An immediate consequence of the Poincar\'e inequality and \eqref{eq: equivalent norm Hs} is that 
    \[
     \|u\|_{\widetilde{H}^s(\Omega )}\vcentcolon =\|(-\Delta)^{s/2}u\|_{L^2(\R^n)}
    \]
    defines an equivalent norm on $\widetilde{H}^s(\Omega)$, which will be used in the well-posedness theory below.
    
    Finally, let us recall the following useful estimate (see \cite{Ozawa}).
	\begin{proposition}\label{critical}
		For any $1<p<\infty$ there exists a constant $C>0$ only depending on $p$ and $n$ such that 
		\begin{equation}
		\|u\|_{L^q(\R^n)}\leq C q^{1-1/p}\|(-\Delta)^{\frac{n}{2p}}u\|_{L^p(\R^n)}^{1-p/q}\|u\|_{L^p(\R^n)}^{p/q} 
		\end{equation}
		for all $p\leq q<\infty$.
	\end{proposition}

    \subsection{Time dependent function spaces}

    As in the present article we study third order nonlocal wave equations, we also need some standard time-dependent function spaces, which we introduce next. If $(a,b)\subset\R$ and $X$ is a Banach space, then we denote by $C^k([a,b]\,;X)$, $L^p(a,b\,;X)$, with $k\in\N_0\cup\{\infty\}$ and $1\leq p\leq \infty$, the space of $k-$times continuously differentiable functions and the space of measurable functions $u\colon (a,b)\to X$ such that $t\mapsto \|u(t)\|_X\in L^p([a,b])$. For $k\in\N_0$ and $1\leq p\leq\infty$, they are Banach spaces when endowed with the following norms
	\begin{equation}
		\label{eq: Bochner spaces}
		\begin{split}
            \|u\|_{L^p(a,b\,;X)}&\vcentcolon = \left(\int_a^b\|u(t)\|_{X}^p\,dt\right)^{1/p}<\infty,\\
            \|u\|_{L^{\infty}(a,b;X)}&\vcentcolon = \sup_{a<t<b}\|u(t)\|_{X},\\
            \|u\|_{C^k([a,b];X)}&\vcentcolon = \sup_{0\leq \ell\leq k}\|\partial_t^{\ell}u\|_{L^{\infty}(a,b;X)}.
		\end{split}
	\end{equation} 
	
	Moreover, when $u\in L^1_{\loc}(a,b\,;X)$ and $X$ is a space of functions defined in a subset of $\R^n$, like $L^2(\Omega)$ or $H^s(\R^n)$, then we identify $u$ with a function $u(x,t)$ and $u(t)$ denotes the function $x\mapsto u(x,t)$ for almost all $t$. Furthermore, when $X$ is a Banach space over a domain $\Omega\subset\R^n$, then we identify the (vectorial) distributional derivative $\frac{du}{dt}\in \distr((a,b)\,;X)$ with the (scalar) distributional derivative $\partial_tu\in \distr(\Omega\times (a,b))$. Here $ \distr((a,b)\,;X)$ stands for the class of $X$ valued distributions on $(a,b)$ and $\distr(\Omega\times (a,b))$ are the scalar valued distributions on $\Omega\times (a,b)$.

    Finally, for any $k\in\N_0$ and $1\leq p\leq \infty$, we introduce the generalized Sobolev space $W^{k,p}(a,b;X)$ inductively by
    \[
        W^{k,p}(a,b;X)\vcentcolon =\{ u\in W^{k-1,p}(a,b;X)\,;\, \partial_t u\in W^{k-1,p}(a,b;X) \}.
    \]
    Here, we use the convention $W^{0,p}(a,b;X)=L^p(a,b;X)$. 
    The spaces $W^{k,p}(a,b;X)$ are Banach spaces under the norm
    \[
        \|u\|_{W^{k,p}(a,b;X)}\vcentcolon = \left(\sum_{\ell=0}^k \|\partial_t^{\ell} u\|_{L^p(a,b;X)}^p\right)^{1/p}.
    \]
    Furthermore, in the special case $p=2$ we denote the Hilbert spaces $W^{k,2}(a,b;X)$ by $H^k(a,b;X)$.

	\section{Well-posedness results for linear and nonlinear forward problems}
    \label{sec: well-posedness forward}

In this section, we establish well-posedness results for various third order nonlocal wave equations. Section~\ref{subsec: linear MGT} deals with the existence of solutions to the MGT equation and, in particular, it provides all the estimates needed to tackle the existence of solutions to the semilinear nonlocal MGT equations and nonlocal JMGT equations of Westervelt-type in Sections \ref{sec: Well-posedness semilinear problem} and \ref{subsec: well-posedness westervelt}, respectively.
	
	\subsection{Well-posedness of the linear nonlocal MGT equation}
    \label{subsec: linear MGT}
    
	In a first step, we prove the unique solvability of the following homogeneous, nonlocal MGT equation
	\begin{equation}
    \label{linear MGT without ext cord}
	\begin{cases}
		(\partial_t^3+\alpha\partial_t^2+b(-\Delta)^s\partial_t+c(-\Delta)^s+q)u=F\ & {\rm in}\ \Omega_T,\\
		u=0\ & {\rm in}\ (\Omega_e)_T, \\
		u(0)=u_0, \partial_t u(0)=u_1, \partial_t^2 u(0)=u_2\ & {\rm in}\ \Omega,
	\end{cases}
	\end{equation}
	where $(u_0,u_1,u_2)\in \widetilde{H}^s(\Omega)\times \widetilde{H}^s(\Omega)\times \widetilde{L}^2(\Omega)$ and $F\in L^2(0,T;\widetilde{L}^2(\Omega))$. Afterwards, in a second step, we will extend this existence result to nonzero exterior conditions. 
    
    Before discussing the well-posedness result for the above linear nonlocal MGT equation \eqref{linear MGT without ext cord}, we need some auxiliary lemmas. First, let us introduce the following space
    \begin{equation}
        H^1_T(0,T;X)\vcentcolon = \{u\in H^1(0,T;X)\,;\,u(T)=0\}
    \end{equation}
    for any Banach space $X$. Then, one may show the following lemma (cf.,~e.g.,~\cite{PZ-KK}).
    \begin{lemma}
    \label{lemma: density result}
        Let $\Omega\subset\R^n$ be a bounded Lipschitz domain and $s>0$. Then $C_c^{\infty}(\Omega\times [0,T))$ is dense in the Banach space
        \[
            X^{0,1}_{s,0,T}(\Omega_T)\vcentcolon = L^2(0,T;\widetilde{H}^s(\Omega))\cap H^1_T(0,T;\widetilde{L}^2(\Omega)),
        \]
        which is endowed with the norm
        \[
            \|u\|_{X^{0,1}_{s,0}(\Omega_T)}\vcentcolon = \max\Big(\|u\|_{L^2(0,T;\widetilde{H}^s(\Omega))},\|u\|_{H^1(0,T;\widetilde{L}^2(\Omega))}\Big).
        \]
        Furthermore, the set
        \[
            \mathcal{D}_T \vcentcolon =\left\{\sum_{k=1}^N \eta_k w_k\,;\,\eta_k \in C_c^{\infty}([0,T)),\,1\leq k\leq N,\,N\in\N\right\}
        \]
        is dense in $X^{0,1}_{s,0,T}(\Omega_T)$, where $(w_j)_{j\in\N}$ is any orthonormal basis of $\widetilde{H}^s(\Omega)$.
    \end{lemma}

    \begin{proof}
        One may first observe by the trace theorem that $X^{0,1}_{s,0,T}(\Omega_T)$ is a closed subspace of 
        \[
            X^{0,1}_{s,0}(\Omega_T)\vcentcolon = L^2(0,T;\widetilde{H}^s(\Omega))\cap H^1(0,T;\widetilde{L}^2(\Omega)),
        \]
        which carries the norm $\|\cdot\|_{X^{0,1}_{s,0}(\Omega_T)}$. As $X^{0,1}_{s,0}(\Omega_T)$ is a Banach space by \cite[Lemma 2.3.1]{Interpolation-spaces-Bergh-Lofstrom}, we deduce that $X^{0,1}_{s,0,T}(\Omega_T)$ is a Banach space. For the density assertion, we may recall that by the use of cutoff functions and mollification in time one can show that $C_c^{\infty}([0,T);\widetilde{H}^s(\Omega))$ is dense in $X^{0,1}_{s,0,T}(\Omega_T)$ (see \cite[Proposition B.1]{PZ-YL-porous}). Now, by classical results, it follows that $C_c^{\infty}(\Omega\times [0,T))$ is dense in $H^1(0,T;\widetilde{H}^s(\Omega))$ and thus, in particular, any $u\in C_c^{\infty}([0,T);\widetilde{H}^s(\Omega))$ can be approximated in $X^{0,1}_{s,0}(\Omega_T)$ by a sequence $(u_n)_{n\in\N}\subset C_c^{\infty}(\Omega\times [0,T))$. Hence, we have established the first density result.

        For the second density assertion, it is enough to recall that by classical results the set $\mathcal{D}_T$ is dense in $H^1(0,T;\widetilde{H}^s(\Omega))$ and so in particular we can approximate any $u\in C_c^{\infty}([0,T);\widetilde{H}^s(\Omega))$ by a sequence $(u_N)_{N\in\N}\subset \mathcal{D}_T$ in the topology of $H^1(0,T;\widetilde{H}^s(\Omega))$ and thus in particular in $X^{0,1}_{s,0,T}(\Omega_T)$.
        This finishes the proof.
    \end{proof}

    Furthermore, let us recall the following lemma, whose proof relies on Sobolev's inequality and Proposition \ref{critical}.
    \begin{lemma}[{\cite[Theorem 3.1]{PZ1}}]
    \label{lemma: potential}
        Let $\Omega\subset\R^n$ be a bounded Lipschitz domain, $T>0$ and $s>0$. Suppose that $q\in L^{\infty}(0,T;L^p(\Omega))$ is real-valued and $2\leq p\leq \infty$ satisfies \eqref{conp}, then the multiplication map $\widetilde{H}^s(\Omega)\ni u\mapsto q(t)u\in \widetilde{L}^2(\Omega)$ is continuous for a.e.~$0<t<T$ and there holds
        \begin{equation}
        \label{eq: estimate potential}
            \|q(t)u\|_{L^2(\Omega)}\leq \|q(t)\|_{L^p(\Omega)}\|u\|_{\widetilde{H}^s(\Omega)}
        \end{equation}
        for all $u\in \widetilde{H}^s(\Omega)$.
    \end{lemma}

 Throughout this work, we use the following definition of weak solutions to the nonlocal MGT equation \eqref{linear MGT without ext cord}.
	\begin{definition}
    \label{def: weak sol}
		Let $\Omega\subset\R^n$ be a bounded Lipschitz domain, $T>0$ and $s>0$. Suppose that $q\in L^{\infty}(0,T;L^p(\Omega))$ is real-valued and $2\leq p\leq \infty$ satisfies \eqref{conp}. For given initial data $(u_0,u_1,u_2)\in \widetilde{H}^s(\Omega)\times \widetilde{H}^s(\Omega)\times \widetilde{L}^2(\Omega)$ and the source term $F\in L^2(0,T;\widetilde{L}^2(\Omega))$, we say that $u\in W^{1,\infty}(0,T;\widetilde{H}^s(\Omega))\cap W^{2,\infty}(0,T;\widetilde{L}^2(\Omega))$ is a \emph{(weak) solution} to \eqref{linear MGT without ext cord} if
        \begin{enumerate}[(i)]
            \item $u(0)=u_0$, $\partial_t u(0)=u_1$  in $\Omega$
            \item and for any $\varphi\in C_c^{\infty}(\Omega\times [0,T))$ one has
            \begin{equation}
            \label{weakint}
		      \begin{split}
		          -\int_{\Omega_T}(\partial_t^2 u)\partial_t\varphi\,dxdt+\int_0^T\langle L_{\alpha,b,c,q}u,\varphi\rangle\,dt=\int_{0}^T\langle F,\varphi\rangle \,dt+\int_{\Omega}u_2\varphi(0)\,dx,
		      \end{split}
		      \end{equation}
              where we introduced the operator
              \[
                L_{\alpha,b,c,q}\vcentcolon =  \alpha\partial_t^2+b(-\Delta)^s\partial_t+c(-\Delta)^s+q
              \]
              and $\langle \cdot,\cdot\rangle$ denotes throughout the whole article the natural (spatial) duality pairing.
        \end{enumerate}
	\end{definition}
    
	The well-posedness result for the nonlocal MGT equation \eqref{linear MGT without ext cord} reads as follows.
	\begin{theorem}[Homogeneous MGT equations]
    \label{wellposednessv}
		Let $\Omega\subset\R^n$ be a bounded Lipschitz domain, $T>0$ and $s>0$. Suppose that $q\in L^{\infty}(0,T;L^p(\Omega))$ is real-valued and $2\leq p\leq \infty$ satisfies \eqref{conp}. Then for all
        \[
            (u_0,u_1,u_2,F)\in \widetilde{H}^s(\Omega)\times \widetilde{H}^s(\Omega)\times \widetilde{L}^2(\Omega)\times L^2(0,T;\widetilde{L}^2(\Omega)),
        \]
        there exists a unique solution 
        \begin{equation}
        \label{eq: regularity weak sol}
            u\in W^{1,\infty}(0,T;\widetilde{H}^s(\Omega))\cap W^{2,\infty}(0,T;\widetilde{L}^2(\Omega))\cap H^3(0,T;H^{-s}(\Omega))
        \end{equation}
        to problem \eqref{linear MGT without ext cord}. Moreover, it satisfies the energy estimate
		\begin{equation}
        \label{energyest}
        \begin{split}
		 &\|\partial_t^2 u\|_{L^{\infty}(0,T;L^2(\Omega))}+\|u\|_{W^{1,\infty}(0,T;\widetilde H^s(\Omega))} \\
		 &\lesssim  \|(-\Delta)^{s/2}u_0\|_{L^2(\R^n)}+\|(-\Delta)^{s/2}u_1\|_{L^2(\R^n)}+\|u_2\|_{L^2(\Omega)}+\|F\|_{L^2(\Omega_T)}. 
        \end{split}
		\end{equation}
	\end{theorem}

    \begin{remark}
    \label{remark: on unbounded case}
        The proof below shows that the boundedness of $\Omega$ is not essential and could be replaced by the weaker condition that $\Omega$ is only bounded in one direction (see \cite[Theorem 3.6]{LRZ-nonlocal-diffusion} and \cite[Theorem 2.2]{JRPZ}).
    \end{remark}

 \begin{remark}
    \label{rem: dependence of energy constant}
        We emphasize that the constant appearing in the energy estimate \eqref{energyest}, which we denote by $C_0$, depends monotonically and continuously on the norm $\|q\|_{L^{\infty}(0,T;L^p(\Omega))}$; see Lemma~\ref{lemma: potential}, estimates \eqref{eq: first energy estimate for existence}–\eqref{eq: fourth energy estimate for existence}, and Gronwall’s inequality \cite[p.~559]{DautrayLionsVol5}. In addition, $C_0$ depends on $n,s,\Omega$ through the fractional Poincar\'e constant (see \cite[Theorem 2.2]{JRPZ}), as well as on the parameters $\alpha,b,c$, and the finite time horizon $T$. 
    \end{remark}

    \begin{proof}
	The proof is divided into four steps. In the first three steps we demonstrate the existence of a solution to \eqref{linear MGT without ext cord} and in the last step we show that the constructed solution is unique.
	
	{\bf Step 1:} \textit{Approximate problem.} Let us recall that since $\widetilde{H}^s(\Omega)$ is a separable Hilbert space, there exists an orthonormal basis $(w_m)_{m\in\N}$ of $\widetilde{H}^s(\Omega)$ and hence the finite-dimensional subspaces 
		\begin{equation}
        \label{eq: Galerkin approx space}
		\widetilde{H}^s_m\vcentcolon =\text{span}\{w_1,\ldots,w_m\},
		\end{equation}
		for $m\in\N$, constitute a Galerkin approximation of $\widetilde{H}^s(\Omega)$ (see \cite[Chapter XVIII, \S 2]{DautrayLionsVol5}). If $\Omega\subset\R^n$ is bounded, as assumed in this paper, then one can take, for example, $(w_m)_{m\in\N}$ to be the eigenfunctions of the fractional Laplacian $(-\Delta)^s$ with homogeneous Dirichlet conditions (see \cite[Lemma 3.5]{LTZ1}), but this is not essential for our argument, and for the sake of generality as well as justifying Remark \ref{remark: on unbounded case} we stay in this abstract setting. Using that the embedding $\widetilde{H}^s(\Omega)\hookrightarrow \widetilde{L}^2(\Omega)$ is dense, we see that the finite-dimensional subspace $\widetilde{H}^s_m$, $m\in\N$, form a Galerkin approximation of $\widetilde{L}^2(\Omega)$ too. Hence, there exist sequences $(u^m_{j})_{m\in\N}\subset\widetilde{H}^s(\Omega)$, $0\leq j\leq 2$, satisfying
        \begin{equation}
        \label{eq: convergence intial cond}
            u^m_{j}\in\widetilde{H}^s_m \text{ for }0\leq j\leq 2\text{ and }\begin{cases}
                u^m_0\to u_0\text{ in }\widetilde{H}^s(\Omega),\\
                 u^m_1\to u_1\text{ in }\widetilde{H}^s(\Omega),\\
                  u^m_2\to u_2\text{ in }\widetilde{L}^2(\Omega)
            \end{cases}
        \end{equation}
        as $m\to\infty$. Now, we want to find functions
        \[
            u_m=\sum_{j=1}^m c_m^jw_j\in\widetilde{H}^s_m
        \]
        solving the \emph{approximate problem}
        \begin{equation}
        \label{eq: approximate problem}
        \begin{cases}
            \langle \partial_t^3 u_m(t),w_j\rangle_{L^2(\Omega)}+\langle L_{\alpha,b,c,q}u_m(t),w_j\rangle=\langle F(t),w_j\rangle_{L^2(\Omega)}&\text{ for }0<t<T,\\
            u_m(0)=u^m_0,\,\partial_t u_m(0)=u^m_1,\,\partial_t^2 u_m(0)=u^m_2
        \end{cases}
        \end{equation}
        for all $1\leq j\leq m$. Thus, the approximate problem is the (formal) projection of the original problem \eqref{linear MGT without ext cord} onto the subspaces $\widetilde{H}^s_m$. Next, let us introduce the quantities
        \begin{equation}
        \label{eq: matrix elements}
        \begin{split}
            A^m_{ij}&=\langle w_i,w_j\rangle_{L^2(\Omega)}, \, B^m_{ij}=\langle (-\Delta)^{s/2} w_i,(-\Delta)^{s/2}w_j\rangle_{L^2(\R^n)}\\
            Q^m_{ij}&=\langle q(t)w_i,w_j\rangle_{L^2(\Omega)},\,F^m_{j}=\langle F(t),w_j\rangle_{L^2(\Omega)}
        \end{split}
        \end{equation}
        for $1\leq i,j\leq m$. Using \eqref{eq: matrix elements}, the problem \eqref{eq: approximate problem} can be equivalently formulated as
        \begin{equation}
        \label{eq: approximate problem 2}
             \begin{cases}
            A_m\partial_t^3 c_m+\alpha A_m\partial_t^2 c_m+bB_m\partial_t c_m+cB_m c_m+Q_m c_m=F_m &\text{ for }0<t<T,\\
            c_m(0)=d_{0,m},\,\partial_t c_m(0)=d_{1,m},\,\partial_t^2 c_m(0)=d_{2,m},
        \end{cases}
        \end{equation}
        where we set $c_m=(c_m^1,\ldots,c_m^m)$, $F_m=(F^m_1,\ldots,F^m_m)$, $A_m=(A^m_{ij})_{1\leq i,j\leq m}$, $B_m=(B^m_{ij})_{1\leq i,j\leq m}$, $Q_m=(Q^m_{ij})_{1\leq i,j\leq m}$,
        \begin{equation}
        \label{eq: intial conditions expansion}
            u^m_j=\sum_{k=1}^m d_{j,m}^k w_k
        \end{equation}
        and $d_{j,m}=(d_{j,m}^1,\ldots,d_{j,m}^m)$ for $0\leq j\leq 2$. By the linear independence of $(w_j)_{j\in\N}$, we know that $\det A_m\neq 0$, which ensures the invertibility the matrices $A_m$. Hence, \eqref{eq: approximate problem 2} is equivalent to
        \begin{equation}
        \label{eq: approximate problem 3}
             \begin{cases}
            \partial_t^3 c_m+\alpha \partial_t^2 c_m+b\widetilde{B}_m\partial_t c_m+c\widetilde{B}_m c_m+\widetilde{Q}_m c_m=\widetilde{F}_m &\text{ for }0<t<T,\\
            c_m(0)=d_{0,m},\,\partial_t c_m(0)=d_{1,m},\,\partial_t^2 c_m(0)=d_{2,m},
        \end{cases}
        \end{equation}
        where $\widetilde{B}_m=A_m^{-1}B_m$, $\widetilde{Q}_m=A_m^{-1}Q_m$ and $\widetilde{F}_m=A_m^{-1}F_m$. Finally, defining 
        \begin{equation}
        \label{eq: final coefficients}
        \begin{split}
            \mathcal{B}_m &=\begin{pmatrix}
               0 & 1 &  0\\
                0 & 0 & 1\\
                -(c\widetilde{B}_m +\widetilde{Q}_m) & -b\widetilde{B}_m & -\alpha
           \end{pmatrix},\\
           \mathcal{F}_m &=(0,0,\widetilde{F}_m)\text{ and } D_m=(d_{0,m},d_{1,m},d_{2,m}),
        \end{split}
        \end{equation}
        we see that $C_m=(c_m,\partial_t c_m,\partial_t^2 c_m)$ solves the first order problem
         \begin{equation}
        \label{eq: approximate problem 4}
             \begin{cases}
            \partial_t C_m=\mathcal{B}_mC_m+\mathcal{F}_m &\text{ for }0<t<T,\\
            C_m(0)=D_m.
        \end{cases}
        \end{equation}
        Before proceeding, let us note that as $q\in L^{\infty}(0,T;L^p(\Omega))$ we only have $\mathcal{B}_m\in L^{\infty}(0,T;\R^{3m\times 3m})$. Thus, we cannot directly apply the Picard--Lindel\"of theorem to construct a unique solution, but the Banach fixed point theorem still allows to prove the existence of a unique solution $C_m$. To show this, let us note that $C_m\in C([0,T];\R^{3m})\cap H^1(0,T;\R^{3m})$ solves \eqref{eq: approximate problem 4} if and only if $C_m$ is a fixed point of the map $\Phi\colon C([0,T];\R^{3m})\to  C([0,T];\R^{3m})$ defined by
        \begin{equation}\label{eq: contraction map}
            \Phi(C_m)(t)=D_m+\int_0^t \mathcal{B}_m C_m\,d\tau+\int_0^t \mathcal{F}_m\,d\tau.
        \end{equation}
        Next, let us fix some $0<T_0\leq T$ and observe that for all $C_1,C_2\in C([0,T_0];\R^{3m})$, one easily estimates
       \begin{equation}
        \begin{split}
             \|\Phi(C_1)-\Phi(C_2)\|_{L^{\infty}(0,T_0;\R^{3m})}&\leq T_0\|\mathcal{B}_m\|_{L^{\infty}(0,T;\R^{3m\times 3m})}\|C_1-C_2\|_{L^{\infty}(0,T_0;\R^{3m})}.
        \end{split}
       \end{equation}
        Hence, if $T_0$ is sufficiently small, then $\Phi$ is a strict contraction from $C([0,T_0];\R^{3m})$ to itself. Therefore, the Banach fixed point theorem guarantees the existence of a unique solution $C_m\in C([0,T_0];\R^{3m})\cap H^1(0,T_0;\R^{3m})$ on $[0,T_0]$. Iterating this argument, we may deduce the existence of a unique solution $C_m\in C([0,T];\R^{3m})\cap H^1(0,T;\R^{3m})$ on $[0,T]$. Hence, we have established that for each $m\in\N$ there exists a unique solution 
        \begin{equation}
        \label{eq: regularity of coeff cm}
            c_m\in C^2([0,T];\R^m)\cap H^3(0,T;\R^m)
        \end{equation}
        of \eqref{eq: approximate problem 2} and the corresponding functions $u_m$, defined via \eqref{eq: intial conditions expansion}, solve \eqref{eq: approximate problem} and satisfy
        \begin{equation}
        \label{eq: regularity of coeff um}
            u_m\in C^2([0,T];\widetilde{H}^s_m)\cap H^3(0,T;\widetilde{H}^s_m).
        \end{equation}
	
	{\bf Step 2:} \textit{Energy estimate.} To derive suitable energy estimates, we multiply \eqref{eq: approximate problem} by $\partial_t^2 c_m^j\in C([0,T])$ (see \eqref{eq: regularity of coeff cm}) and sum the resulting identity from $j=1$ to $j=m$. We conclude that 
    \[
        \langle \partial_t^3 u_m(\tau),\partial_t^2 u_m(\tau)\rangle_{L^2(\Omega)}+\langle L_{\alpha,b,c,q}u_m(\tau),\partial_t^2 u_m(\tau)\rangle=\langle F(\tau),\partial_t^2 u_m(\tau)\rangle_{L^2(\Omega)}
    \]
    for all $0<\tau<T$. Thus, an integration over $[0,t]\subset [0,T]$, applying the integration by parts formula and using Lemma \ref{lemma: potential}, we deduce
    \begin{equation}
    \label{eq: first energy estimate for existence}
        \begin{split}
            &\frac{1}{2}\|\partial_t^2 u_m(t)\|_{L^2(\Omega)}^2+\frac{b}{2}\|(-\Delta)^{s/2}\partial_t u_m(t)\|_{L^2(\R^n)}^2\\
            &=\frac{1}{2}\|u_2^m\|_{L^2(\Omega)}^2+\frac{b}{2}\|(-\Delta)^{s/2}u^m_1\|_{L^2(\R^n)}^2+c\langle (-\Delta)^{s/2}u_0^m,(-\Delta)^{s/2}u_1^m\rangle_{L^2(\R^n)}\\
            &\quad -\alpha\int_0^t\|\partial_t^2 u_m\|_{L^2(\Omega)}^2\,d\tau-\int_0^t\langle qu_m,\partial_t^2 u_m\rangle_{L^2(\Omega)}\,d\tau\\
            &\quad +c\int_{0}^t\|(-\Delta)^{s/2}\partial_t u_m\|_{L^2(\R^n)}^2\,d\tau-c\langle (-\Delta)^{s/2}u_m(t),(-\Delta)^{s/2}\partial_t u_m(t)\rangle_{L^2(\R^n)}\\
            &\quad +\int_0^t\langle F,\partial_t^2 u_m\rangle_{L^2(\Omega)}\,d\tau\\
            &\leq C(\|(-\Delta)^{s/2}u_0^m\|_{L^2(\R^n)}^2+\|(-\Delta)^{s/2}u_1^m\|_{L^2(\R^n)}^2+\|u_2^m\|_{L^2(\Omega)}^2+\|F\|_{L^2(0,T;L^2(\Omega))}^2)\\
            &\quad +C\int_0^t(\|(-\Delta)^{s/2}u_m\|_{L^2(\R^n)}^2+\|(-\Delta)^{s/2}\partial_t u_m\|_{L^2(\R^n)}^2+\|\partial_t^2 u_m\|_{L^2(\Omega)}^2)\,d\tau\\
            &\quad - c\langle (-\Delta)^{s/2}u_m(t),(-\Delta)^{s/2}\partial_t u_m(t)\rangle_{L^2(\R^n)}.
        \end{split}
    \end{equation}
    In the first equality we used the fundamental theorem of calculus and
    \[
    \begin{split}
          \langle (-\Delta)^{s/2}u_m,(-\Delta)^{s/2}\partial_t^2 u_m\rangle_{L^2(\R^n)}&=\partial_t  \langle (-\Delta)^{s/2}u_m,(-\Delta)^{s/2}\partial_t u_m\rangle_{L^2(\R^n)}\\
          &\quad -\|(-\Delta)^{s/2}\partial_t u_m\|_{L^2(\R^n)}^2.
    \end{split}
    \]
    Next note that by the fundamental theorem of calculus and Jensen's inequality, the term in the last line of \eqref{eq: first energy estimate for existence} can be estimated as
    \begin{equation}
    \label{eq: estimate product}
        \begin{split}
            &-c\langle (-\Delta)^{s/2}u_m(t),(-\Delta)^{s/2}\partial_t u_m(t)\rangle_{L^2(\R^n)}\\
            &\leq C\|(-\Delta)^{s/2}u_m(t)\|_{L^2(\R^n)}^2+\eta\|(-\Delta)^{s/2}\partial_t u_m(t)\|_{L^2(\R^n)}^2\\
            &\leq C\|(-\Delta)^{s/2}u_0^m\|_{L^2(\R^n)}^2+C\left\|\int_0^t(-\Delta)^{s/2}\partial_t u_m\,d\tau\right\|_{L^2(\R^n)}^2+\eta\|(-\Delta)^{s/2}\partial_t u_m(t)\|_{L^2(\R^n)}^2\\
            &\leq C\|(-\Delta)^{s/2}u_0^m\|_{L^2(\R^n)}^2+C\int_0^t \|(-\Delta)^{s/2}\partial_t u_m\|_{L^2(\R^n)}^2\,d\tau+\eta\|(-\Delta)^{s/2}\partial_t u_m(t)\|_{L^2(\R^n)}^2
        \end{split}
    \end{equation}
    for any $\eta>0$, where we used
    \begin{equation}
    \label{eq: fund thm of calc existence}
        (-\Delta)^{s/2}u_m(t)=(-\Delta)^{s/2}u^m_0+\int_0^t(-\Delta)^{s/2}\partial_t u_m\,d\tau.
    \end{equation}
    This is permissible as $u_m\in C^1([0,T];\widetilde{H}^s_m)$ (see \eqref{eq: regularity of coeff um}). Combining this with \eqref{eq: first energy estimate for existence} and choosing $\eta=b/4$, we deduce the energy estimate
    \begin{equation}
    \label{eq: second energy estimate for existence}
        \begin{split}
             &\frac{1}{2}\|\partial_t^2 u_m(t)\|_{L^2(\Omega)}^2+\frac{b}{4}\|(-\Delta)^{s/2}\partial_t u_m(t)\|_{L^2(\R^n)}^2\\
             &\leq C(\|(-\Delta)^{s/2}u_0^m\|_{L^2(\R^n)}^2+\|(-\Delta)^{s/2}u_1^m\|_{L^2(\R^n)}^2+\|u_2^m\|_{L^2(\Omega)}^2+\|F\|_{L^2(0,T;L^2(\Omega))}^2)\\
             &\quad +C\int_0^t(\|(-\Delta)^{s/2}u_m\|_{L^2(\R^n)}^2+\|(-\Delta)^{s/2}\partial_t u_m\|_{L^2(\R^n)}^2+\|\partial_t^2 u_m\|_{L^2(\Omega)}^2)\,d\tau.
        \end{split}
    \end{equation}
    By using the estimate for $\|(-\Delta)^{s/2}u_m(t)\|_{L^2(\R^n)}^2$ contained in \eqref{eq: estimate product}, this implies
    \begin{equation}
        \label{eq: third energy estimate for existence}
        \begin{split}
             &\|\partial_t^2 u_m(t)\|_{L^2(\Omega)}^2+\|(-\Delta)^{s/2}\partial_t u_m(t)\|_{L^2(\R^n)}^2+\|(-\Delta)^{s/2} u_m(t)\|_{L^2(\R^n)}^2\\
             &\leq C(\|(-\Delta)^{s/2}u_0^m\|_{L^2(\R^n)}^2+\|(-\Delta)^{s/2}u_1^m\|_{L^2(\R^n)}^2+\|u_2^m\|_{L^2(\Omega)}^2+\|F\|_{L^2(0,T;L^2(\Omega))}^2)\\
             &\quad +C\int_0^t(\|(-\Delta)^{s/2}u_m\|_{L^2(\R^n)}^2+\|(-\Delta)^{s/2}\partial_t u_m\|_{L^2(\R^n)}^2+\|\partial_t^2 u_m\|_{L^2(\Omega)}^2)\,d\tau.
        \end{split}
    \end{equation}
    Taking into account the convergence results in \eqref{eq: convergence intial cond}, we may deduce
    \begin{equation}
        \label{eq: fourth energy estimate for existence}
        \begin{split}
             &\|\partial_t^2 u_m(t)\|_{L^2(\Omega)}^2+\|(-\Delta)^{s/2}\partial_t u_m(t)\|_{L^2(\R^n)}^2+\|(-\Delta)^{s/2} u_m(t)\|_{L^2(\R^n)}^2\\
             &\leq C(\|(-\Delta)^{s/2}u_0\|_{L^2(\R^n)}^2+\|(-\Delta)^{s/2}u_1\|_{L^2(\R^n)}^2+\|u_2\|_{L^2(\Omega)}^2+\|F\|_{L^2(0,T;L^2(\Omega))}^2)\\
             &\quad +C\int_0^t(\|(-\Delta)^{s/2}u_m\|_{L^2(\R^n)}^2+\|(-\Delta)^{s/2}\partial_t u_m\|_{L^2(\R^n)}^2+\|\partial_t^2 u_m\|_{L^2(\Omega)}^2)\,d\tau
        \end{split}
    \end{equation}
    for sufficiently large $m\in\N$, where the constant $C>0$ is independent of $m$.
    Thus, Gronwall's inequality \cite[p.~559]{DautrayLionsVol5} ensures that
    \begin{equation}
    \label{eq: fifth energy estimate for existence}
        \begin{split}
            &\|\partial_t^2 u_m(t)\|_{L^2(\Omega)}^2+\|(-\Delta)^{s/2}\partial_t u_m(t)\|_{L^2(\R^n)}^2+\|(-\Delta)^{s/2} u_m(t)\|_{L^2(\R^n)}^2\\
            &\leq C(\|(-\Delta)^{s/2}u_0\|_{L^2(\R^n)}^2+\|(-\Delta)^{s/2}u_1\|_{L^2(\R^n)}^2+\|u_2\|_{L^2(\Omega)}^2+\|F\|_{L^2(0,T;L^2(\Omega))}^2)
        \end{split}
    \end{equation}
    for all $0\leq t\leq T$ and sufficiently large $m\in\N$. \\
	
	{\bf Step 3.} \textit{Passing to the limit.} From the estimate \eqref{eq: fifth energy estimate for existence}, we deduce the existence of a function
    \begin{equation}
    \label{eq: limit function u}
        u\in W^{1,\infty}(0,T;\widetilde{H}^s(\Omega))\cap W^{2,\infty}(0,T;\widetilde{L}^2(\Omega))
    \end{equation}
    and up to extraction of a subsequence we have the convergence results
    \begin{equation}
    \label{eq: weak star convergence of galerkin}
        \begin{cases}
            u_m\weakstar u & \text{ in }L^{\infty}(0,T;\widetilde{H}^s(\Omega)),\\
        \partial_t u_m\weakstar \partial_t u  & \text{ in }L^{\infty}(0,T;\widetilde{H}^s(\Omega)),\\
        \partial_t^2 u_m\weakstar \partial_t^2 u &  \text{ in }L^{\infty}(0,T;\widetilde{L}^2(\Omega))
        \end{cases}
    \end{equation}
    as $m\to \infty$. From these convergence results and the trace theorem, we may deduce that
    \begin{equation}
    \label{eq: first two intial conditions}
        \begin{split}
            u(0)=u_0\text{ and }\partial_t u(0)=u_1.
        \end{split}
    \end{equation}
    More concretely, the trace theorem guarantees that for any Banach space $X$ the map $H^1(0,T; X)\ni u\mapsto u(0)\in X$ is continuous and thus using that \eqref{eq: weak star convergence of galerkin} implies
    \begin{equation}
    \label{eq: weak convergence of galerkin}
        \begin{cases}
            u_m\weak u & \text{ in }H^1(0,T;\widetilde{H}^s(\Omega)),\\
        \partial_t u_m\weak \partial_t u & \text{ in }H^1(0,T;\widetilde{L}^2(\Omega))
        \end{cases}
    \end{equation}
    as $m\to\infty$. We deduce that
    \[
        \begin{cases}
            u^m_0=u_m(0)\weak u(0) & \text{ in }\widetilde{H}^s(\Omega),\\
            u^m_1=\partial_t u_m(0)\weak \partial_t u(0) & \text{ in }\widetilde{L}^2(\Omega)
        \end{cases}
    \]
    as $m\to \infty$. Hence, \eqref{eq: convergence intial cond} ensures \eqref{eq: first two intial conditions}.
    Next, we multiply \eqref{eq: approximate problem} by $\eta_j\in C_c^{\infty}([0,T))$, $j\in\N$, integrate over $[0,T]$ and apply the integration by parts formula to obtain
    \[
    \begin{split}
		         & -\int_{\Omega_T}(\partial_t^2 u_m)\partial_t\eta_j w_j\,dxdt+\int_0^T\langle L_{\alpha,b,c,q}u_m,\eta_jw_j\rangle\,dt\\
                  &=\int_{0}^T\langle F,\eta_j w_j\rangle \,dt+\int_{\Omega}u^m_2\eta_j(0)w_j\,dx
		      \end{split}
    \]
    for all $0\leq j\leq m$. By \eqref{eq: weak star convergence of galerkin} and Lemma \ref{lemma: potential}, we can pass to the limit $m\to\infty$ and obtain
    \[
    \begin{split}
		        & -\int_{\Omega_T}(\partial_t^2 u)\partial_t\eta_j w_j\,dxdt+\int_0^T\langle L_{\alpha,b,c,q}u,\eta_jw_j\rangle\,dt\\
                  &=\int_{0}^T\langle F,\eta_j w_j\rangle \,dt+\int_{\Omega}u_2\eta_j(0)w_j\,dx
		      \end{split}
    \]
    for all $j\in\N$. So, we can apply the second density result in Lemma \ref{lemma: density result} and the trace theorem to conclude that we have
    \begin{equation}
    \label{eq: identity pde}
    \begin{split}
		          -\int_{\Omega_T}(\partial_t^2 u)\partial_t\varphi\,dxdt+\int_0^T\langle L_{\alpha,b,c,q}u,\varphi\rangle\,dt=\int_{0}^T\langle F,\varphi\rangle \,dt+\int_{\Omega}u_2\varphi(0)\,dx
		      \end{split}
    \end{equation}
    for any $\varphi\in X^{0,1}_{s,0,T}(\Omega_T)$. Hence, \eqref{eq: identity pde} holds, in particular, for any $\varphi\in C_c^{\infty}(\Omega\times [0,T))$, which is nothing else than the identity \eqref{weakint}. Thus, $u$ is a weak solution in the sense of Definition \ref{def: weak sol}. Next, note that the asserted regularity $\partial_t^3 u\in L^2(0,T;H^{-s}(\Omega))$ is an immediate consequence of \eqref{eq: identity pde}. Finally, the energy estimate \eqref{energyest} follows from \eqref{eq: fifth energy estimate for existence}, the weak convergence results \eqref{eq: weak star convergence of galerkin} and the fact that for any Banach space $X$ the norm $\|\cdot\|_{L^{\infty}(0,T;X)}$ is weak-$\ast$ lower semicontinuous. \\

	{\bf Step 4.} \textit{Uniqueness of solutions.} In light of Definition \ref{def: weak sol}, it is enough to show that any solution $u\in W^{1,\infty}(0,T;\widetilde{H}^s(\Omega))\cap W^{2,\infty}(0,T;\widetilde{L}^2(\Omega))$ of 
    \begin{equation}
    \label{uniqueness}
    \begin{cases}
		(\partial_t^3+\alpha\partial_t^2+b(-\Delta)^s\partial_t+c(-\Delta)^s+q)u=0\ & {\rm in}\ \Omega_T,\\
		u=0\ & {\rm in}\ (\Omega_e)_T, \\
		u(0)=0, \partial_t u(0)=0, \partial_t^2 u(0)=0\ & {\rm in}\ \Omega,
	\end{cases}
    \end{equation}
    needs to be identically zero. First, we show the uniqueness of solutions for the case $\alpha=q=0$ and then explain how to modify the argument to deduce the uniqueness in the general case.
    
    The proof is similar to the one for second order wave equations, whose proof goes back to O.~A.~Ladyzhenskaya  \cite{Ladyzhenskaya}. To this end, let us consider for fixed $0<t_0<T$ the functions
    \begin{equation}
    \label{eq: def of phi}
        \varphi(t)\vcentcolon =\begin{cases}
            \int_t^{t_0} u(\tau)\,d\tau&\text{ for }0\leq t\leq t_0,\\
            0&\text{ for }t_0\leq t\leq T
        \end{cases}
    \end{equation}
    and
    \begin{equation}
    \label{eq: def of psi}
        \psi(t)\vcentcolon =\begin{cases}
            \int_t^{t_0}(\tau-t) u(\tau)\,d\tau&\text{ for }0\leq t\leq t_0,\\
            0&\text{ for }t_0\leq t\leq T.
        \end{cases}
    \end{equation}
    Furthermore, note that Fubini's theorem yields
    \begin{equation}
    \label{eq: equivalent expression psi}
        \begin{split}
            \int_{t}^{t_0} \varphi(\tau)\,d\tau &=\int_{t}^{t_0}\int_{\tau}^{t_0}u(\rho)\,d\rho\,d\tau\\
            &=\int_0^{t_0}\int_0^{t_0}\chi_{[t,t_0]}(\tau)\chi_{[\tau,t_0]}(\rho)u(\rho)\,d\rho\,d\tau\\
        &=\int_0^{t_0}\int_0^{t_0}\chi_{[t,t_0]}(\tau)\chi_{[0,\rho]}(\tau)u(\rho)\,d\tau\,d\rho=\psi(t)
        \end{split}
    \end{equation}
    for all $0\leq t\leq t_0$. The functions $\varphi$ and $\psi$ have the following properties (cf.~e.g.~\cite{PZ-KK}):
    \begin{enumerate}[(i)]
        \item\label{phi} $\varphi\in H^1_T(0,T;\widetilde{H}^s(\Omega))\cap C([0,T];\widetilde{H}^s(\Omega))$ with
      \begin{equation}
      \label{eq: derivative of phi}
           \partial_t\varphi=-u\chi_{[0,t_0]}
      \end{equation}
        \item\label{psi} and $\psi,\partial_t \psi\in H^1_T(0,T;\widetilde{H}^s(\Omega))\cap C([0,T];\widetilde{H}^s(\Omega))$ with
        \begin{equation}
        \label{eq: derivatives of psi}
        \psi(0)=\int_0^{t_0}\varphi(\tau)\,d\tau\text{ and }\begin{cases}
            \partial_t \psi=-\varphi,\\
            \partial_t^2 \psi =\chi_{[0,t_0]}u.
        \end{cases}
        \end{equation}
    \end{enumerate}
    Next, let us observe that by Definition \ref{def: weak sol} and Lemma \ref{lemma: density result}, we can use $\psi$ as a test function, which leads to
    \begin{equation}
    \label{eq: first identity for uniqueness}
    \begin{split}
        &-\int_{\Omega_T}(\partial_t^2 u)\partial_t \psi\,dxdt+b\int_0^T \langle (-\Delta)^{s/2} \partial_t u,(-\Delta)^{s/2}\psi\rangle_{L^2(\R^n)}\,dt\\
        &+c\int_0^T \langle (-\Delta)^{s/2} u,(-\Delta)^{s/2}\psi\rangle_{L^2(\R^n)}\,dt=0.
    \end{split}
    \end{equation}
	Next, observe that by the regularity assumptions on $u$, $u(0)=\partial_t u(0)=0$, \eqref{eq: def of phi}--\eqref{eq: def of psi}, \ref{phi}--\ref{psi} and the integration by parts formula, we may compute
    \begin{equation}
    \label{eq: highest order term}
        \begin{split}
            -\int_{\Omega_T}(\partial_t^2 u)\partial_t \psi\,dxdt&=\int_0^{t_0}\langle \partial_t^2 u,\varphi\rangle\,dt\\
            &=-\int_0^{t_0}\langle \partial_t u,\partial_t \varphi\rangle\,dt-\langle \partial_t u(0),\varphi(0)\rangle+\langle \partial_t u(t_0),\varphi(t_0)\rangle\\
            &=\int_0^{t_0}\langle \partial_t u,u\rangle \,dt=\int_0^{t_0}\partial_t \frac{\|u\|_{L^2(\Omega)}^2}{2}\,dt\\
            &=\frac{\|u(t_0)\|_{L^2(\Omega)}^2-\|u(0)\|_{L^2(\Omega)}^2}{2}=\frac{\|u(t_0)\|_{L^2(\Omega)}^2}{2},
        \end{split}
    \end{equation}
    \begin{equation}
    \label{eq: viscosity term}
        \begin{split}
            &\int_0^T \langle (-\Delta)^{s/2} \partial_t u,(-\Delta)^{s/2}\psi\rangle_{L^2(\R^n)}\,dt\\
            &=-\int_0^{t_0}\langle (-\Delta)^{s/2}u,(-\Delta)^{s/2}\partial_t \psi\rangle_{L^2(\R^n)}\,dt\\
            &\quad +\langle (-\Delta)^{s/2}u(t_0),(-\Delta)^{s/2}\psi(t_0)\rangle_{L^2(\R^n)}-\langle (-\Delta)^{s/2}u(0),(-\Delta)^{s/2}\psi(0)\rangle_{L^2(\R^n)}\\
            &=\int_0^{t_0}\langle (-\Delta)^{s/2}u,(-\Delta)^{s/2}\varphi\rangle_{L^2(\R^n)}\,dt=-\int_0^{t_0}\langle (-\Delta)^{s/2}\partial_t\varphi,(-\Delta)^{s/2}\varphi\rangle_{L^2(\R^n)}\,dt\\
            &=-\int_0^{t_0}\partial_t \frac{\|(-\Delta)^{s/2}\varphi\|_{L^2(\R^n)}^2}{2}\,dt\\
            &=\frac{\|(-\Delta)^{s/2}\varphi(0)\|_{L^2(\R^n)}^2-\|(-\Delta)^{s/2}\varphi(t_0)\|_{L^2(\R^n)}^2}{2}=\frac{\|(-\Delta)^{s/2}\varphi(0)\|_{L^2(\R^n)}^2}{2}
        \end{split}
    \end{equation}
    and
    \begin{equation}
    \label{eq: spatial term}
        \begin{split}
            &\int_0^T \langle (-\Delta)^{s/2} u,(-\Delta)^{s/2}\psi\rangle_{L^2(\R^n)}\,dt=-\int_0^{t_0}\langle (-\Delta)^{s/2}\partial_t\varphi,(-\Delta)^{s/2}\psi\rangle_{L^2(\R^n)}\,dt\\
            &=\int_0^{t_0}\langle (-\Delta)^{s/2}\varphi,(-\Delta)^{s/2}\partial_t \psi\rangle_{L^2(\R^n)} \,dt-\langle (-\Delta)^{s/2}\varphi(t_0),(-\Delta)^{s/2}\psi(t_0)\rangle_{L^2(\R^n)}\\
            &\quad +\langle (-\Delta)^{s/2}\varphi(0),(-\Delta)^{s/2}\psi(0)\rangle_{L^2(\R^n)}\\
            &=-\int_0^{t_0}\| (-\Delta)^{s/2}\varphi\|^2_{L^2(\R^n)}\,dt +\langle (-\Delta)^{s/2}\varphi(0),(-\Delta)^{s/2}\psi(0)\rangle_{L^2(\R^n)}\\
            &=-\int_0^{t_0}\| (-\Delta)^{s/2}\varphi\|^2_{L^2(\R^n)}\,dt +\int_0^{t_0}\langle (-\Delta)^{s/2}\varphi(0),(-\Delta)^{s/2}\varphi\rangle_{L^2(\R^n)}\,dt.
        \end{split}
    \end{equation}
    Inserting \eqref{eq: highest order term}, \eqref{eq: viscosity term} and \eqref{eq: spatial term} into \eqref{eq: first identity for uniqueness}, we deduce the estimate
    \begin{equation}
    \label{eq: first estimate uniqueness}
        \begin{split}
            &\|u(t_0)\|_{L^2(\Omega)}^2+b\|(-\Delta)^{s/2}\varphi(0)\|_{L^2(\R^n)}^2\\
            &=2c\int_0^{t_0}\| (-\Delta)^{s/2}\varphi\|^2_{L^2(\R^n)}\,dt -2c\int_0^{t_0}\langle (-\Delta)^{s/2}\varphi(0),(-\Delta)^{s/2}\varphi\rangle_{L^2(\R^n)}\,dt\\
            &\leq C\int_0^{t_0}\| (-\Delta)^{s/2}\varphi\|^2_{L^2(\R^n)}\,dt +\eta \|(-\Delta)^{s/2}\varphi(0)\|_{L^2(\R^n)}^2
        \end{split}
    \end{equation}
    for any $\eta>0$, where the constant $C>0$ only depends on $c$, $T>0$ and $\eta$. Choosing $\eta\leq b/2$, we can absorb the last term on the left hand side and arrive at the bound
    \begin{equation}
    \label{eq: second estimate uniqueness}
        \|u(t_0)\|_{L^2(\Omega)}^2+b\|(-\Delta)^{s/2}\varphi(0)\|_{L^2(\R^n)}^2\leq C\int_0^{t_0}\| (-\Delta)^{s/2}\varphi\|^2_{L^2(\R^n)}\,dt
    \end{equation}
    for some $C>0$ depending on $T,b,c$. Next, let us introduce the function
    \begin{equation}
    \label{eq: def of Phi}
        \Phi(t)\vcentcolon =\int_0^t u(\tau)\,d\tau.
    \end{equation}
    Then, using $\Phi(t_0)=\varphi(0)$ and $\varphi(\tau)=\Phi(t_0)-\Phi(\tau)$ for $0\leq \tau\leq t_0$, we get
    \begin{equation}
    \label{eq: final absorbing uniqueness}
        \begin{split}
            &\|u(t_0)\|_{L^2(\Omega)}^2+b\|(-\Delta)^{s/2}\Phi(t_0)\|_{L^2(\R^n)}^2\\
            &\leq C\int_0^{t_0}\|(-\Delta)^{s/2}(\Phi(t_0)-\Phi(t))\|_{L^2(\R^n)}^2\,dt\\
            &\leq C\int_0^{t_0}\| (-\Delta)^{s/2}\Phi\|^2_{L^2(\R^n)}\,dt +CT_0\|(-\Delta)^{s/2}\Phi(t_0)\|_{L^2(\R^n)}^2
        \end{split}
    \end{equation}
    for all $0<t_0\leq T_0$ and some $C>0$ depending only on $T,b,c$, where $0<T_0\leq T$ is some fixed time. If we choose $T_0>0$ so that $CT_0\leq b/2$, then we may absorb the last term on the left hand side and obtain
    \[
        \|u(t_0)\|_{L^2(\Omega)}^2+b\|(-\Delta)^{s/2}\Phi(t_0)\|_{L^2(\R^n)}^2\leq C\int_0^{t_0}\| (-\Delta)^{s/2}\Phi\|^2_{L^2(\R^n)}\,dt.
    \]
    Thus, from Gronwall's inequality it follows that $u=\Phi=0$ on $[0,T_0]$. As the time horizon $T_0>0$ only depends on $T,b,c$, we get after finitely many iterations that $u=0$ on $[0,T]$.
    
    Finally, let us consider the problem of uniqueness for $\alpha\neq 0$ and $q\neq 0$.
    \begin{enumerate}[(a)]
        \item\label{nonzero alpha} The term $\alpha\partial_t^2$ does not cause any problems as we may calculate
        \[
            \begin{split}
                \int_{\Omega_T}(\partial_t^2 u)\psi\,dxdt&=\left.\langle \partial_t u,\psi\rangle_{L^2(\Omega)}\right|_{t=0}^{t=t_0}-\int_{\Omega_{t_0}}(\partial_t u)\partial_t\psi\,dxdt\\
        &=\int_{\Omega_{t_0}}(\partial_t u)\varphi\,dxdt=\left.\langle u,\varphi\rangle_{L^2(\Omega)}\right|_{t=0}^{t=t_0}-\int_{\Omega_{t_0}}u\partial_t\varphi\,dxdt\\
                &=\int_0^{t_0}\|u\|_{L^2(\Omega)}^2\,dt.
            \end{split}
        \]
        Due to the presence of the term $\|u(t_0)\|_{L^2(\Omega)}^2$ on the left hand side of the estimate \eqref{eq: first estimate uniqueness}, the term $\alpha\int_{\Omega_T}(\partial_t^2 u)\psi\,dxdt$ can be put on the right hand side of \eqref{eq: first estimate uniqueness} and does not cause any problems in the subsequent Gronwall argument.
        \item On the other hand for the potential term, we use Lemma \ref{lemma: potential} and the representation $\psi(t)=\int_t^{t_0}\varphi(\tau)\,d\tau$, $0\leq t\leq t_0$, to estimate
        \begin{equation}
        \label{eq: first estimate potential for uniqueness}
        \begin{split}
            &\left|\int_{\Omega_T}q u\psi\,dxdt\right|\leq C\int_0^{t_0}\|q\|_{L^p(\Omega)}\|u\|_{L^2(\Omega)}\|(-\Delta)^{s/2}\psi\|_{L^2(\R^n)}\\
            &\leq C\left(\int_0^{t_0}\|q\|_{L^{p}(\Omega)}^2 \|u\|_{L^2(\Omega)}^2\,dt+\int_0^{t_0}\left\|\int_t^{t_0}(-\Delta)^{s/2}\varphi(\tau)\,d\tau\right\|_{L^2(\R^n)}^2\,dt\right).
        \end{split}
        \end{equation}
        Furthermore, by Jensen's inequality we have
        \[
        \begin{split}
            \left\|\int_t^{t_0}(-\Delta)^{s/2}\varphi(\tau)\,d\tau\right\|_{L^2(\R^n)}^2&\leq \left(\int_t^{t_0}\|(-\Delta)^{s/2}\varphi\|_{L^2(\R^n)}\,d\tau\right)^2\\
            &\leq (t_0-t)\int_t^{t_0}\|(-\Delta)^{s/2}\varphi(\tau)\|_{L^2(\R^n)}^2\,d\tau\\
            &\leq T\int_0^{t_0}\|(-\Delta)^{s/2}\varphi(\tau)\|_{L^2(\R^n)}^2\,d\tau.
        \end{split}
        \]
        Combining this with \eqref{eq: first estimate potential for uniqueness}, we deduce
        \[
        \begin{split}
            &\left|\int_{\Omega_T}q u\psi\,dxdt\right|\leq C\int_0^{t_0}\|q\|_{L^p(\Omega)}\|u\|_{L^2(\Omega)}\|(-\Delta)^{s/2}\psi\|_{L^2(\R^n)}\\
            &\leq C\left(\int_0^{t_0}\|q\|_{L^{p}(\Omega)}^2 \|u\|_{L^2(\Omega)}^2\,dt+\int_0^{t_0}\|(-\Delta)^{s/2}\varphi\|_{L^2(\R^n)}^2\,dt\right).
        \end{split}
        \]
        By the very same argument as in \ref{nonzero alpha}, the term $\int_{\Omega_T}q u\psi\,dxdt$ does not lead to any difficulties in the above Gronwall argument. 
    \end{enumerate}
    This finishes the proof of Theorem \ref{wellposednessv}.
    \end{proof}

    Next, we want to discuss the solvability of the problem 
    \begin{equation}
		    \label{linear MGT with ext cord}
		\begin{cases}
			(\partial_t^3+\alpha\partial_t^2+b(-\Delta)^s\partial_t+c(-\Delta)^s+q)u=F\ &  {\rm in}\ \Omega_T,\\
			u=\varphi\ & {\rm in}\ (\Omega_e)_T, \\
			u(0)=u_0, \partial_t u(0)=u_1, \partial_t^2 u(0)=u_2\ & {\rm in}\ \Omega,
		\end{cases}
		\end{equation} 
        where the initial conditions
        \begin{equation}
        \label{eq: nonhom intial}
            (u_0,u_1,u_2)\in H^s(\R^n)\times H^s(\R^n)\times L^2(\R^n)
        \end{equation}
        and the exterior condition 
        \begin{equation}
        \label{eq: ext cond}
            \varphi\in W^{1,\infty}(0,T;H^{2s}(\R^n))\cap W^{2,\infty}(0,T;L^2(\R^n))\cap H^{3}(0,T;L^2(\R^n))
        \end{equation}
        need to satisfy a suitable compatibility condition, namely,
        \begin{equation}
        \label{eq: compatibility}
            \varphi(0)=u_0,\,\partial_t\varphi(0)=u_1,\,\partial_t^2 \varphi(0)=u_2\text{ a.e. in }\Omega_e.
        \end{equation}
        In the remainder of this article, we use the following notion of solutions for the non-homogeneous nonlocal MGT equation \eqref{linear MGT with ext cord}.
        \begin{definition}
        \label{eq: solutions to non-homogeneous MGT with ext cond}
            Let $\Omega\subset\R^n$ be a bounded Lipschitz domain, $T>0$ and $s>0$. Suppose that $q\in L^{\infty}(0,T;L^p(\Omega))$ is real-valued and $2\leq p\leq \infty$ satisfies \eqref{conp}. Let the initial data $(u_0,u_1,u_2)$ and the exterior condition $\varphi$ obey the conditions \eqref{eq: nonhom intial}, \eqref{eq: ext cond} and \eqref{eq: compatibility} and suppose that $F\in L^2(0,T;\widetilde{L}^2(\Omega))$. Then we say that $u\in W^{1,\infty}(0,T;H^s(\R^n))\cap W^{2,\infty}(0,T;L^2(\R^n))$ is a \emph{(weak) solution} to \eqref{linear MGT with ext cord} if
        \begin{enumerate}[(I)]
            \item\label{initial cond with ext cond} $u(0)=u_0$, $\partial_t u(0)=u_1$ in $\Omega$,
            \item\label{sol attains exterior cond} $u(t)=\varphi(t)$ in $\Omega_e$ for a.e.~$0<t<T$
            \item\label{PDE satisfies by sol with ext cond} and there holds
            \begin{equation}
            \label{weakint with exterior cond}
		      \begin{split}
		          -\int_{\Omega_T}(\partial_t^2 u)\partial_t\psi\,dxdt+\int_0^T\langle L_{\alpha,b,c,q}u,\psi\rangle\,dt=\int_{0}^T\langle F,\psi\rangle \,dt+\int_{\Omega}u_2\psi(0)\,dx
		      \end{split}
		      \end{equation}
              for any $\psi\in C_c^{\infty}(\Omega\times [0,T))$.
        \end{enumerate}
        \end{definition}
        We have the following well-posedness result for non-homogeneous MGT equations.
	\begin{corollary}[Non-homogeneous MGT equations]
    \label{wellu}
		Let $\Omega\subset\R^n$ be a bounded Lipschitz domain, $T>0$ and $s>0$. Suppose that $q\in L^{\infty}(0,T;L^p(\Omega))$ is real-valued and $2\leq p\leq \infty$ satisfies \eqref{conp}. Then for each initial datum $(u_0,u_1,u_2)$ and exterior condition $\varphi$, having properties \eqref{eq: nonhom intial}, \eqref{eq: ext cond} and \eqref{eq: compatibility}, and the source term $F\in L^2(0,T;\widetilde{L}^2(\Omega))$,
        there exists a unique solution 
        \begin{equation}
        \label{eq: regularity of sol u to linear MGT with ext cond}
            u\in W^{1,\infty}(0,T;H^s(\R^n))\cap W^{2,\infty}(0,T;L^2(\R^n))
        \end{equation}
        to \eqref{linear MGT with ext cord}. Furthermore, it satisfies the energy estimate
		\begin{equation}
         \label{energyu}
		    \begin{split}
		        	 &\|u-\varphi\|_{L^{\infty}(0,T;\widetilde{L}^2(\Omega))}+\|u-\varphi\|_{W^{1,\infty}(0,T;\widetilde{H}^s(\Omega))}\\
		 &\lesssim \|(-\Delta)^{s/2}(u_0-\varphi(0))\|_{L^2(\R^n)}+\|(-\Delta)^{s/2}(u_1-\partial_t\varphi(0))\|_{L^2(\R^n)}\\
         &\quad +\|u_2-\partial_t^2\varphi(0)\|_{L^2(\Omega)}+\|F- (\partial_t^3+ L_{\alpha,b,c,q})\varphi\|_{L^2(\Omega_T)}.
		    \end{split}
		\end{equation}
	\end{corollary}	

    The proof is similar to \cite{PZ1} or \cite{LTZ1} and we only give it for completeness.
    
	\begin{proof}
	    We first assert that $u\in  W^{1,\infty}(0,T;H^s(\R^n))\cap W^{2,\infty}(0,T;L^2(\R^n))$ solves \eqref{linear MGT with ext cord} in the sense of Definition  \ref{eq: solutions to non-homogeneous MGT with ext cond} if and only if $v=u-\varphi \in  W^{1,\infty}(0,T;\widetilde{H}^s(\Omega))\cap W^{2,\infty}(0,T;\widetilde{L}^2(\Omega))$ solves 
        \begin{equation}
		    \label{transformed problem}
		\begin{cases}
			(\partial_t^3+\alpha\partial_t^2+b(-\Delta)^s\partial_t+c(-\Delta)^s+q)v=\widetilde{F}\ &  {\rm in}\ \Omega_T,\\
			v=0 & {\rm in}\ (\Omega_e)_T, \\
			v(0)=v_0, \partial_t v(0)=v_1, \partial_t^2 v(0)=v_2\ & {\rm in}\ \Omega
		\end{cases}
		\end{equation} 
        in the sense of Definition \ref{def: weak sol}, where
        \[
           \widetilde{F}= F-(\partial_t^3+\chi_{\Omega} L_{\alpha,b,c,q})\varphi,\,v_0=u_0-\varphi(0),\,v_1=u_1-\partial_t \varphi(0),\,v_2=u_2-\partial_t^2 \varphi(0).
        \]
        To see this claim we make the following observations:
        \begin{enumerate}[(a)]
            \item The assumptions \eqref{eq: ext cond} on the exterior condition $\varphi$, the Sobolev embedding $H^1(0,T;X)\hookrightarrow L^{\infty}(0,T;X)$ for any Banach space $X$ and the Lipschitz continuity of $\partial\Omega$ guarantee that 
            \begin{equation}
            \label{eq: regularity of sol and sol without ext}
               \begin{split}
                   &u\in  W^{1,\infty}(0,T;H^s(\R^n))\cap W^{2,\infty}(0,T;L^2(\R^n))\text{ satisfies \ref{sol attains exterior cond} }\,\Leftrightarrow\\
                   & v\in W^{1,\infty}(0,T;\widetilde{H}^s(\Omega))\cap W^{2,\infty}(0,T;\widetilde{L}^2(\Omega)).
               \end{split} 
            \end{equation}
            Furthermore, \eqref{eq: ext cond} implies
            \begin{equation}
            \label{eq: regularity of source from phi}
                (\partial_t^3+\chi_{\Omega} L_{\alpha,b,c,q})\varphi\in L^2(0,T;\widetilde{L}^2(\Omega)).
            \end{equation}
            \item By the trace theorem, we know that $u$ satisfies \ref{initial cond with ext cond} if and only if $v(0)=u_0-\varphi(0)=v_0$ and $\partial_t v(0)=u_1-\partial_t \varphi(0)=v_1$.
            \item Finally, we have
            \[
                -\int_{\Omega_T}(\partial_t^2 \varphi)\partial_t\psi\,dxdt=\int_{\Omega_T}(\partial_t^3\varphi)\psi\,dxdt+\int_\Omega \partial_t^2 \varphi(0)\psi(0)\,dx
            \]
            for any $\psi\in C_c^{\infty}(\Omega\times [0,T))$ and so $u$ satisfies \eqref{weakint with exterior cond} if and only if $v$ satisfies \eqref{weakint} with $F\vcentcolon =\widetilde{F}$, $u_2\vcentcolon = v_2$.
        \end{enumerate}
        Now, by the above equivalence it is enough to show that $v$ uniquely exists, but this follows from Theorem \ref{wellposednessv} and hence we have already established the unique solvability of \eqref{linear MGT with ext cord}. The energy estimate \eqref{energyu} is nothing else than the energy estimate for the homogeneous problem \eqref{transformed problem} (see \eqref{energyest}).
	\end{proof}
	
	\subsection{Well-posedness of semilinear nonlocal MGT equations} 
    \label{sec: Well-posedness semilinear problem}
    
	In a next step, we consider the following semilinear nonlocal wave equation
    \begin{equation}
    \label{nonlinearg}
        \begin{cases}
            (\partial_t^3+\alpha\partial_t^2+b(-\Delta)^s\partial_t+c(-\Delta)^s+q)u+f(u)=F\ & {\rm in}\ \Omega_T,\\
		u=0\ & {\rm in}\ (\Omega_e)_T, \\
		u(0)=u_0,\,\partial_tu(0)=u_1,\,\partial_t^2u(0)=u_2\ & {\rm in}\ \Omega,
        \end{cases}
    \end{equation}
    where $(u_0,u_1,u_2)\in \widetilde{H}^s(\Omega)\times \widetilde{H}^s(\Omega)\times\widetilde{L}^2(\Omega)$, $F\in L^2(0,T;\widetilde{L}^2(\Omega))$ and $f(u)$ denotes the \emph{Nemytskii operator} induced by a \emph{Carath\'eodory function} $f\colon \Omega_T\times \R\to\R$, that is 
    \[
        f(u)(x,t)=f(x,t,u(x,t))
    \]
    for a.e.~$(x,t)\in\Omega_T$. Later on, we will impose certain differentiability and growth conditions, but let us recall here that $f$ is Carath\'eodory function, if
		\begin{enumerate}[(a)]
			\item $\tau\mapsto f(x,t,\tau)$ is continuous for a.e.~$(x,t)\in \Omega_T$
			\item and $(x,t)\mapsto f(x,t,\tau)$ is measurable for all $\tau \in\R$.
		\end{enumerate}
    In the sequel, we will make use of the following simple lemma.
    \begin{lemma}[Nemytskii operators]
    \label{lemma: Nemytskii operators}
        Let $\Omega\subset\R^n$ be a bounded domain, $T>0$ and $1\leq p,q<\infty$. Moreover, assume that $f\colon \Omega_T\times \R\to\R$ is a Carath\'eodory function satisfying
			\begin{equation}
            \label{eq: weak estimate on partial f}
            \left|f(x,t,\tau)\right|\leq a+b|\tau|^{\theta}
        \end{equation}
		for all $\tau\in\R$, a.e.~$(x,t)\in\Omega_T$ and some constants $0<\theta<\infty$, $a\geq 0$, $b>0$.
        If $0<\theta<\min(p,q)$, then the Nemytskii operator $u\mapsto f(u)$ is continuous from $L^q(0,T;L^p(\Omega))$ to $L^{q/\theta}(0,T;L^{p/\theta}(\Omega))$. Furthermore, we have the estimate
        \begin{equation}
        \label{eq: estimate Nemytskii operator}
            \|f(u)\|_{L^{q/\theta}(0,T;L^{p/\theta}(\Omega))}\lesssim a+b\|u\|^{\theta}_{L^q(0,T;L^p(\Omega))}
        \end{equation}
        for any $u\in L^q(0,T;L^p(\Omega))$.
    \end{lemma}

    \begin{remark}
    \label{remark: more general caratheodory fcts}
        It is immediate that Lemma \ref{lemma: Nemytskii operators} can also be applied to Carath\'eodory functions $f\colon\Omega_T\times\R\to\R$ satisfying an estimate of the form
        \[
            |f(x,t,\tau)|\lesssim a+b|\tau|^{\Theta}+c|\tau|^{\theta}
        \]
        for some constants $0\leq \Theta<\theta< \min(p,q)$, $a,b\geq 0$ and $c>0$.
    \end{remark}
    
    \begin{proof}
        Throughout the proof, we will use that for all $1\leq r,s<\infty$ one has $u\in L^r(0,T;L^s(\Omega))$ if and only if $u\colon \Omega_T\to \R$ is measurable and $\|u\|_{L^r(0,T;L^s(\Omega))}<\infty$ (see \cite[Proposition 1.8.1]{droniou2001integration}). It is a known fact that the measurability of $u\colon\Omega_T\to\R$ implies the measurability of the composition $f(u)\colon\Omega_T\to\R$. Using \eqref{eq: weak estimate on partial f} and $0<\theta<\min(p,q)$, we easily get
        \begin{equation}
        \label{eq: boundedness Nemytskii}
        \begin{split}
            \|f(u)\|_{L^{q/\theta}(0,T;L^{p/\theta}(\Omega))}&\lesssim \|a+b|u|^{\theta}\|_{L^{q/\theta}(0,T;L^{p/\theta}(\Omega))}\\
            &\lesssim a+b\|u\|^{\theta}_{L^q(0,T;L^p(\Omega))}<\infty.
        \end{split}
        \end{equation}
        Hence, the remark at the beginning of the proof ensures $f(u)\in L^{q/\theta}(0,T;L^{p/\theta}(\Omega))$. Thus, it remains to show the continuity of $f$. For this purpose, assume that $(u_k)_{k\in\N}\subset L^q(0,T;L^p(\Omega))$ converges to some $u\in L^q(0,T;L^p(\Omega))$. Up to extracting a subsequence, we can assume that
        \begin{equation}
        \label{eq: a.e. convergence}
            u_k\to u\text{ for a.e.~}(x,t)\in\Omega_T
        \end{equation}
        as $k\to\infty$. Next, we recall that there exists $h\in L^{q}(0,T;L^p(\Omega))$ such that
        \begin{equation}
        \label{eq: uniform bound}
            |u_k(x,t)|\leq h(x,t)
        \end{equation}
        for all $k\in\N$ and a.e.~$(x,t)\in\Omega_T$. This follows from a slight modification of the arguments for \cite[Theorem 4.9]{Brezis}. Hence, \eqref{eq: weak estimate on partial f}, \eqref{eq: a.e. convergence}, \eqref{eq: uniform bound} and the continuity of $f$ in $\tau$ ensure that
        \[
        \begin{split}
            f(u_k)&\to f(u)\text{ a.e.~in }\Omega_T\text{ as }k\to\infty,\\
            |f(u_k)|&\leq a+|h|^{\theta}\in L^{q/\theta}(0,T;L^{p/\theta}(\Omega)).
        \end{split}
        \]
        Therefore, we can apply Lebesgue's dominated convergence theorem in the mixed Lebesgue space $L^{q/\theta}(0,T;L^{p/\theta}(\Omega))$ to conclude that
        \[
            f(u_k)\to f(u)\text{ in }L^{q/\theta}(0,T;L^{p/\theta}(\Omega))
        \]
        as $k\to\infty$. By standard arguments this also holds for the full sequence $(u_k)_{k\in\N}$ and thus the Nemytskii operator $u\mapsto f(u)$ is continuous. This finishes the proof.
    \end{proof}
    
    We have the following well-posedness result for the semilinear problem \eqref{nonlinearg}. 

    \begin{theorem}[Homogeneous, semilinear MGT equations]
    \label{wellnonlinear}
        Let $\Omega\subset\R^n$ be a bounded Lipschitz domain, $T>0$ and $s>0$. Suppose that $q\in L^{\infty}(0,T;L^p(\Omega))$ is real-valued with $2\leq p\leq \infty$ satisfying \eqref{conp} and let the nonlinearity $f\colon \Omega_T\times \R\to\R$ obey the following conditions:
        \begin{enumerate}[(i)]
            \item\label{prop f} $f$ is a Carath\'eodory function with $f(0)=0$.
            \item\label{prop f growth} $\partial_\tau f$ is a Carath\'eodory function that obeys the growth condition
			\begin{equation}
				\label{eq: bound on derivative}
				\left|\partial_\tau f(x,t,\tau)\right|\lesssim |\tau|^{\gamma}+|\tau|^r
			\end{equation}
			for a.e.~$(x,t)\in\Omega_T$, all $\tau\in\R$ and some exponents $0<\gamma\leq r<\infty$, where $r$ satisfies \eqref{conr}.
        \end{enumerate}
        Moreover, let $C_0>0$ stand for the constant appearing in the energy estimate \eqref{energyest}.
        Under the above assumptions, we can find a constant $\delta_0>0$ such that for all $0<\delta\leq \delta_0$ and all $(u_0,u_1,u_2,F)\in \widetilde{H}^s(\Omega)\times \widetilde{H}^s(\Omega)\times\widetilde{L}^2(\Omega)\times L^2(0,T;\widetilde{L}^2(\Omega))$ satisfying the \emph{smallness condition}
        \begin{equation}
        \label{eq: smallness on source F}
           \|(-\Delta)^{s/2}u_0\|_{L^2(\R^n)}+\|(-\Delta)^{s/2}u_1\|_{L^2(\R^n)}+\|u_2\|_{L^2(\Omega)}+\|F\|_{L^2(\Omega_T)}\leq \frac{\delta}{2C_0},
        \end{equation}
        there exists a unique (small) solution 
        \[
           u_\delta\in W^{1,\infty}(0,T;\widetilde{H}^s(\Omega))\cap W^{2,\infty}(0,T;\widetilde{L}^2(\Omega))\cap H^3(0,T;H^{-s}(\Omega))
        \]
        of the problem \eqref{nonlinearg}. This unique solution fulfills the energy estimate
        \begin{equation}
        \label{eq: energy estimate nonlinear eq well-posedness}
            \begin{split}
		 &\|\partial_t^2 u_\delta\|_{L^{\infty}(0,T;L^2(\Omega))}+\|u_\delta\|_{W^{1,\infty}(0,T;\widetilde H^s(\Omega))} \\
		 &\lesssim  \|u_0\|_{\widetilde{H}^s(\Omega)}+\|u_1\|_{\widetilde{H}^s(\Omega)}+\|u_2\|_{L^2(\Omega)}+\|F\|_{L^2(\Omega_T)}\\
         &\quad+\|u_\delta\|^{\gamma+1}_{L^\infty(0,T;\widetilde{H}^s(\Omega))}+\|u_\delta\|^{r+1}_{L^\infty(0,T;\widetilde{H}^s(\Omega))}. 
        \end{split}
        \end{equation}
        Finally, if $0<\delta_1\leq \delta_2\leq \delta_0$ and $(u_0,u_1,u_2,F)\in \widetilde{H}^s(\Omega)\times \widetilde{H}^s(\Omega)\times\widetilde{L}^2(\Omega)\times L^2(0,T;\widetilde{L}^2(\Omega))$ satisfy \eqref{eq: smallness on source F} with $\delta=\delta_1$, then $u_{\delta_1}=u_{\delta_2}$.
    \end{theorem}

    \begin{remark}
        \label{rem: smallness in semilinear}
        We emphasize that the constant $\delta_0>0$ depends on the constant appearing in the estimate \eqref{eq: bound on derivative}, as well as on the parameters $\alpha,b,c$, on $n,\gamma,r,s,\Omega$ through the fractional Sobolev and Poincar\'e constants, and on the finite time horizon $T$. In contrast to well-posedness theory of the linear MGT equation, we also make use of the embedding $L^{r_1}(\Omega)\hookrightarrow L^{r_2}(\Omega)$ for $r_1\geq r_2$ (see, e.g., \eqref{eq: hoelder semilinear}). Consequently, the boundedness of $\Omega$ is essential and $\delta_0$ depends through this estimate on $|\Omega|$ as well. In particular, we refer the reader to Remark~\ref{rem: dependence of energy constant} and the condition \eqref{eq: choice of delta 0 st contraction}, where $C_1>0$ is the constant in \eqref{L2} and $C_2>0$ the constant in \eqref{eq: estimate 2 for contraction prop final}.
    \end{remark}

    \begin{proof}
    We divide the proof into three steps. In Step 1 we collect some basic estimates, which are used in Step 2 to demonstrate the existence of a solution. Eventually in Step 3 we show that the constructed solution is unique.\\
        
        \noindent \textit{Step 1.} In this step we establish some a priori estimates (cf.,~e.g., \cite[Theorem 3.8]{PZ1}) that are used in Step 2 to prove the contraction property of the solution map. First of all, for any $\varphi\in L^{\infty}(0,T;H^s(\R^n))$, we have
        \begin{equation}
	|f(x,t,\tau)|=\bigg|\int_0^{\tau}\partial_\tau f(x,t,s)ds\bigg|\lesssim |\tau|^{\gamma+1}+|\tau|^{r+1}
	\end{equation}
	which yields 
	\begin{equation}  
	\|f(\var)(t)\|_{L^2(\Omega)}\lesssim \|\var(t)\|^{\gamma+1}_{L^{2(\gamma+1)}(\Omega)}+\|\var(t)\|^{r+1}_{L^{2(r+1)}(\Omega)}. 
	\end{equation}  
	In the range $2s>n$, the Sobolev embedding $H^s(\R^n)\hookrightarrow L^\infty(\R^n)$ guarantees the estimate
	\begin{equation}  
	\|\var(t)\|^{\gamma+1}_{L^{2(\gamma+1)}(\Omega)}+\|\var(t)\|^{r+1}_{L^{2(r+1)}(\Omega)}\lesssim \|\var(t)\|^{\gamma+1}_{H^s(\R^n)}+\|\var(t)\|^{r+1}_{H^s(\R^n)}.  
	\end{equation}  
	In the critical case $n=2s$, we use $2(r+1)\ge 2(\gamma+1)> 2$ and Proposition \ref{critical} to get
	\begin{equation}  
	\|\var(t)\|^{\gamma+1}_{L^{2(\gamma+1)}(\Omega)}\lesssim\|(-\Delta)^{s/2}\var(t)\|^\gamma_{L^2(\R^n)}\|\var(t)\|_{L^2(\R^n)}
	\end{equation}  
    and
	\begin{equation}  
	\|\var(t)\|^{r+1}_{L^{2(r+1)}(\Omega)}\lesssim\|(-\Delta)^{s/2}\var(t)\|^r_{L^2(\R^n)}\|\var(t)\|_{L^2(\R^n)}.
	\end{equation}  
	In the range $2s<n$, we use $1<\gamma+1\leq r+1\leq \frac{n}{n-2s}$ and Sobolev's inequality to obtain
	\begin{equation}  
    \label{eq: hoelder semilinear}
	\|\var(t)\|_{L^{2(\gamma+1)}(\Omega)}\lesssim\|\var(t)\|_{L^{2(r+1)}(\Omega)}\lesssim \|\var(t)\|_{L^{\frac{2n}{n-2s}}(\Omega)}\lesssim \|(-\Delta)^{s/2}\var(t)\|_{L^2(\R^n)}.  
	\end{equation}  
    Therefore, in all cases, we have
	\begin{equation}  
    \label{L2}
	\|f(\var)\|_{L^2(\Omega_T)}\lesssim\|\var\|^{\gamma+1}_{L^\infty(0,T;H^s(\R^n))}+\|\var\|^{r+1}_{L^\infty(0,T;H^s(\R^n))}.
	\end{equation}  

        \noindent \textit{Step 2.} Next, we define
        \begin{equation}
        \label{eq: solution space X delta}
            X_\delta\vcentcolon = \{u\in W^{1,\infty}(0,T;\widetilde{H}^s(\Omega))\cap W^{2,\infty}(0,T;\widetilde{L}^2(\Omega))\,;\,\|u\|_X\leq \delta\},
        \end{equation}
        with
        \begin{equation}
        \label{eq: norm X delta}
            \|u\|_X=\|u\|_{W^{1,\infty}(0,T;\widetilde{H}^s(\Omega))}+\|\partial_t^2 u\|_{L^{\infty}(0,T;\widetilde{L}^2(\Omega))},
        \end{equation}
        for any $0<\delta\leq \delta_0$, where $\delta_0>0$ is a given constant that will be fixed later. By standard arguments, one sees that 
        \[
            X\vcentcolon =W^{1,\infty}(0,T;\widetilde{H}^s(\Omega))\cap W^{2,\infty}(0,T;\widetilde{L}^2(\Omega))
        \]
        endowed with $\|\cdot\|_X$ is a Banach space (see \cite[Lemma 2.3.1]{Interpolation-spaces-Bergh-Lofstrom}) and hence $X_\delta$, $0<\delta\leq \delta_0$, are complete metric spaces.

        Next, from \eqref{L2} and Theorem \ref{wellposednessv}, we deduce that for all
        \[
        (v,u_0,u_1,u_2,F)\in X_\delta \times \widetilde{H}^s(\Omega)\times\widetilde{H}^s(\Omega)\times \widetilde{L}^2(\Omega)\times L^2(0,T;\widetilde{L}^2(\Omega)),
        \]
        there exists a unique solution $u_\delta\in X$ of
        \begin{equation}
       \label{eq: PDE for solution map nonlinear}
        \begin{cases}
            (\partial_t^3+\alpha\partial_t^2+b(-\Delta)^s\partial_t+c(-\Delta)^s+q)u=F-f(v)\ & {\rm in}\ \Omega_T,\\
		u=0\ & {\rm in}\ (\Omega_e)_T, \\
		u(0)=u_0,\,\partial_tu(0)=u_1,\,\partial_t^2u(0)=u_2\ & {\rm in}\ \Omega.
        \end{cases}
    \end{equation}
        Furthermore, $u_\delta$ satisfies the energy estimate
        \begin{equation}
        \label{eq: energy estimate for nonlinear problem}
        \begin{split}
             \|u_\delta\|_X&\leq C_0\big( \|(u_0,u_1,u_2,F)\|_{\pi}+\|f(v)\|_{L^2(\Omega_T)}\big),
        \end{split}
        \end{equation}
        where $\|\cdot\|_{\pi}$ is the natural norm on the product space $\widetilde{H}^s(\Omega)\times\widetilde{H}^s(\Omega)\times \widetilde{L}^2(\Omega)\times L^2(0,T;\widetilde{L}^2(\Omega))$.
        Thus, we can introduce the solution map $S_\delta \colon X_\delta\to X$ that sends each $v\in X_\delta$ to its unique solution $u_\delta\in X$ of \eqref{eq: PDE for solution map nonlinear}, which is provided by Theorem~\ref{wellposednessv}. Now, we show that $S_\delta (X_\delta)\subset X_\delta$, when $0<\delta\leq \delta_0$ and $\delta_0$ is sufficiently small. Noticing that \eqref{eq: smallness on source F} implies 
        \[
            \|(u_0,u_1,u_2,F)\|_{\pi}\leq \frac{\delta}{2C_0},
        \]
         we deduce from \eqref{eq: energy estimate for nonlinear problem}, \eqref{L2} and Poincar\'e's inequality that for any $v\in X_\delta$ we have
        \begin{equation}
        \label{eq: first appearance C1}
            \begin{split}
                \|S_\delta(v)\|_X&\leq \delta/2+C_1\big(\|v\|^{\gamma+1}_{L^\infty(0,T;\widetilde{H}^s(\Omega))}+\|v\|^{r+1}_{L^\infty(0,T;\widetilde{H}^s(\Omega))}\big)\\
                &\leq \delta/2+C_1(\delta^\gamma+\delta^r)\delta\\
                &\leq \delta/2+C_1(\delta_0^\gamma+\delta_0^r)\delta
            \end{split}
        \end{equation}
        for some $C_1$ independent of $\delta$ and $v$. Hence, choosing $\delta_0>0$ so that 
        \[
            \delta_0^\gamma+\delta_0^r\leq 1/(2C_1),
        \]
        we deduce that $\|S_\delta(v)\|_X\leq \delta$ and $S_\delta$ map $X_\delta$ to itself. In the next step, we demonstrate that the maps $S_\delta$, $0<\delta\leq \delta_0$, are strict contractions for all sufficiently small $\delta_0>0$. To this end, let us choose $v^j\in X_\delta$, $j=1,2$, for some fixed $0<\delta\leq \delta_0$ and set $u^j_\delta\vcentcolon = S_\delta(v^j)\in X_\delta$. Then $u\vcentcolon = u^1_\delta-u^2_\delta\in X$ solves
        \begin{equation}
            \label{eq: PDE for contraction of solution map nonlinear}
        \begin{cases}
            (\partial_t^3+\alpha\partial_t^2+b(-\Delta)^s\partial_t+c(-\Delta)^s+q)u=-(f(v^1)-f(v^2))\ & {\rm in}\ \Omega_T,\\
		u=0\ & {\rm in}\ (\Omega_e)_T, \\
		u(0)=0,\,\partial_tu(0)=0,\,\partial_t^2u(0)=0\ & {\rm in}\ \Omega
        \end{cases}
        \end{equation}
        and Theorem \ref{wellposednessv} ensures that
        \begin{equation}
        \label{eq: first estimate for contraction}
            \|u^1_\delta-u^2_\delta\|_X\leq C_0\|f(v^1)-f(v^2)\|_{L^2(\Omega_T)}.
        \end{equation}
        We notice that 
        \begin{equation}
        \label{eq: estimate of f for difference}
            \begin{split}
                |f(x,t,s_1)-f(x,t,s_2)|&=\left|\int_{s_1}^{s_2}\partial_\tau f(x,t,\tau)\,d\tau\right|\\
                &\leq C(|s_1|^\gamma+|s_2|^\gamma +|s_1|^r+|s_2|^r)|s_1-s_2|
            \end{split}
        \end{equation}
        for a.e.~$(x,t)\in \Omega_T$ and all $s_1,s_2\in\R$ implies
        \begin{equation}
        \label{eq: difference of nonlinearities}
        \begin{split}
            \|f(v^1)-f(v^2)\|_{L^{\infty}(0,T;L^2(\Omega))}&\leq C\|(|v^1|^\gamma+|v^2|^\gamma)|v^1-v^2|\|_{L^{\infty}(0,T;L^2(\Omega))}\\
            &\quad +C\|(|v^1|^r+|v^2|^r)|v^1-v^2|\|_{L^{\infty}(0,T;L^2(\Omega))}.
        \end{split}
        \end{equation}
        Next, we recall the inequality
        \begin{equation}
        \label{eq: estimate 2 for contraction prop}
        \begin{split}
            &\|(|v^1(t)|^\rho+|v^2(t)|^\rho)|v^1(t)-v^2(t)|\|_{L^2(\Omega)}\\
            &\leq C\big(\|v^1(t)\|_{\widetilde{H}^s(\Omega)}^{\rho}+\|v^2(t)\|_{\widetilde{H}^s(\Omega)}^{\rho}\big)\|v^1(t)-v^2(t)\|_{\widetilde{H}^s(\Omega)},
        \end{split}
        \end{equation}
        which has been established in \cite[eq.~(3.26)]{LTZ2} and holds for all exponents $0<\rho<\infty$ satisfying the constraints in \eqref{conr}. Thus, inserting \eqref{eq: estimate 2 for contraction prop} with $\rho=\gamma, r$ into \eqref{eq: difference of nonlinearities} and using $v^j\in X_\delta$ with $0<\delta\leq \delta_0$, we get
        \begin{equation}
        \label{eq: estimate 2 for contraction prop final}
        \begin{split}
            &\|f(v^1)-f(v^2)\|_{L^{\infty}(0,T;L^2(\Omega))}\\
            &\leq C\big(\|v^1\|_{L^{\infty}(0,T;\widetilde{H}^s(\Omega))}^{\gamma}+\|v^2\|_{L^{\infty}(0,T;\widetilde{H}^s(\Omega))}^{\gamma}\big)\|v^1-v^2\|_{L^{\infty}(0,T;\widetilde{H}^s(\Omega))}\\
            &\quad +C\big(\|v^1\|_{L^{\infty}(0,T;\widetilde{H}^s(\Omega))}^{r}+\|v^2\|_{L^{\infty}(0,T;\widetilde{H}^s(\Omega))}^{r}\big)\|v^1-v^2\|_{L^{\infty}(0,T;\widetilde{H}^s(\Omega))}\\
            &\leq C_2(\delta_0^\gamma+\delta_0^{r})\|v^1-v^2\|_X.
        \end{split}
        \end{equation}
        Now, we choose $\delta_0>0$ so that 
        \begin{equation}
        \label{eq: choice of delta 0 st contraction}
            \delta_0^\gamma+\delta_0^r\leq \min(1/(2C_1),1/(2 C_0C_2 T^{1/2})).
        \end{equation}
        Hence, \eqref{eq: first estimate for contraction}, \eqref{eq: estimate 2 for contraction prop final} and \eqref{eq: choice of delta 0 st contraction} guarantee that $S_\delta$ is a strict contraction for all $0<\delta\leq \delta_0$. Therefore, we may invoke Banach's fixed point theorem to deduce the existence of a unique solution $u_\delta\in X_\delta$ of \eqref{eq: PDE for solution map nonlinear}. The estimate \eqref{eq: energy estimate nonlinear eq well-posedness} is a direct consequence of \eqref{eq: energy estimate for nonlinear problem}, \eqref{L2} and Poincar\'e's inequality.\\

        \noindent \textit{Step 3.} Finally, let us prove that $u_{\delta_1}=u_{\delta_2}$ for all $0<\delta_1\leq \delta_2\leq \delta_0$, whenever $(u_0,u_1,u_2,F)\in \widetilde{H}^s(\Omega)\times \widetilde{H}^s(\Omega)\times\widetilde{L}^2(\Omega)\times L^2(0,T;\widetilde{L}^2(\Omega))$ satisfies \eqref{eq: smallness on source F} with $\delta=\delta_1$. Here, $u_{\delta_j}$ stands for the unique fixed point of $S_{\delta_j}$, which is defined as a map from $X_{\delta_j}$ to itself (cf.~Step 2). First, observe that we have $X_{\delta_1}\subset X_{\delta_2}$ for all $\delta_1\leq \delta_2$ and by Theorem \ref{wellposednessv} we have
        $S_{\delta_2}|_{X_{\delta_1}}=S_{\delta_1}$. Hence, $u_{\delta_1}\in X_{\delta_2}$ is a fixed point of $S_{\delta_2}$ as well. Therefore, by the uniqueness of the fixed point $u_{\delta_2}\in X_{\delta_2}$ of $S_{\delta_2}$ it follows that $u_{\delta_1}=u_{\delta_2}$.
    \end{proof}

    Using a similar argument as in Section \ref{subsec: linear MGT}, we can derive the unique solvability of the following non-homogeneous, semilinear nonlocal MGT equation
    \begin{equation}
    \label{nonlinearg with ext cond}
        \begin{cases}
            (\partial_t^3+\alpha\partial_t^2+b(-\Delta)^s\partial_t+c(-\Delta)^s+q)u+f(u)=F & \text{ in }\Omega_T,\\
		u=\varphi & \text{ in }(\Omega_e)_T, \\
		u(0)=u_0,\,\partial_tu(0)=u_1,\,\partial_t^2u(0)=u_2 & \text{ in }  \Omega.
        \end{cases}
    \end{equation}
    More precisely, repeating the arguments in the proof of Corollary \ref{wellu} shows the following well-posedness result, which for simplicity only deals with exterior conditions supported in $(\Omega_e)_T$.
    \begin{proposition}[Non-homogeneous, semilinear MGT equations]
    \label{prop: nonlinearg with ext cond}
        Let $\Omega\subset\R^n$ be a bounded Lipschitz domain, $T>0$ and $s>0$. For all potentials $q\colon\Omega_T\to\R$ and nonlinearities $f\colon \Omega_T\times \R\to\R$ obeying the conditions of Theorem \ref{wellnonlinear}, we can find a constant $\delta_0>0$ with the following property: If we have given
        \[
            (u_0,u_1,u_2,F)\in H^s(\R^n)\times H^s(\R^n)\times L^2(\R^n)\times L^2(0,T;\widetilde{L}^2(\Omega))
        \]
        and
        \[
            \varphi\in W^{1,\infty}(0,T;\widetilde{H}^{2s}(\Omega_e))\cap W^{2,\infty}(0,T;L^2(\R^n))\cap H^{3}(0,T;L^2(\R^n))
        \]
        satisfying \eqref{eq: compatibility} and the smallness condition
        \begin{equation}
        \label{eq: smallness on source F with ext cond}
        \begin{split}
           &\|u_0-\varphi(0)\|_{\widetilde{H}^s(\Omega)}+\|u_1-\partial_t\varphi(0)\|_{\widetilde{H}^s(\Omega)}+\|u_2-\partial_t^2\varphi(0)\|_{L^2(\Omega)}\\
           &\quad+\|F-b(-\Delta)^s\partial_t\varphi-c(-\Delta)^s\varphi\|_{L^2(\Omega_T)}\leq \frac{\delta}{2C_0},
           \end{split}
        \end{equation}
        for some $0<\delta\leq \delta_0$ and the constant $C_0$ appearing in the energy estimate \eqref{energyest}, then problem \eqref{nonlinearg with ext cond} has a unique (small) solution 
        \[
           u_\delta\in W^{1,\infty}(0,T;H^s(\R^n))\cap W^{2,\infty}(0,T;L^2(\R^n)).
        \]
        In fact, the solution $u_\delta$ is given by $u_\delta=v_\delta+\varphi$, where $v_\delta \in X_\delta$ is the unique solution of \eqref{nonlinearg} with initial conditions $u_0-\varphi(0), u_1-\partial_t\varphi(0),u_2-\partial_t^2\varphi(0)$ and the source term $F-\chi_\Omega (b(-\Delta)^s\partial_t\varphi+c(-\Delta)^s\varphi)$ (see Theorem \ref{wellnonlinear}).
    \end{proposition}

    \subsection{Well-posedness of nonlocal JMGT equations of Westervelt-type}
    \label{subsec: well-posedness westervelt}
    
    In this section, we investigate the existence of solutions to the PDE
    \begin{equation}
        \label{eq: sysmgtWtype well-posedness}
	\begin{cases}
		(\partial_t^3+\alpha \partial_t^2+b(-\Delta)^s\partial_t+c(-\Delta)^s)u=
        F+\mathcal{F}(u,\partial_tu,\partial_t^2u) & {\rm in}\ \Omega_T, \\
		u=\varphi,& {\rm in} \ (\Omega_e)_T,\\
        u(0)=u_0,\,\partial_tu(0)=u_1,\,\partial_t^2u(0)=u_2  & {\rm in}\ \Omega,
	\end{cases}
    \end{equation}
	where $\mathcal{F}$ is one of the two Westervelt-type nonlinearities discussed in Section \ref{sec: intro}, that is,
    \begin{equation}
    \label{eq: nonlinearity F}
        \mathcal{F}(u,\partial_tu,\partial_t^2u)=\partial_t^2(\beta u^2)\text{ or } \mathcal{F}(u,\partial_tu,\partial_t^2u)=\partial_t(\kappa(\partial_tu)^2)
    \end{equation}
   in which we assume that the Westervelt parameters $(\beta,\kappa)$ satisfy
   \begin{equation}
   \label{eq: Westervelt regularity}
        (\beta,\kappa)\in W^{2,\infty}(0,T;L^\infty(\Omega))\times W^{1,\infty}(0,T;L^\infty(\Omega)).
   \end{equation}
   Let us recall that we only consider nonlocal JMGT equations in the supercritical range $2s>n$ and therefore we have at our disposal the so-called Morrey embedding
   \begin{equation}
    \label{eq: Morrey embedding}
        H^s(\R^n)\hookrightarrow C^{s-n/2}_{*}(\R^n)\hookrightarrow L^{\infty}(\R^n),
   \end{equation}
    where $C^{\rho}_*(\R^n)$, $\rho>0$, are the classical Zygmund spaces that coincide with the H\"older spaces $C^{\alpha}(\R^n)$ for $\alpha\notin \Z$ (see, e.g., \cite{Triebel-Theory-of-function-spaces}). 
    
    Then, using \eqref{eq: Westervelt regularity} and \eqref{eq: Morrey embedding}, one can see that for all $v\in X_\delta$ we have
    \begin{equation}
    \label{eq: expansion beta}
    \begin{split}
        \partial_t^2(\beta v^2)&=\partial_t((\partial_t\beta)v^2+2\beta (\partial_t v)v)\\
        &=(\partial_t^2\beta)v^2+4(\partial_t\beta)(\partial_t v)v+2\beta(\partial_t^2 v)v+2\beta(\partial_t v)^2
    \end{split}
    \end{equation}
    for a.e.~$(x,t)\in \Omega_T$. From this expansion, we may also observe that we need to work in the supercritical range to apply our well-posedness theory for MGT equations. In fact, $(\partial_t^2 v)v\in L^2(\Omega_T)$ and $v\in X_\delta$ force us to have $v\in L^2(0,T;L^{\infty}(\Omega))$.  In turn, equation \eqref{eq: expansion beta} guarantees that
	\begin{equation}  
    \label{eq: estimate beta}
    \begin{split}
	\|\partial_t^2(\beta v^2)\|_{L^{2}(\Omega_T)} &\lesssim \|\beta\|_{W^{2,\infty}(0,T;L^\infty(\Omega))}\|v\|^2_{W^{1,\infty}(0,T;L^{\infty}(\Omega))}\\
    &\quad+\|\beta\|_{L^{\infty}(0,T;L^{\infty}(\Omega))}\|\partial_t^2 v\|_{L^{\infty}(0,T;L^2(\Omega))}\|v\|_{L^{\infty}(0,T;L^{\infty}(\Omega))} \\
     &\lesssim \|\beta\|_{W^{2,\infty}(0,T;L^\infty(\Omega))}\|v\|^2_{X}\\
     &\lesssim \delta_0\|\beta\|_{W^{2,\infty}(0,T;L^\infty(\Omega))}\|v\|_X
	\end{split}
    \end{equation}
    for all $v\in X_\delta$, $0<\delta\leq \delta_0$.
    Here, we used \eqref{eq: Morrey embedding} and Poincar\'e's inequality. Similarly, the other Westervelt-type nonlinearity, appearing in the velocity potential formulation of the (nonlocal) JMGT equation, can be bounded as
    \begin{equation}
    \label{eq: estimate kappa}
    \begin{split}
        \|\partial_t(\kappa (\partial_t v)^2)\|_{L^{2}(\Omega_T)}&=\|(\partial_t\kappa)(\partial_t v)^2+2\kappa\partial_t v\partial_t^2 v\|_{L^{2}(\Omega_T)}\\
        &\lesssim \|\kappa\|_{W^{1,\infty}(0,T;L^\infty(\Omega))}\|v\|_{X}^2\\
        &\lesssim \delta_0\|\kappa\|_{W^{1,\infty}(0,T;L^\infty(\Omega))}\|v\|_{X}.
    \end{split}
    \end{equation}	
    Thus, as in the proof of Theorem \ref{wellnonlinear}, for sufficiently small $\delta_0>0$ and any $0<\delta\leq \delta_0$, we have a well-defined map $S_\delta\colon X_\delta\to X_\delta$, which sends every $v\in X_\delta$ to the unique solution $u_\delta\in X_\delta$ of problem \eqref{eq: sysmgtWtype well-posedness} with a frozen nonlinearity, that is, 
    \begin{equation}
        \label{eq: linearized sysmgtWtype well-posedness}
	\begin{cases}
		(\partial_t^3+\alpha \partial_t^2+b(-\Delta)^s\partial_t+c(-\Delta)^s)u=F+\mathcal{F}(v,\partial_t v,\partial_t^2 v) & {\rm in}\ \Omega_T, \\
		u=0,& {\rm in} \ (\Omega_e)_T,\\
        u(0)=u_0,\,\partial_tu(0)=u_1,\,\partial_t^2u(0)=u_2  & {\rm in}\ \Omega.
	\end{cases}
    \end{equation}
    Furthermore, one can check that the solution maps $S_\delta$ are strict contractions as long as $0<\delta\leq \delta_0$ and $\delta_0>0$ is sufficiently small. In fact, this can be seen by continuing as in Step 2 of the proof of Theorem \ref{wellnonlinear} and using the following estimates,
    \[
       \begin{split}
           \|\partial_t^2[\beta(v_1^2-v_2^2)]\|_{L^2(\Omega_T)}&=\|\partial_t^2[\beta(v_1-v_2)(v_1+v_2)]\|_{L^2(\Omega_T)}\\
           &\lesssim \|\beta\|_{W^{2,\infty}(0,T;L^{\infty}(\Omega))}\|v_1+v_2\|_X\|v_1-v_2\|_X\\
           &\lesssim \delta_0\|\beta\|_{W^{2,\infty}(0,T;L^{\infty}(\Omega))}\|v_1-v_2\|_X
       \end{split} 
    \]
    and 
    \[
     \begin{split}
           \|\partial_t[\kappa ((\partial_t v_1)^2-(\partial_t v_2)^2]\|_{L^{2}(\Omega_T)}&=\|\partial_t[\kappa (\partial_t v_1-\partial_t v_2)(\partial_t v_1+\partial_t v_2)]\|_{L^{2}(\Omega_T)}\\
           &\lesssim \|\kappa\|_{W^{1,\infty}(0,T;L^\infty(\Omega))}\|v_1+v_2\|_X\|v_1-v_2\|_X\\
           &\lesssim \delta_0\|\kappa\|_{W^{1,\infty}(0,T;L^\infty(\Omega))}\|v_1-v_2\|_{X}
    \end{split} 
    \]
    for all $v_1,v_2\in X_\delta$.
    Now, Banach's fixed point theorem guarantees the well-posedness of \eqref{eq: sysmgtWtype well-posedness} for $\varphi=0$ and the case of nonzero exterior condition can once again be established in the same manner as Proposition \ref{prop: nonlinearg with ext cond}. In this way, we arrive at the following result.
    
	\begin{theorem}[Non-homogeneous JMGT equations of Westervelt-type]
    \label{well-Wel}
		Let $\Omega\subset\R^n$ be a bounded Lipschitz domain, $T>0$ and $s>0$ with $2s>n$. Suppose that $(\beta,\kappa)$ obey \eqref{eq: Westervelt regularity} and let $\mathcal{F}$ be given by \eqref{eq: nonlinearity F}. Then one can find a constant $\delta_0>0$ so that for all $0<\delta\leq \delta_0$ and all $(u_0,u_1,u_2,F,\varphi)\in H^s(\R^n)\times H^s(\R^n)\times L^2(\R^n)\times L^2(0,T;\widetilde{L}^2(\Omega))\times W^{1,\infty}(0,T;\widetilde{H}^{2s}(\Omega_e))\cap W^{2,\infty}(0,T;L^2(\R^n))\cap H^{3}(0,T;L^2(\R^n))$ satisfying the compatibility condition \eqref{eq: compatibility} and the smallness condition
        \begin{equation}
        \label{eq: smallness on source F with ext cond beta, kappa}
        \begin{split}
           &\|u_0-\varphi(0)\|_{\widetilde{H}^s(\Omega)}+\|u_1-\partial_t\varphi(0)\|_{\widetilde{H}^s(\Omega)}+\|u_2-\partial_t^2\varphi(0)\|_{L^2(\Omega)}\\
           &\quad+\|F-\chi_\Omega (b(-\Delta)^s\partial_t\varphi+c(-\Delta)^s\varphi)\|_{L^2(\Omega_T)}\leq \frac{\delta}{2C_0},
           \end{split}
        \end{equation}
        there exists a unique (small) solution 
        \[
           u_\delta\in W^{1,\infty}(0,T;H^s(\R^n))\cap W^{2,\infty}(0,T;L^2(\R^n))
        \]
        of the problem \eqref{eq: sysmgtWtype well-posedness}. In fact, the solution $u_\delta$ is given by $u_\delta=v_\delta+\varphi$, where $v_\delta \in X_\delta$ is the unique solution of \eqref{eq: sysmgtWtype well-posedness} with initial conditions $u_0-\varphi(0), u_1-\partial_t\varphi(0),u_2-\partial_t^2\varphi(0)$ and the source term $F-\chi_\Omega (b(-\Delta)^s\partial_t\varphi+c(-\Delta)^s\varphi)$. We again have $u_{\delta_1}=u_{\delta_2}$ for $0<\delta_1\leq \delta_2\leq \delta_0$ as long as both solutions are well-defined and there holds an energy estimate similar to \eqref{eq: energy estimate nonlinear eq well-posedness}.
    \end{theorem}

    \begin{remark}
    \label{rem: smallness in westervelt}
        We emphasize that the constants $\delta_0>0$ for the Westervelt-type nonlocal JMGT equations depend on $n,s,\alpha,b,c$, and $\Omega$ (see Remark~\ref{rem: dependence of energy constant}), on the constant in the Morrey embedding \eqref{eq: Morrey embedding}, as well as the norms $\|\beta\|_{W^{2,\infty}(0,T;L^{\infty}(\Omega))}$ and $\|\kappa\|_{W^{1,\infty}(0,T;L^{\infty}(\Omega))}$, respectively.
    \end{remark}
        
	\section{Runge approximation and uniqueness of linear perturbations}
    \label{sec: recovery of lin perturb}

    The purpose of this section is on the one hand to settle the foundations for the nonlinear problems treated in Sections \ref{sec: proof of theorem 1.5}--\ref{sec: recovery beta and kappa} and on the other hand to prove the unique determination of linear perturbations. We will see in the next section that after a first order linearization the inverse problem for the nonlinear PDE \eqref{sysmgt} is reduced to an inverse problem for the nonlocal MGT equation and so the determination of linear perturbations turns out to be an important preliminary step for the nonlinear problems as well.
    
    In Section \ref{subsec: rigorous def of DN maps} we first rigorously introduce the DN maps for the PDEs considered in this article and then prove in Section \ref{subsec: runge} a suitable Runge approximation property for the solutions to linear third order wave equations. Finally, Section \ref{sec: determ of q} contains the unique determination of linear perturbations. 

    \subsection{DN maps}
    \label{subsec: rigorous def of DN maps}
    
    Let us start by noticing that the well-posedness results of Section \ref{sec: well-posedness forward} allow us to define the DN maps for the linear and nonlinear problems as long as the exterior values are sufficiently small. To formulate the precise definitions, we need to introduce for any nonempty open set $W\subset\Omega_e$, $\delta>0$ and $q\in L^{\infty}(0,T;L^p(\Omega))$ with $p$ satisfying \eqref{conp} the following sets
    \begin{equation}
    \label{eq: def Y delta}
        Y^q_\delta(W)\vcentcolon =\left\{\psi\in C_c^{\infty}(W_T)\,;\, \|\psi\|_{W^{1,\infty}(0,T; H^{2s}(\R^n))}\leq \frac{\delta}{2(b+|c|)C_0}\right\},
    \end{equation}
    where $C_0>0$, the constant appearing in the energy estimate \eqref{eq: energy estimate nonlinear eq well-posedness}, depends continuously on $\|q\|_{L^{\infty}(0,T;L^p(\Omega))}$ and $b,c$ are the constants in front of $(-\Delta)^s\partial_t$ and $(-\Delta)^s$ in the (nonlocal) MGT equation. Observe that if $\varphi\in Y_\delta$ for some $\delta\leq \delta_0$, where $\delta_0>0$ is the constant from Proposition \ref{prop: nonlinearg with ext cond} and Theorem \ref{well-Wel}, respectively, and $u_0=u_1=u_2=F=0$, then $\varphi$ satisfies the smallness conditions \eqref{eq: smallness on source F with ext cond} and \eqref{eq: smallness on source F with ext cond beta, kappa}. Furthermore, if $\widetilde{q}\in L^{\infty}(0,T;L^p(\Omega))$ is another potential and $\widetilde{C}_0>0$ denotes the constant in the corresponding energy estimate \eqref{eq: energy estimate nonlinear eq well-posedness}, then we set
    \begin{equation}
    \label{eq: def Y delta intersection two potentials}
        Y^{q\cap \widetilde{q}}_{\delta}(W)\vcentcolon =\left\{\psi\in C_c^{\infty}(W_T)\,;\, \|\psi\|_{W^{1,\infty}(0,T; H^{2s}(\R^n))}\leq \frac{\delta}{2(b+|c|)\max(C_0,\widetilde{C}_0)}\right\}
    \end{equation}
    for any $\delta>0$. Now, using the results of Section \ref{sec: well-posedness forward} we can define the DN maps as follows:
    
    \begin{definition}
    \label{def: DN maps}
        Let $\Omega\subset\R^n$ be a bounded Lipschitz domain, $T>0$ and $s>0$. Suppose that $q\in L^{\infty}(0,T;L^p(\Omega))$ is real-valued with $p$ satisfying \eqref{conp} and assume the Carath\'eodory function $f\colon \Omega_T\times\R\to\R$ obeys the conditions of Theorem \ref{wellnonlinear}. Moreover, let $(\beta,\kappa)\in W^{2,\infty}(0,T;L^\infty(\Omega))\times W^{1,\infty}(0,T;L^\infty(\Omega))$ and denote by $\mathcal{F}$ the Westervelt-type nonlinearities \eqref{eq: nonlinearity F}.
        \begin{enumerate}[({DN}1)]
            \item\label{linear DN map} We define the DN map $\Lambda_q\colon C_c^{\infty}((\Omega_e)_T)\to\distr((\Omega_e)_T)$ related to 
            \begin{equation}
		    \label{eq: PDE for DN linear}
		\begin{cases}
			(\partial_t^3+\alpha\partial_t^2+b(-\Delta)^s\partial_t+c(-\Delta)^s+q)u=0\ &  {\rm in}\ \Omega_T,\\
			u=\varphi\ & {\rm in}\ (\Omega_e)_T, \\
			u(0)=\partial_t u(0)=\partial_t^2 u(0)=0 & {\rm in}\ \Omega
		\end{cases}
		\end{equation} 
        by
        \begin{equation}
        \label{eq: DN linear}
            \langle \Lambda_q \varphi,\rho\rangle\vcentcolon = \int_0^T \langle b(-\Delta)^{s/2}\partial_t u+c(-\Delta)^{s/2}u,(-\Delta)^{s/2}\rho\rangle_{L^2(\R^n)} \,dt,
        \end{equation}
        where $\varphi,\rho\in C_c^{\infty}((\Omega_e)_T)$ and $u$ denotes the unique solution of \eqref{eq: PDE for DN linear} (see Corollary \ref{wellu}).
            \item\label{nonlinear DN map 1} We define the DN map $\Lambda_{q,f}\colon Y^q_{\delta_0}(\Omega_e)\to\distr((\Omega_e)_T)$ related to 
            \begin{equation}
		    \label{eq: PDE for DN nonlinear 1}
		\begin{cases}
            (\partial_t^3+\alpha\partial_t^2+b(-\Delta)^s\partial_t+c(-\Delta)^s+q)u+f(u)=0 & \text{ in }\Omega_T,\\
		u=\varphi & \text{ in } (\Omega_e)_T, \\
		u(0)=\partial_tu(0)=\partial_t^2u(0)=0 & \text{ in } \Omega
        \end{cases}
		\end{equation} 
        by
        \begin{equation}
        \label{eq: DN nonlinear 1}
            \langle \Lambda_{q,f}\varphi,\rho\rangle\vcentcolon = \int_0^T \langle b(-\Delta)^{s/2}\partial_t u+c(-\Delta)^{s/2}u,(-\Delta)^{s/2}\rho\rangle_{L^2(\R^n)} \,dt,
        \end{equation}
        where $\varphi\in Y^q_{\delta_0}(\Omega_e)$, $\rho\in  C_c^{\infty}((\Omega_e)_T)$ and $u\in \varphi+X_{\delta_0}$ denotes the unique solution to \eqref{eq: PDE for DN nonlinear 1} (see Proposition \ref{prop: nonlinearg with ext cond}).
        \item\label{nonlinear DN map 2} If $2s>n$, then we define the DN map $\Lambda_\mathcal{F}\colon Y_{\delta_0}\vcentcolon = Y^0_{\delta_0}(\Omega_e)\to\distr((\Omega_e)_T)$ related to 
            \begin{equation}
		    \label{eq: PDE for DN nonlinear 2} 
		\begin{cases}
		(\partial_t^3+\alpha \partial_t^2+b(-\Delta)^s\partial_t+c(-\Delta)^s)u=\mathcal{F}(u,\partial_t u,\partial_t^2 u) & {\rm in}\ \Omega_T, \\
		u=\varphi,& {\rm in} \ (\Omega_e)_T,\\
        u(0)=\partial_tu(0)=\partial_t^2u(0)=0  & {\rm in}\ \Omega
	\end{cases}
		\end{equation} 
        by
        \begin{equation}
        \label{eq: DN nonlinear 2}
            \langle\Lambda_{\mathcal{F}}\varphi,\rho\rangle\vcentcolon = \int_0^T \langle b(-\Delta)^{s/2}\partial_t u+c(-\Delta)^{s/2}u,(-\Delta)^{s/2}\rho\rangle_{L^2(\R^n)} \,dt,
        \end{equation}
        where $\varphi\in Y_{\delta_0}(\Omega_e)$, $\rho\in  C_c^{\infty}((\Omega_e)_T)$ and $u\in \varphi+X_{\delta_0}$ denotes the unique solution to \eqref{eq: PDE for DN nonlinear 2} (see Theorem \ref{well-Wel}). Moreover, we set $\Lambda_\beta\vcentcolon =\Lambda_\mathcal{F}$ when $\mathcal{F}(u,\partial_tu,\partial_t^2u)=\partial_t^2(\beta u^2)$ and $\Lambda_\kappa\vcentcolon =\Lambda_\mathcal{F}$ when $\mathcal{F}(u,\partial_tu,\partial_t^2u)=\partial_t(\kappa(\partial_tu)^2)$.
        \end{enumerate}
    \end{definition}

    \begin{remark}
    \label{eq: interpretation of equal DN maps}
    In this remark we discuss what we mean by equality of DN maps for small exterior conditions.
    \begin{enumerate}[(a)]
        \item\label{domain of def for nonlin DN map 1} Suppose that $(q_1,f_1),(q_2,f_2)$ are two pairs of coefficients as in Definition \ref{def: DN maps} with constants $\delta_0^1,\delta_0^2>0$ and $C_0^1,C_0^2>0$. If we define $\delta_0=\min(\delta_0^1,\delta_0^2)$ and $C_0 =\max(C_0^1,C_0^2)$, then the DN maps $\Lambda_{q_1,f_1}$ and $\Lambda_{q_2,f_2}$ are well-defined on the set $Y^{q_1\cap q_2}_{\delta_0}(\Omega_e)\subset Y^{q_1}_{\delta_0^1}(\Omega_e)\cap Y^{q_2}_{\delta_0^2}(\Omega_e)$. 
        \item\label{domain of def for nonlin DN map 2} Suppose that $(\beta_1,\kappa_1),(\beta_2,\kappa_2)$ are two pair of coefficients as in Definition \ref{def: DN maps} with constants $\delta_0^1,\delta_0^2>0$ and $C_0^1,C_0^2>0$. Then $C_0^1=C_0^2$ and hence $\Lambda_{\beta_1}, \Lambda_{\beta_2}$ and $\Lambda_{\kappa_1},\Lambda_{\kappa_2}$, respectively, are defined on the same set, namely $Y_{\delta_0}(\Omega_e)$ with $\delta_0=\min(\delta_0^1,\delta_0^2)$.
        \item In view of \ref{domain of def for nonlin DN map 1}, we say that for some open sets $W_1,W_2\subset\Omega_e$ and coefficients $(q_1,f_1), (q_2,f_2)$ as above, there holds
        \begin{equation}
            \label{eq: equality of DN maps}
            \left.\Lambda_{q_1,f_1}\varphi\right|_{(W_2)_T}=\Lambda_{q_2,f_2}\varphi|_{(W_2)_T}
        \end{equation}
        for all small exterior conditions $\varphi$ support in $(W_1)_T$, if the identity \eqref{eq: equality of DN maps} holds for all $\varphi\in Y_{\delta_0}^{q_1\cap q_2}(W_1)$, where $\delta_0=\min(\delta_0^1,\delta_0^2)$. A similar convention is used for the DN maps $\Lambda_{\kappa}$ and $\Lambda_{\beta}$ (see \ref{domain of def for nonlin DN map 2}).
    \end{enumerate}
   \end{remark}

	\subsection{Runge approximation}
    \label{subsec: runge}
    
	Next, using the unique continuation property (UCP) of the fractional Laplacian (see Proposition~\ref{prop: basic facts on frac Lap}) together with a classical Hahn--Banach argument, we establish the Runge approximation property for solutions to the nonlocal MGT equation~\eqref{eq: PDE for DN linear}. 

The Runge approximation property plays a key role in the uniqueness proofs of the inverse problems considered in this work. Roughly speaking, it allows us to approximate arbitrary square-integrable functions in $\Omega_T$ by solutions generated from exterior data $\varphi \in C_c^\infty(W_T)$, for any open set $W \subset \Omega_e$. For instance, when combined with a suitable integral identity (Lemma~\ref{integral}), this property enables the identification of the potential $q$ in the MGT equation from the DN map $\Lambda_q$ (see Section~\ref{sec: determ of q}).

    Before proceeding, let us introduce the time-reversal of a function. For any (strongly) measurable map $h\colon (0,T)\to X$, where $X$ is a generic Banach space, we define the measurable map
    \begin{equation}
    \label{eq: time reversal invariance}
        h^{\star}(t)=h(T-t)
    \end{equation}
    and call $h^\star$ the time-reversal of $h$. Now, the Runge approximation result can be phrased as follows.
	\begin{proposition}[Runge approximation]
    \label{Runge}
		Let $\Omega\subset\R^n$ be a bounded Lipschitz domain, $T>0$ and $s\in\R_+\setminus\N$. If $W\subset\Omega_e$ is a nonemtpy open set and $q\in L^{\infty}(0,T;L^p(\Omega))$ is a real-valued potential with $2\leq p\leq \infty$ satisfying \eqref{conp}, then the Runge set
		\begin{equation}
        \label{eq: Runge set}
		\mathcal R_W:=\{u_{\varphi}-\varphi\,;\, \varphi\in C_c^\infty(W_T)\}
		\end{equation}
		is dense in $L^2(0,T;\widetilde{L}^2(\Omega))$. Here, $u_\varphi$ denotes the unique solution to the nonlocal MGT equation
		\begin{equation}
		\begin{cases}
			(\partial_t^3+\alpha\partial_t^2+b(-\Delta)^s\partial_t+c(-\Delta)^s+q)u=0 & {\rm in} \ \Omega_T, \\
			u=\varphi & {\rm in}\ (\Omega_e)_T,\\
			u(0)=\partial_tu(0)=\partial_t^2u(0)=0  & {\rm in}\ \Omega
		\end{cases}
		\end{equation}
        (see Corollary \ref{wellu}).
	\end{proposition} 
    
	\begin{proof} The proof is similar to the one in \cite[Proposition 4.2]{PZ1} (see also \cite[Proposition 4.1]{LTZ2} or \cite[Theorem 3.1]{KLW}) dealing with second order in time viscous nonlocal wave equations.

  First, observe that $\mathcal R_W$ is a linear subspace of $L^2(0,T;\widetilde{L}^2(\Omega))$. To prove the desired density result, it therefore suffices, by the Hahn--Banach theorem, to show that any functional $F \in L^2(0,T;\widetilde{L}^2(\Omega))$ that vanishes on $\mathcal R_W$ must be identically zero. In other words, if
\[
\langle F, v \rangle = 0 \quad \text{for all } v \in \mathcal R_W,
\]
then necessarily $F \equiv 0$. Therefore, we may assume that $F\in L^2(0,T;\widetilde{L}^2(\Omega))$ satisfies
	$$
	\langle F,u_\varphi-\varphi\rangle=0,\quad \forall \varphi\in C_c^\infty(W_T).
	$$
	Furthermore, let us denote by $w_F$ the unique solution to the following backward problem 
	\begin{equation}
    \label{sysbackmgt}
	\begin{cases}
		(\partial_t^3-\alpha\partial_t^2+b(-\Delta)^s\partial_t-c(-\Delta)^s-q)w=F& {\rm in} \ \Omega_T, \\
		w=0& {\rm in} \ (\Omega_e)_T,\\
		w(T)=\partial_tw(T)=\partial_t^2w(T)=0& {\rm in} \ \Omega.
	\end{cases}
	\end{equation}
	Note that by writing $\mathcal{Q}=q^{\star}$ (see  \eqref{eq: time reversal invariance}) and applying a time-reversal $t\mapsto T-t$, we may deduce from Theorem \ref{wellposednessv} that the above backward PDE admits a unique solution $w_F$. 
	We next assert that we have
    \begin{equation}
    \label{eq: integration by parts second order derivative}
	\int_0^T\<\partial_t^2(u_\varphi-\varphi),w_F\>dt=\int_0^T\<u_\varphi,\partial_t^2 w_F\>dt, 
	\end{equation}
	\begin{equation}
    \label{eq: integration by parts damping}
	\int_0^T\<(-\Delta)^{s/2}(u_\varphi-\varphi),(-\Delta)^{s/2}\partial_tw_F\>dt=-\int_0^T\<(-\Delta)^{s/2}\partial_t(u_\varphi-\varphi),(-\Delta)^{s/2}w_F\>dt
	\end{equation}
    and 
    \begin{equation}
    \label{eq: integration by parts third order derivative}
	\int_0^T\<\partial_t^3(u_\varphi-\varphi),w_F\>dt=-\int_0^T\<u_\varphi-\varphi,\partial_t^3w_F\>dt. 
	\end{equation}
    The identities \eqref{eq: integration by parts second order derivative}--\eqref{eq: integration by parts damping} follow from the classical integration by parts formula and observing that $u_\varphi-\varphi,w_F\in H^1(0,T;\widetilde{H}^s(\Omega))\cap H^2(0,T;\widetilde{L}^2(\Omega))$. To justify formula \eqref{eq: integration by parts third order derivative}, one can first prove the identity for solutions of the associated parabolically regularized problem (i.e.~solutions to the same PDE but with an additional viscosity term $\eps (-\Delta)^s\partial_t^2$) and then passing to the limit $\eps\to 0$. A similar approach has been used in the articles \cite[Claim 4.2]{LTZ2} and \cite{PZ-KK} for second order nonlocal wave equations. The details of these arguments are presented in Appendix~\ref{appendix: parabolic regularization}. Theorem~\ref{thm: parabolic regularization} studies the unique solvability of parabolically regularized problems and the convergence properties of the solutions $u_\eps$ as $\eps\to 0$, while Proposition~\ref{prop: integration by parts formula} establishes the integration by parts formula \eqref{eq: integration by parts third order derivative}. Hence, we may calculate
	\begin{equation}
    \begin{split}
	0 &=\<F,u_\varphi-\varphi\>=\<(\partial_t^3-\alpha\partial_t^2+b(-\Delta)^s\partial_t-c(-\Delta)^s-q)w_F,u_\varphi-\varphi\> \\
	 &=-\<(\partial_t^3+\alpha\partial_t^2+b(-\Delta)^s\partial_t+c(-\Delta)^s+q)(u_\varphi-\varphi),w_F\> \\
	 &=\<b(-\Delta)^s\partial_t\varphi+c(-\Delta)^s\varphi,w_F\>=\<(-\Delta)^s\varphi,cw_F-b\partial_t w_F\>,\ \forall \varphi\in C_c^\infty(W_T). 
	\end{split}
    \end{equation}
    Let us note that in the second last equality sign, we have used that $u_\varphi-\varphi$ solves the nonlocal MGT equation with source term $F\vcentcolon = -\chi_\Omega(b(-\Delta)^s \partial_t \varphi+c(-\Delta)^s \varphi)$ (see~Proposition~\ref{prop: nonlinearg with ext cond}). Therefore, we have shown that $\widetilde w_F=cw_F-b\partial_t w_F\in L^2(0,T;\widetilde H^s(\Omega))$ satisfies
	$$(-\Delta)^s\widetilde w_F=\widetilde w_F=0\quad {\rm in}\ W_T.$$
	By the UCP of the fractional Laplacian (Proposition \ref{prop: basic facts on frac Lap}), we have $\widetilde w_F=0$ in $\R^n_T$, which shows that $w_F$ solves the ODE problem
	\[
        \begin{cases}
            \partial_tw =\frac{c}{b}w\text{ in }(0,T),\\
            w(T)=0.
        \end{cases}
    \]
	As $w_F\in H^1(0,T;\widetilde{H}^s(\Omega))$, we necessarily have $w_F=0$, which in turn implies $F=0$. This finishes the proof. 
    \end{proof}

    \begin{remark}
    Let us comment here about Runge approximation results for 2nd order in time nonlocal wave equations. The last author together with Y.-H.~Lin and T.~Tyni established in \cite{LTZ1} a $L^2(0,T;\widetilde H^s(\Omega))$ Runge approximation theorem for the nonlocal wave equations $(\partial_t^2+(-\Delta)^s+q)u=0$, in which the conditions for $q$ are the same as those considered in this manuscript. The main ingredient in the proof is the existence of so-called very weak solutions $u\in C([0,T];\widetilde{L}^2(\Omega))$. In \cite{PZ2}, the last author extended this approximation result to damped nonlocal wave equations $(\partial_t^2+\gamma \partial_t+(-\Delta)^s+q)u=0$. In the forthcoming work \cite{PZ-KK}, the $L^2(0,T;\widetilde{H}^s(\Omega))$ Runge approximation is strengthened and generalized to other types of nonlocal wave equations.
	\end{remark}

	\subsection{{Determining the coefficient \texorpdfstring{$q$}{q}}}
    \label{sec: determ of q}
In this subsection, we will prove the unique determination of the linear coefficient $q$ from the knowledge of the DN map. To achieve this, we first establish an integral identity which builds a bridge between the DN map and $q$. By the integral identity and the Runge approximation (Proposition \ref{Runge}), we can uniquely determine $q$.  
%In this subsection, we will present an integral identity, which builds a bridge between the DN map and the potential term. 
	\begin{lemma}
    \label{integral}
		{\rm (Integral identity)}\ Let $\Omega\subset\R^n$ be a bounded Lipschitz domain,   $T>0$ and $s>0$. Suppose that the potentials $q_1,q_2\in L^{\infty}(0,T;L^p(\Omega))$ are real valued and $p$ satisfies the condition \eqref{conp}. Then for all $\varphi_1,\varphi_2\in C_c^\infty((\Omega_e)_T)$, we have the integral identity
		\begin{equation}
        \label{eq: integral identity potential}
		\int_{\Omega_T}(q_1-q_2^{\star})(u_1-\varphi_1)(u_2-\varphi_2)^{\star} dxdt=\langle (\Lambda_{q_1}-\Lambda_{q_2^\star})\varphi_1,\varphi_2^{\star}\rangle,
		\end{equation}
		where $u_j$ is the unique solution to the problem
		\begin{equation}
        \label{sysfj}
		\begin{cases}
			(\partial_t^3+\alpha\partial_t^2+b(-\Delta)^s\partial_t+c(-\Delta)^s+q_j)u_j=0 & {\rm in} \ \Omega_T, \\
			u_j=\varphi_j & {\rm in} \ (\Omega_e)_T,\\
			u_j(0)=\partial_tu_j(0)=\partial_t^2u_j(0)=0 & {\rm in} \ \Omega
		\end{cases}
		\end{equation}
        for $j=1,2$ (see Corollary \ref{wellu}).
	\end{lemma}

     \begin{proof}
	We follow the same strategy as in \cite[Lemma 4.3]{PZ1} (see also \cite[Proposition 3.4]{PZ2}). Let $Q_1,Q_2\in L^{\infty}(0,T;L^p(\Omega))$ with $p$ satisfying the condition \eqref{conp}. Suppose that $w_1$ and $w_2^{\star}$ are the unique solutions to
	\begin{equation}
	\begin{cases}
		(\partial_t^3+\alpha\partial_t^2+b(-\Delta)^s\partial_t+c(-\Delta)^s+Q_1)w=0 & {\rm in} \ \Omega_T, \\
		w=\varphi_1 & {\rm in} \ (\Omega_e)_T,\\
		w(0)=\partial_tw(0)=\partial_t^2w(0)=0 & {\rm in} \ \Omega
	\end{cases}
	\end{equation}
	and
	\begin{equation}
	\begin{cases}
		(\partial_t^3-\alpha\partial_t^2+b(-\Delta)^s\partial_t-c(-\Delta)^s-Q_2)w=0 & {\rm in} \ \Omega_T, \\
		w=\varphi_2^{\star} & {\rm in} \ (\Omega_e)_T,\\
		w(T)=\partial_tw(T)=\partial_t^2w(T)=0 & {\rm in} \ \Omega,
	\end{cases}
	\end{equation}
    respectively (see Corollary \ref{wellu}).
	By construction $w_2^{\star}$ is the time-reversal of the unique solution $w_2$ to the problem
	\begin{equation}
	\begin{cases}
		(\partial_t^3+\alpha\partial_t^2+b(-\Delta)^s\partial_t+c(-\Delta)^s+Q_2^{\star})w=0 & {\rm in} \ \Omega_T, \\
		w=\varphi_2 & {\rm in} \ (\Omega_e)_T,\\
		w(0)=\partial_tw(0)=\partial_t^2w(0)=0 & {\rm in} \ \Omega
	\end{cases}
	\end{equation}
	(cf.,~e.g.~proof of Proposition \ref{Runge}). Next, we calculate
    \begin{equation}
    \label{f1f2}
        \begin{split}
            &\int_{\Omega_T}(Q_1-Q_2)(w_1-\varphi_1)(w_2-\varphi_2)^{\star} dxdt\\
            &=-\int_0^T \langle (\partial_t^3+\alpha\partial_t^2+b(-\Delta)^s\partial_t+c(-\Delta)^s)(w_1-\varphi_1),(w_2-\varphi_2)^\star \rangle \,dt\\
            &\quad -\int_0^T \langle (b(-\Delta)^s\partial_t+c(-\Delta)^s)\varphi_1,(w_2-\varphi_2)^\star\rangle\,dt\\
            &\quad -\int_0^T \langle (\partial_t^3-\alpha\partial_t^2+b(-\Delta)^s\partial_t-c(-\Delta)^s)(w_2-\varphi_2)^\star ,w_1-\varphi_1\rangle \,dt\\
            &\quad -\int_0^T \langle (b(-\Delta)^s\partial_t-c(-\Delta)^s)\varphi_2^\star,w_1-\varphi_1\rangle\,dt\\
            &=-\int_0^T \langle (b(-\Delta)^s\partial_t+c(-\Delta)^s)\varphi_1,(w_2-\varphi_2)^\star\rangle\,dt\\
             &\quad -\int_0^T \langle (b(-\Delta)^s\partial_t-c(-\Delta)^s)\varphi_2^\star,w_1-\varphi_1\rangle\,dt\\
            &=\int_0^T \langle (b(-\Delta)^s\partial_t-c(-\Delta)^s)(w_2^\star-\varphi_2^\star),\varphi_1\rangle\,dt\\
             &\quad +\int_0^T \langle (b(-\Delta)^s\partial_t+c(-\Delta)^s)(w_1-\varphi_1),\varphi_2^\star\rangle\,dt\\
             &=\int_0^T \langle (b(-\Delta)^s\partial_t-c(-\Delta)^s)w_2^\star,\varphi_1\rangle\,dt\\
             &\quad +\int_0^T \langle (b(-\Delta)^s\partial_t+c(-\Delta)^s)w_1,\varphi_2^\star\rangle\,dt\\
             &=\langle \Lambda_{Q_1}\varphi_1,\varphi_2^\star\rangle-\langle \Lambda_{Q_2^\star}\varphi_2,\varphi_1^\star\rangle.
        \end{split}
    \end{equation}
    In the first equality we used the observations that we made in the proof Corollary \ref{wellu} and the fact that $\varphi_1,\varphi_2$ are supported on $(\Omega_e)_T$. In the second equality, we applied the formulas \eqref{eq: integration by parts second order derivative}--\eqref{eq: integration by parts third order derivative} that remain valid in the present case. In the third equality, we performed again an integration by parts. In the fourth equality we used that the contributions that only contain $\varphi_1$ and $\varphi_2^\star$ compensate each other. Finally, in the fifth equality we made a change of variables for the first integral and used the definitions of the DN maps.

	If $Q_1=Q_2=q_j$ for $j=1,2$, then we see from the identity \eqref{f1f2} that 
    \begin{equation}
    \label{eq: adjoint DN map}
        \<\Lambda_{q_j}\varphi_1,\varphi_2^{\star}\>=\<\Lambda_{q_j^{\star}}\varphi_2,\varphi_1^{\star}\>.
    \end{equation}
	Thus, we deduce from \eqref{eq: adjoint DN map} and \eqref{f1f2} with $Q_1=q_1$ and $Q_2=q_2^\star$ that there holds
	\begin{equation}
    \begin{split}
	\int_{\Omega_T}(q_1-q_2^{\star})(u_1-\varphi_1)(u_2-\varphi_2)^{\star} dxdt &=
	\<\Lambda_{q_1}\varphi_1,\varphi_2^{\star}\>-\<\Lambda_{q_2}\varphi_2,\varphi_1^{\star}\> \\
	 &=\<\Lambda_{q_1}\varphi_1,\varphi_2^{\star}\>-\<\Lambda_{q_2^\star}\varphi_1,\varphi_2^{\star}\> \\
	 &=\<(\Lambda_{q_1}-\Lambda_{q_2^\star})\varphi_1,\varphi_2^{\star}\>, 
	\end{split}
    \end{equation}
	which completes the proof. 
    \end{proof}
	
	We are now in a position to prove the unique determination of $L^p$ potentials from the DN map.

    \begin{proposition}[Uniqueness of linear perturbations]
    \label{prop: uniqueness of linear perturbations}
        Let $\Omega\subset\R^n$ be a bounded Lipschitz domain,   $T>0$ and $s\in \R_+\setminus \N$. Suppose that the potentials $q_1,q_2\in L^{\infty}(0,T;L^p(\Omega))$ are real valued, $p$ satisfies the condition \eqref{conp} and $q_2$ is time-reversal invariant (i.e. $q_2=q_2^\star$). If $W_1,W_2\subset\Omega_e$ are two nonempty open sets and we have
        \begin{equation}
        \label{eq: equal dn maps linear}
            \left.\Lambda_{q_1}\varphi\right|_{(W_2)_T}=
            \left.\Lambda_{q_2} \varphi\right|_{(W_2)_T}
        \end{equation}
        for all $\varphi\in C_c^{\infty}((W_1)_T)$, then there holds $q_1=q_2$ in $\Omega_T$.
    \end{proposition}
	
	\begin{proof} 
    First, note that according to Lemma \ref{integral}, $q_2=q_2^\star$ and \eqref{eq: equal dn maps linear}, we have the following identity	\begin{equation}
    \label{eq: identity for detection of potentials}
	\int_{\Omega_T}(q_1-q_2)(u_1-\varphi_1)(u_2-\varphi_2)^{\star} dxdt=0
	\end{equation}
    for any $\varphi_j\in C_c^{\infty}((W_j)_T)$, $j=1,2$, where $u_j$ denotes the unique solution to \eqref{sysfj}. Using Proposition \ref{Runge} and Lemma \ref{lemma: potential}, we may deduce that 
	\[
        \int_{\Omega_T}(q_1-q_2)(u_1-\varphi_1)\eta dxdt=0
    \]
    for all $\eta\in C_c^{\infty}(\Omega_T)$ and all $\varphi_1\in C_c^{\infty}((W_1)_T)$. Using the same argument again, we get
    \[
        \int_{\Omega_T}(q_1-q_2)\psi\eta dxdt=0
    \]
    for all $\psi,\eta\in C_c^{\infty}(\Omega_T)$. This ensures that $q_1=q_2$ a.e. in $\Omega_T$.
	\end{proof}
	
	\section[Proof of Theorem~\ref{thmIP1}]{Proof of Theorem~\ref{thmIP1}: Polynomial-type nonlinearities}
    \label{sec: proof of theorem 1.5}

    In this section, we establish the proof of Theorem \ref{thmIP1} that relies on the method of higher order linearization and the UCP for the fractional Laplacian (Proposition \ref{prop: basic facts on frac Lap}). Let us remark that a higher order linearization was not necessary for second order nonlocal wave equations \cite{LTZ1,LTZ2}, while a second order linearization procedure has already been used for second order nonlocal viscous wave equations \cite{PZ1}.
    
More concretely, in Section \ref{sec: first order linearization}, we perform a first order linearization that allows us to recover the linear perturbation $q$ and afterwards, in Section \ref{sec: second order linearization}, we carry out a second order linearization yielding that $\partial_{\tau}^2 g(x,t,0)$ is uniquely determined by the DN map. Finally, in Section \ref{sec: higher order linearization}, we generalize the results of the first two paragraphs and demonstrate that the $N-$th order linearization, for $3\leq N\leq m$, allows us to recover $\partial_{\tau}^{N}g(x,t,0)$. We note that  the Runge approximation (Proposition \ref{Runge}) plays a key role in recovering $\{\partial_\tau^{\ell}g(x,t,0)\}_{\ell=1}^N$.   
    
	\subsection{First order linearization and recovery of \texorpdfstring{$q_j$}{potential}}
    \label{sec: first order linearization}

      We begin with some preparations on the first order linearization of the solutions and the related DN maps.
	
    We consider the following semilinear MGT equation
	\begin{equation}
	    \label{uj}
	\begin{cases}
		(\partial_t^3+\alpha\partial_t^2+b(-\Delta)^s\partial_t+c(-\Delta)^s+q_j)u_j^\varepsilon+g_j(u_j^\varepsilon)=0\ & {\rm in}\ \Omega_T,\\
		u^\varepsilon_j=\eps\cdot\varphi\ & {\rm in}\ (\Omega_e)_T, \\
		u^\varepsilon_j(0)=\partial_tu^\varepsilon_j(0)=\partial_t^2u^\varepsilon_j(0)=0\ & {\rm in}\ \Omega,
    	\end{cases}
	\end{equation}
    where the appearing quantities satisfy the assumptions of Theorem \ref{wellnonlinear}, the nonlinearities $g_j$, $j=1,2$, obey \ref{A1}--\ref{A3}, $\varphi=(\varphi_1,\ldots,\varphi_m)\in C_c^{\infty}((\Omega_e)_T)$ are given exterior conditions and $\eps=(\eps_1,\ldots,\eps_m)\in\R^m$ are some small parameters. Furthermore, $m\in \N_{\geq 2}$ is the constant that appears in \ref{A3}.

    From Proposition \ref{prop: nonlinearg with ext cond}, we deduce that for given $\varphi=(\varphi_1,\ldots,\varphi_m)\in C_c^{\infty}((\Omega_e)_T)$ the problem \eqref{uj} is uniquely solvable for sufficiently small $\eps$. If $\varepsilon=0$, then using $g_j(x,t,0)=0$ we see that $u_j^0=0$ is the unique solution to \eqref{uj}.
  
	The corresponding DN map associated with the above system reads
    \[
        \Lambda_{q_j,g_j}(\eps\cdot\varphi)=(b(-\Delta)^s\partial_tu_j^\varepsilon+c(-\Delta)^su_j^\varepsilon)|_{(\Omega_e)_T},\quad j=1,2.
    \]
		
    Next, let us introduce the following notation. If $Z$ is a Banach space, $\Theta\subset\R^m$ an open subset and $f\colon \Theta\to Z$ a given function, then we put
    \begin{equation}
    \label{eq: difference operator}
        \Delta_\theta f(\theta_0)=f(\theta_0+\theta)-f(\theta_0)
    \end{equation}
    for any $\theta_0\in\Theta$ and $\theta\in\R^m$ such that $\theta_0+\theta\in \Theta$. Furthermore, if $(e_j)_{1\leq j\leq m}$ denotes the canonical basis of $\R^m$, then we define
    \begin{equation}
    \label{eq: difference quotient operator}
        \delta^k_\eta f(\theta_0)=\frac{\Delta_{\eta e_k}f(\theta_0)}{\eta}
    \end{equation}
    for any $1\leq k\leq m$ as long as $\theta_0\in\Theta$ and $\eta\in\R$ satisfy $\theta_0+\eta e_k\in\Theta$.
    
    We have the following first order linearization result (cf.,~e.g.,~\cite{LaLin,KMSK,LTZ1,LTZ2,PZ1}). 
	\begin{lemma}\label{der1}
		Under the above assumptions, there exists a unique solution $v_{j,k}^\varepsilon$ to
        \begin{equation}
            \label{vj}
		\begin{cases}
			(\partial_t^3+\alpha\partial_t^2+b(-\Delta)^s\partial_t+c(-\Delta)^s+q_j)v+\partial_{\tau}g_j(u_j^\varepsilon)v=0\ & {\rm in}\ \Omega_T,\\
			v=\varphi_k\ & {\rm in}\ (\Omega_e)_T, \\
			v(0)=\partial_tv(0)=\partial_t^2v(0)=0\ & {\rm in}\ \Omega,
		\end{cases}
        \end{equation}
        when $\eps\in\R^m$ is sufficiently small. Furthermore, one has
             \begin{equation}
    \label{eq: derivative wrt eps 1}
        v_{j,k}^\varepsilon=\partial_{\varepsilon_k}u_j^\varepsilon \vcentcolon =\lim\limits_{\eta\to 0}\delta^k_{\eta}u_j^{\eps}
    \end{equation}
    in $W^{1,\infty}(0,T;H^s(\R^n))\cap W^{2,\infty}(0,T;L^2(\R^n))$.
	\end{lemma}
    \begin{proof}
        We divide the proof into four steps. The first step establishes the existence of the unique solution $v_{j,k}^\eps$ and the remaining steps demonstrate formula \eqref{eq: derivative wrt eps 1}. \\
        
        \noindent\textit{Step 1.} In light of Corollary \ref{wellu}, the existence of $v_{j,k}^\eps$ follows once we have shown that 
        \begin{equation}
            \label{eq: regularity of partial g u eps}
            \partial_{\tau}g_j(u_j^\varepsilon)\in L^{\infty}(0,T;L^p(\Omega)),    
        \end{equation}
        for some $p$ satisfying the constraints \eqref{conp}. To this end, recall that $g_j$ obeys 
        \begin{equation}
        \label{eq: estimate for der of gj}
        \begin{split}
            |\partial_{\tau} g_j(x,t,\tau)|&\lesssim |\tau|^\gamma+|\tau|^r\lesssim 1+|\tau|^r\\
        \end{split}
        \end{equation}
        for a.e.~$(x,t)\in\Omega_T$, where $0<\gamma\leq r$ and $r$ satisfies \eqref{conr}.

        If $2s<n$, then the estimate \eqref{eq: estimate for der of gj} and Lemma \ref{lemma: Nemytskii operators} guarantee that 
        \begin{equation}
        \label{eq: continuity prop partial tau gj}
            \partial_{\tau}g_j(u_j^\varepsilon)\in L^{q/r}(0,T;L^{\frac{2n}{r(n-2s)}}(\Omega))\hookrightarrow L^{q/r}(0,T;L^{n/s}(\Omega))
        \end{equation}
        for any $q>r$ and $\partial_\tau g_j$ is continuous from $L^q(0,T;L^{\frac{2n}{n-2s}}(\Omega))$ to $L^{q/r}(0,T;L^{n/s}(\Omega))$. So, passing to the limit $q\to\infty$ in \eqref{eq: boundedness Nemytskii} gives \eqref{eq: regularity of partial g u eps} with $p=n/s$. In the range $2s\geq n$, the result can be established in a similar way.\\
        
        \noindent \textit{Step 2.} To prove the convergence result, observe that \eqref{eq: derivative wrt eps 1} is equivalent to
        \begin{equation}
        \label{eq: derivative wrt eps 1 reformulated}
            V_{j,k}^\eps=\lim_{\eta\to 0}\delta^k_{\eta}U_{j}^\eps\text{ in }X,
        \end{equation}
        where $V_{j,k}^\eps=v_{j,k}^\eps-\varphi_k$ and $U_{j}^\eps=u_j^\eps -\eps\cdot\varphi$ solve the related homogeneous PDEs, namely, $V_{j,k}^\eps\in X$ solves
        \begin{equation}
        \label{eq: PDE for V j k}
            \begin{cases}
			(\partial_t^3+\alpha\partial_t^2+b(-\Delta)^s\partial_t+c(-\Delta)^s+q_j)V+\partial_{\tau}g_j(U_j^\varepsilon)V=F_k\ & {\rm in}\ \Omega_T,\\
			V=0\ & {\rm in}\ (\Omega_e)_T, \\
			V(0)=\partial_tV(0)=\partial_t^2V(0)=0\ & {\rm in}\ \Omega
		\end{cases}
        \end{equation}
        and $U_j^{\rho}\in X$, for small $\rho\in\R^m$, solves
        \begin{equation}
        \label{eq: PDE for U j rho}
            \begin{cases}
		(\partial_t^3+\alpha\partial_t^2+b(-\Delta)^s\partial_t+c(-\Delta)^s+q_j)U+g_j(U)=\rho\cdot F\ & {\rm in}\ \Omega_T,\\
		U=0\ & {\rm in}\ (\Omega_e)_T, \\
		U(0)=\partial_t U(0)=\partial_t^2U(0)=0\ & {\rm in}\ \Omega,
    	\end{cases}
        \end{equation}
        where we put
        \begin{equation}
        \label{eq: sources}
            F=(F_1,\ldots,F_m)\text{ and }F_\ell\vcentcolon =-\chi_\Omega (b(-\Delta)^s\partial_t+c(-\Delta)^s)\varphi_\ell\text{ for }1\leq \ell\leq m.
        \end{equation}
        
        \noindent\textit{Step 3.} From \eqref{eq: PDE for V j k}--\eqref{eq: PDE for U j rho}, we infer that $W_{j,k}^{\eps,\eta}\vcentcolon = V_{j,k}^\eps-\delta_{\eta}^k U_{j}^\eps$ satisfies
         \begin{equation}
        \label{eq: PDE for W j k}
            \begin{cases}
		(\partial_t^3+\alpha\partial_t^2+b(-\Delta)^s\partial_t+c(-\Delta)^s+q_j)W=-G_{j,k}^{\eps,\eta} \ & {\rm in}\ \Omega_T,\\
		W=0\ & {\rm in}\ (\Omega_e)_T, \\
		W(0)=\partial_t W(0)=\partial_t^2W(0)=0\ & {\rm in}\ \Omega,
    	\end{cases}
        \end{equation}
        where
        \begin{equation}
        \label{eq: G jk eps}
        \begin{split}
            G_{j,k}^{\eps,\eta} &\vcentcolon = \partial_\tau g_j(U_j^\eps)V_{j,k}^\eps -\delta_{\eta}^k g_j(U_j^\eps)\\
            &=\partial_\tau g_j(U_j^\eps)V_{j,k}^\eps-\left(\int_0^1 \partial_\tau g_j(U_j^\eps+\sigma\Delta_{\eta e_k} U_j^{\eps})\,d\sigma\right) \delta_{\eta}^k U_j^{\eps}\\
            &=\left(\int_0^1 \partial_\tau g_j(U_j^\eps+\sigma\Delta_{\eta e_k} U_j^{\eps})\,d\sigma\right) (V_{j,k}^\eps-\delta_{\eta}^k U_j^{\eps})\\
            &\quad +\left(\partial_\tau g_j(U_j^\eps)-\int_0^1 \partial_\tau g_j(U_j^\eps+\sigma\Delta_{\eta e_k} U_j^{\eps})\,d\sigma\right)V_{j,k}^\eps
        \end{split}
        \end{equation}
        (cf.~\eqref{eq: PDE for difference of U j}). We claim that the last line converges to zero in $L^2(0,T;\widetilde{L}^2(\Omega))$ as $\eta\to 0$, which in turn follows from H\"older's inequality and Lemma \ref{lemma: potential} if we show that 
        \begin{equation}
        \label{eq: convergence of integral of derivative of g j}
            \int_0^1 \partial_\tau g_j(U_j^\eps+\sigma\Delta_{\eta e_k} U_j^{\eps})\,d\sigma\to \partial_\tau g_j(U_j^\eps) \text{ in }L^q(0,T; L^p(\Omega))
        \end{equation}
        as $\eta\to 0$, for some $q\geq 2$ and $p\geq 2$ satisfying \eqref{conp}. To see this, we consider the PDE for $\Delta_{\eta e_k}U_j^\eps$, that is,
        \begin{equation}
               \label{eq: PDE for difference of U j}
            \begin{cases}
		(\partial_t^3+\alpha\partial_t^2+b(-\Delta)^s\partial_t+c(-\Delta)^s+q_j)\Delta_{\eta e_k}U=\eta F_k -\Delta_{\eta e_k}g_j(U_j^\eps)\ & {\rm in}\ \Omega_T,\\
		U=0\ & {\rm in}\ (\Omega_e)_T, \\
		U(0)=\partial_t U(0)=\partial_t^2U(0)=0\ & {\rm in}\ \Omega.
    	\end{cases}
        \end{equation}
        Next, let us note that \eqref{eq: estimate for der of gj} ensures the estimate 
    \begin{equation}
        \begin{split}
            |g_j(x,t,\tau_1)-g_j(x,t,\tau_2)| &=\Big|\int_{\tau_1}^{\tau_2}\partial_\tau g_j(x,t,\sigma)d\sigma\Big| \\
	 &\lesssim (|\tau_1|^\gamma+|\tau_1|^r+|\tau_2|^\gamma+|\tau_2|^r)|\tau_1-\tau_2|
        \end{split}
    \end{equation}
    for all $\tau_1,\tau_2\in\R$. Hence, we obtain
    \begin{equation}
    \begin{split}
        \|\Delta_{\eta e_k} g_j(U_j^\eps)\|_{L^2(\Omega)}&\lesssim \|(|U_j^{\eps+\eta e_k}|^\gamma +|U_j^\eps |^\gamma )\Delta_{\eta e_k}U_j^{\eps}\|_{L^2(\Omega)}\\
        &\quad +\|(|U_j^{\eps+\eta e_k}|^r+|U_j^\eps |^r)\Delta_{\eta e_k}U_j^{\eps}\|_{L^2(\Omega)}
    \end{split}
    \end{equation}
    for a.e.~$0<t<T$. Using \cite[eq.~(3.26)]{LTZ2} and the fact that $0<\gamma\leq r$ satisfy \eqref{conr}, we can conclude that
    \[
    \begin{split}
         &\|\Delta_{\eta e_k} g_j(U_j^\eps)\|_{L^2(\Omega_T)}\\
         &\lesssim \sup_{\rho=\eps,\eps+\eta e_k}(\|U_j^{\rho}\|_{L^{\infty}(0,T;\widetilde{H}^s(\Omega))}^\gamma+\|U_j^{\rho}\|_{L^{\infty}(0,T;\widetilde{H}^s(\Omega))}^r)\|\Delta_{\eta e_k}U_j^{\eps}\|_{L^{\infty}(0,T;\widetilde{H}^s(\Omega))}.
    \end{split}
    \]
    By construction of $U_j^\rho$ (with $\rho$ small), we know that $\|U_j^{\rho}\|_X\lesssim |\rho|$ (see~Theorem \ref{wellnonlinear}) and hence we may deduce that
    \begin{equation}
    \label{eq: estimate difference of nonlin}
        \|\Delta_{\eta e_k} g_j(U_j^\eps)\|_{L^2(\Omega_T)}\lesssim (|\eps|^\gamma+|\eta|^\gamma)\|\Delta_{\eta e_k}U_j^{\eps}\|_{L^{\infty}(0,T;\widetilde{H}^s(\Omega))}
    \end{equation}
    for $\eps$ and $\eta$ small. Combining this with the energy estimate for \eqref{eq: PDE for difference of U j} (see \eqref{energyest}), we deduce that there holds
    \begin{equation}
        \begin{split}
		 &\|\partial_t^2 \Delta_{\eta e_k}U_j^{\eps}\|_{L^{\infty}(0,T;L^2(\Omega))}+\|\Delta_{\eta e_k}U_j^{\eps}\|_{W^{1,\infty}(0,T;\widetilde H^s(\Omega))} \\
		 &\lesssim |\eta|\|F_k\|_{L^2(\Omega_T)}+\|\Delta_{\eta e_k} g_j(U_j^\eps)\|_{L^2(\Omega_T)}\\
         &\lesssim |\eta|\|F_k\|_{L^2(\Omega_T)} +(|\eps|^\gamma+|\eta|^\gamma)\|\Delta_{\eta e_k}U_j^{\eps}\|_{L^{\infty}(0,T;\widetilde{H}^s(\Omega))}
        \end{split}
    \end{equation}
    Thus, if $|\eps|^\gamma+|\eta|^\gamma$ is sufficiently small, then we get
    \[
        \|\partial_t^2 \Delta_{\eta e_k}U_j^{\eps}\|_{L^{\infty}(0,T;L^2(\Omega))}+\|\Delta_{\eta e_k}U_j^{\eps}\|_{W^{1,\infty}(0,T;\widetilde H^s(\Omega))} \lesssim |\eta|\|F_k\|_{L^2(\Omega_T)}.
    \]
    This in turn ensures 
    \begin{equation}
    \label{eq: convergence of difference of U eps}
        \Delta_{\eta e_k}U_j^{\eps}\to 0\text{ in }X
    \end{equation}
    as $\eta\to 0$ for small $\eps$. Arguing as in Step 1 (cf.~Lemma \ref{lemma: Nemytskii operators}), we may conclude the convergence \eqref{eq: convergence of integral of derivative of g j}. \\

    \noindent \textit{Step 4.} In this final step, we show that $W_{j,k}^{\eps,\eta}\to 0$ in $X$ as $\eta\to 0$. First, let us observe that by \eqref{eq: convergence of integral of derivative of g j} and Lemma \ref{lemma: potential} the expression in the second last line of \eqref{eq: G jk eps} can be estimated as
    \[
    \begin{split}
          &\left\|\left(\int_0^1 \partial_\tau g_j(U_j^\eps+\sigma\Delta_{\eta e_k} U_j^{\eps})\,d\sigma\right) (V_{j,k}^\eps-\delta_{\eta}^k U_j^{\eps})\right\|_{L^2(\Omega_T)}\\
          &\lesssim \left\|\int_0^1 \partial_\tau g_j(U_j^\eps+\sigma\Delta_{\eta e_k} U_j^{\eps})\,d\sigma\right\|_{L^q(0,T;L^p(\Omega))}\|V_{j,k}^\eps-\delta_{\eta}^k U_j^{\eps}\|_{L^{\infty}(0,T;\widetilde{H}^s(\Omega))}\\
          & \lesssim (\|\partial_\tau g_j(U_j^\eps)\|_{L^q(0,T;L^p(\Omega))}+\delta)\|V_{j,k}^\eps-\delta_{\eta}^k U_j^{\eps}\|_{L^{\infty}(0,T;\widetilde{H}^s(\Omega))}
    \end{split}
    \]
    for any $\delta>0$, $\eta$ sufficiently small and appropriate exponents $q,p\geq 2$. Equation \eqref{eq: estimate for der of gj} and the estimate $\|U_j^{\rho}\|_X\lesssim |\rho|$, which holds for small $\rho$, ensure that
    \begin{equation}
    \label{eq: smallness of RHS for diff quotient}
      \begin{split}
           &\left\|\left(\int_0^1 \partial_\tau g_j(U_j^\eps+\sigma\Delta_{\eta e_k} U_j^{\eps})\,d\sigma\right) (V_{j,k}^\eps-\delta_{\eta}^k U_j^{\eps})\right\|_{L^2(\Omega_T)}\\
           &\lesssim (|\eps|^{\gamma}+\delta)\|V_{j,k}^\eps-\delta_{\eta}^k U_j^{\eps}\|_{L^{\infty}(0,T;\widetilde{H}^s(\Omega))}.
      \end{split}  
    \end{equation}
    For example, in the case $2s<n$ and $\gamma\geq 1/p=s/n$, we use \eqref{eq: estimate for der of gj}, $0<\gamma\leq r\leq \frac{2s}{n-2s}$ and Sobolev's embedding theorem to deduce
    \begin{equation}
    \label{eq: special case}
    \begin{split}
        \|\partial_\tau g_j(U_j^\eps)\|_{L^{n/s}(\Omega)}&\lesssim \sup_{\mu=\gamma,r}\||U_j^\eps|^
        \mu\|_{L^{n/s}(\Omega)}\lesssim\sup_{\mu=\gamma,r}\|U_j^\eps\|^{\mu}_{L^{\mu n/s}(\Omega)}\\
        &\lesssim \sup_{\mu=\gamma,r}\|U_j^\eps\|^{\mu}_{L^{\frac{2n}{n-2s}}(\Omega)}\lesssim |\eps|^{\gamma}.
    \end{split}
    \end{equation}
    For the remaining cases in the range $2s<n$, one can adapt the estimates in \cite[p.~10]{LTZ2} to our setting, which once more yield \eqref{eq: special case}. 
    
    Therefore, by taking $\delta$, $\eps$ and $\eta$, appearing in \eqref{eq: smallness of RHS for diff quotient}, sufficiently small, we can deduce from the energy estimate for \eqref{eq: PDE for W j k} and Step 3 that $\|W_{j,k}^{\eps,\eta}\|_X\to 0$ as $\eta\to 0$ as long as $\eps$ is small enough. This establishes \eqref{eq: derivative wrt eps 1 reformulated} and we can conclude the proof.
     \end{proof}
	
	From the previous lemma we immediately get the following result on the first order linearization of the DN map:
	\begin{lemma}\label{linear1DtN}
		For all $1\leq j\leq 2$, $1\leq k\leq m$, sufficiently small $\eps$, $\varphi=(\varphi_1,\ldots,\varphi_m)\in C_c^{\infty}((\Omega_e)_T)$ and $\psi\in C_c^{\infty}((\Omega_e)_T)$ we have
       \begin{equation}
       \label{eq: first order linearization of DN map}
           \partial_{\eps_k}\langle \Lambda_{q_j,g_j}(\eps\cdot \varphi),\psi\rangle=\langle \Lambda^k_{q_j,\partial_{\tau}g_j}\varphi,\psi\rangle,
       \end{equation}
       where $\Lambda_{q_j,\partial_\tau g_j}^k$ is given by
       \begin{equation}
       \label{eq: first linearization of DN map}
        \langle \Lambda^k_{q_j,\partial_{\tau}g_j}\varphi,\psi\rangle\vcentcolon =\int_{0}^T \langle b(-\Delta)^{s/2}\partial_t \partial_{\eps_k} u_j^{\eps}+c(-\Delta)^{s/2}\partial_{\eps_k} u_j^{\eps},(-\Delta)^{s/2}\psi\rangle_{L^2(\R^n)} \,dt
       \end{equation}
       and so coincides with the DN map of \eqref{vj} (see~\ref{linear DN map}).
	\end{lemma}

    Next, let us note that for $\eps=0$ problem \eqref{vj} reduces to 
    \begin{equation}
        \label{vjj}
	\begin{cases}
		(\partial_t^3+\alpha\partial_t^2+b(-\Delta)^s\partial_t+c(-\Delta)^s+q_j)v=0\ & {\rm in}\ \Omega_T,\\
		v=\varphi_k\ & {\rm in}\ (\Omega_e)_T, \\
		v(0)=\partial_t v(0)=\partial_t^2 v(0)=0\ & {\rm in}\ \Omega,
	\end{cases}
    \end{equation}
	where we used $u_j^0=0$ and $\partial_\tau g_j(0)=0$. Henceforth, we assume that $W_1,W_2\subset\Omega_e$ are two nonempty open sets, the potential $q_2$ is time-reversal invariant and $s>0$ is not an integer. Additionally, we assume that the DN maps satisfy
    \begin{equation}
    \label{eq: equality of nonlin DN maps}
        \langle \Lambda_{q_1,g_1}\phi,\psi\rangle=\langle\Lambda_{q_2,g_2}\phi,\psi\rangle
    \end{equation}
    for all sufficiently small $\phi\in C_c^{\infty}((W_1)_T)$ and $\psi\in C_c^{\infty}((W_2)_T)$. By Lemma \ref{linear1DtN}, we know that \eqref{eq: equality of nonlin DN maps} ensures that
    \begin{equation}
    \label{eq: DN map linearized problem}
        \langle \Lambda_{q_1} \varphi_k,\psi\rangle=\langle \Lambda_{q_2} \varphi_k,\psi\rangle
    \end{equation}
    for all $\varphi_k\in C_c^{\infty}((W_1)_T)$ and $\psi\in C_c^{\infty}((W_2)_T)$. Here, we are using the linearity of the problem \eqref{vjj} to see that all test functions $\varphi_k\in C_c^{\infty}((W_1)_T)$ can be used in \eqref{eq: DN map linearized problem} and not only small exterior conditions. As demonstrated in Proposition \ref{prop: uniqueness of linear perturbations}, this ensures
    \begin{equation}
    \label{eq: same potentials}
        q\vcentcolon = q_1=q_2.
    \end{equation}
    Because solutions to \eqref{vj} are unique, we get
    \begin{equation}
    \label{eq: equality of first derivatives}
        \partial_{\varepsilon_k}u^\varepsilon|_{\varepsilon=0}\vcentcolon=\partial_{\varepsilon_k}u_1^\varepsilon|_{\varepsilon=0}=\partial_{\varepsilon_k}u_2^\varepsilon|_{\varepsilon=0}
    \end{equation}
    for any $1\leq k\leq m$.

	\subsection{Second order linearization and recovery of \texorpdfstring{$\partial_{\tau}^2g_j(0)$}{second der g} }
    \label{sec: second order linearization}
    
    By formal differentiation of \eqref{vj} with respect to $\eps_\ell$, we may expect that $v_{j,k,\ell}^\eps=\partial^2_{\eps_k \eps_\ell}u_j^\eps$ solves
    \begin{equation}
        \label{vj2}
	\begin{cases}
		(\partial_t^3+\alpha\partial_t^2+b(-\Delta)^s\partial_t+c(-\Delta)^s+q+\partial_{\tau}g_j(u_j^\varepsilon))v=-\partial_{\tau}^2g_j(u_j^\varepsilon)v_{j,k}^\eps v_{j,\ell}^\eps & \text{ in } \Omega_T,\\
		v=0 & \text{ in } (\Omega_e)_T, \\
		v(0)=\partial_tv(0)=\partial_t^2v(0)=0 & \text{ in } \Omega.
	\end{cases}
    \end{equation}
	Similarly to Lemma \ref{der1}, we prove that $\partial_{\varepsilon_k\varepsilon_\ell }^2u_j^\varepsilon$ is well-defined and coincides with the unique solution of \eqref{vj2}.               
	\begin{lemma}\label{par2}
		Under the above assumptions, \eqref{vj2} has a unique solution $v_{j,k, \ell}^\varepsilon$ for any sufficiently small $\eps\in\R^m$, $1\leq k,\ell\leq m$ and $j=1,2$. Furthermore, one has
             \begin{equation}
    \label{eq: derivative wrt eps 2}
        v_{j,k,\ell}^\varepsilon=\partial^2_{\varepsilon_k\eps_\ell}u_j^\varepsilon \vcentcolon =\lim\limits_{\eta\to 0}\delta^{\ell}_{\eta}\partial_{\eps_k}u_j^{\eps}\text{ in }X.
    \end{equation}
	\end{lemma}
	\begin{proof} 
    The proof is established in two steps. The first step once again proves the unique solvability of \eqref{vj2} and the second step establishes formula \eqref{eq: derivative wrt eps 2}.\\

    \noindent \textit{Step 1.} By Corollary \ref{wellu} and \eqref{eq: regularity of partial g u eps}, it suffices to show that 
    \begin{equation}
    \label{eq: L2 integrability of source 2nd lin}
        \partial^2_\tau g_j(u_j^{\eps})v_{j,k}^{\eps}v_{j,\ell}^\eps\in L^2(0,T;\widetilde{L}^2(\Omega)).
    \end{equation}
    Before proceeding, let us recall that we suppose 
    \begin{equation}
       |\partial^2_\tau g(x,t,\tau)|\lesssim \begin{cases}
                1,& \text{when }m=2,\\
                1+|\tau|^k,& \text{when }m\geq 3
            \end{cases}
     \end{equation}
        for a.e.~$(x,t)\in\Omega_T$, where $m\geq 2$ and $k\geq 0$. Furthermore, for $2s<n$, we assume that $(m,k)\in \N_{\geq 2}\times \R_{\geq 0}$ fulfill the constraints
        \begin{equation}
        \label{eq: restriction on m,k}
            m\leq \frac{n}{n-2s}\text{ and } k\leq \min\left(\frac{n}{(n-2s)m},\frac{2s}{n-2s}\right).
        \end{equation}
    Without loss of generality, we may assume that $k>0$, as otherwise \eqref{eq: L2 integrability of source 2nd lin} is trivially valid. Additionally, it follows directly from Sobolev's embedding theorem that \eqref{eq: L2 integrability of source 2nd lin} holds in the case $2s\geq n$ and so we only need to consider the case $2s<n$.
    
    If $1\leq p\leq \infty$ satisfies
    \begin{equation}
    \label{eq: Hölder 2nd lin}
       \frac{1}{p}+\frac{n-2s}{n}\leq\frac{1}{2}\text{ and } \partial^2_\tau g_j(u_j^{\eps})\in L^{\infty}(0,T;L^p(\Omega)),
    \end{equation}
    then H\"older's inequality together with Sobolev's embedding theorem yields
    \begin{equation}
    \label{eq: Holder estimate for partial 2 gj}
        \|\partial^2_\tau g_j(u_j^{\eps})v_{j,k}^{\eps}v_{j,\ell}^\eps\|_{L^2(\Omega)}\leq \|\partial^2_\tau g_j(u_j^{\eps})\|_{L^p(\Omega)}\|v_{j,k}^{\eps}\|_{H^s(\R^n)}\|v_{j,\ell}^{\eps}\|_{H^s(\R^n)}
    \end{equation}
    for a.e.~$0<t<T$. For $m=2$, we can clearly choose $p=\infty$ (see~\eqref{eq: restriction on m,k}). In the case $m\geq 3$ we want to choose
    \begin{equation}
    \label{eq: choice of p}
        p=\frac{2n}{k(n-2s)}.
    \end{equation}
    The reason for this choice is twofold. On the one hand, it implies that
    \[
        \frac{1}{p}+\frac{n-2s}{n}=\frac{k(n-2s)}{2n}+\frac{n-2s}{n}\leq \frac{1}{2m}+\frac{1}{m}\leq \frac{1}{6}+\frac{1}{3}=\frac{1}{2}.
    \]
    On the other hand, by Sobolev's embedding theorem and Lemma \ref{lemma: Nemytskii operators}, the condition \eqref{eq: choice of p} guarantees $\partial^2_\tau g_j(u_j^{\eps})\in L^{\infty}(0,T;L^p(\Omega))$ (see~\eqref{eq: boundedness Nemytskii}). In fact, we have the estimate
    \begin{equation}
    \label{eq: estimate partial squared gj}
        \|\partial^2_\tau g_j(u_j^{\eps})\|_{L^{\infty}(0,T;L^p(\Omega))}\lesssim 1+\|u_\eps\|^k_{L^{\infty}(0,T;H^s(\R^n))}.
    \end{equation}
    Combining the aforementioned results, we arrive at the conclusion \eqref{eq: L2 integrability of source 2nd lin} and furthermore there holds
    \begin{equation}
    \label{eq: estimate source 2nd linearization}
    \begin{split}
        &\|\partial^2_\tau g_j(u_j^{\eps})v_{j,k}^{\eps}v_{j,\ell}^\eps\|_{L^2(\Omega_T)}\lesssim \|\partial^2_\tau g_j(u_j^{\eps})v_{j,k}^{\eps}v_{j,\ell}^\eps\|_{L^{\infty}(0,T;L^2(\Omega))}\\
        &\lesssim (1+\|u_\eps\|^k_{L^{\infty}(0,T;H^s(\R^n))})\|v_{j,k}^{\eps}\|_{L^{\infty}(0,T;H^s(\R^n))}\|v_{j,\ell}^{\eps}\|_{L^{\infty}(0,T;H^s(\R^n))}
    \end{split}
    \end{equation}
    (see~\eqref{eq: Holder estimate for partial 2 gj} and \eqref{eq: estimate partial squared gj}).\\
    
    \noindent \textit{Step 2.} First of all, notice that we have
    \begin{equation}
    \label{eq: product difference quotient}
        \delta_{\eta}^k (fg)(\theta_0)=f(\theta_0+\eta e_k)\delta_{\eta}^k g(\theta_0)+g(\theta_0)\delta_{\eta}^k f(\theta_0)
    \end{equation}
    for all functions $f,g\colon \Theta\to Z$, which map some open set $\Theta\subset\R^m$ to a Banach space $Z$ consisting of functions taking values in a commutative algebra, and $\theta_0,\theta_0+\eta e_k\in\Theta$ for some $\eta\neq 0$, $1\leq k\leq m$. In the sequel, we will refer to \eqref{eq: product difference quotient} as the product formula. A straightforward calculation, using \eqref{vj}, \eqref{eq: same potentials}, \eqref{vj2} and \eqref{eq: product difference quotient}, shows that 
    \[
    w^{\eps,\eta}_{j,k,\ell}=v^{\eps}_{j,k,\ell}-\delta^{\ell}_{\eta}v^{\eps}_{j,k}=v^{\eps}_{j,k,\ell}-\delta^{\ell}_{\eta}V^{\eps}_{j,k}\in X
    \]
    solves
    \begin{equation}
         \label{eq: PDE for W j k ell}
            \begin{cases}
		(\partial_t^3+\alpha\partial_t^2+b(-\Delta)^s\partial_t+c(-\Delta)^s+q+\partial_\tau g_j(U_j^{\eps+\eta e_{\ell}}))w=-G_{j,k,\ell}^{\eps,\eta} \ & {\rm in}\ \Omega_T,\\
		w=0\ & {\rm in}\ (\Omega_e)_T, \\
		w(0)=\partial_t w(0)=\partial_t^2w(0)=0\ & {\rm in}\ \Omega,
    	\end{cases}
    \end{equation}
    where
    \begin{equation}
    \label{eq: G jkl}
    \begin{split}
        G_{j,k,\ell}^{\eps,\eta}&\vcentcolon =(\partial_\tau g_j(U_j^{\eps})-\partial_\tau g_j(U_j^{\eps+\eta e_{\ell}}))v^{\eps}_{j,k,\ell}\\
        &\quad +\partial_{\tau}^2 g_j(U_j^{\eps})V^{\eps}_{j,k}V^{\eps}_{j,\ell}-(\delta^{\ell}_\eta \partial_\tau g_j(U_j^{\eps}))V^{\eps}_{j,k}\\
            &=(\partial_\tau g_j(U_j^{\eps})-\partial_\tau g_j(U_j^{\eps+\eta e_{\ell}}))v^{\eps}_{j,k,\ell}\\
            &\quad +\partial_{\tau}^2 g_j(U_j^{\eps})V^{\eps}_{j,k}V^{\eps}_{j,\ell}-\left(\int_0^1 \partial^2_\tau g_j(U_j^\eps+\sigma\Delta_{\eta e_\ell } U_j^{\eps})\,d\sigma\right) V^{\eps}_{j,k}\delta_{\eta}^{\ell} U_j^{\eps}.
    \end{split}
    \end{equation} 
    In \eqref{eq: PDE for W j k ell} and \eqref{eq: G jkl}, we have once again used the notation $V_{j,k}^\eps=v_{j,k}^\eps-\varphi_k$ and $U_{j}^\eps=u_j^\eps -\eps\cdot\varphi$, which solve by construction the PDEs \eqref{eq: PDE for V j k}--\eqref{eq: PDE for U j rho}. The energy estimate for \eqref{eq: PDE for W j k ell} implies that it suffices to show the convergence $G^{\eps,\eta}_{j,k,\ell}\to 0$ in $L^2(\Omega_T)$ as $\eta\to 0$ to conclude \eqref{eq: derivative wrt eps 2}. Here, we make use of the fact that the constant $C>0$ in the energy estimate can be chosen independently of $\eta>0$. In fact, according to Theorem \ref{wellposednessv} the constant $C>0$ in the energy estimate depends increasingly on $\|\partial_\tau g_j(U_j^{\eps+\eta e_{\ell}})\|_{L^{\infty}(0,T;L^p(\Omega))}$, with $p$ satisfying \eqref{conp}, and a uniform estimate of this norm for small $\eta$ is provided by \eqref{eq: special case}. 
    
    Next, let us explain the convergence of $G^{\eps,\eta}_{j,k,\ell}\to 0$ in $L^2(\Omega_T)$ as $\eta\to 0$:
    \begin{enumerate}[(a)]
        \item The $L^2(\Omega_T)$ convergence of the second last line in \eqref{eq: G jkl} is ensured by Lemma \ref{lemma: potential}, the continuity property \eqref{eq: continuity prop partial tau gj} of $\partial_\tau g_j$ and \eqref{eq: convergence of difference of U eps}.
        \item The $L^2(\Omega_T)$ convergence of the last line in \eqref{eq: G jkl} follows from \eqref{eq: convergence of difference of U eps}, \eqref{eq: derivative wrt eps 1 reformulated}, Step 1 and Lemma \ref{lemma: Nemytskii operators}. In fact, we may write
    \begin{equation}
    \label{eq: Gjkl second term goes to zero}
    \begin{split}
         &\partial_{\tau}^2 g_j(U_j^{\eps})V^{\eps}_{j,k}V^{\eps}_{j,\ell}-\left(\int_0^1 \partial^2_\tau g_j(U_j^\eps+\sigma\Delta_{\eta e_\ell } U_j^{\eps})\,d\sigma\right) V^{\eps}_{j,k}\delta_{\eta}^{\ell} U_j^{\eps}\\
         &\quad =\left(\partial_{\tau}^2 g_j(U_j^{\eps})-\int_0^1 \partial^2_\tau g_j(U_j^\eps+\sigma\Delta_{\eta e_\ell } U_j^{\eps})\,d\sigma\right)V^{\eps}_{j,\ell}V^{\eps}_{j,k}\\
         &\quad \quad +\left(\int_0^1 \partial^2_\tau g_j(U_j^\eps+\sigma\Delta_{\eta e_\ell } U_j^{\eps})\,d\sigma\right) (V^{\eps}_{j,\ell}-\delta_{\eta}^{\ell} U_j^{\eps})V^{\eps}_{j,k}
    \end{split}
    \end{equation}
    The third line goes to zero in $L^2(\Omega_T)$, because of \eqref{eq: Holder estimate for partial 2 gj}, \eqref{eq: estimate partial squared gj}, \eqref{eq: derivative wrt eps 1 reformulated} and \eqref{eq: convergence of difference of U eps}. Moreover, from \eqref{eq: convergence of difference of U eps}, \eqref{eq: Holder estimate for partial 2 gj}, Lemma \ref{lemma: Nemytskii operators} and the dominated convergence theorem we infer that the second line converges to zero in $L^2(\Omega_T)$ as $\eta\to 0$.
    \end{enumerate}
    This finishes the proof.
    \end{proof}
	
	Using Lemmas \ref{linear1DtN} \& \ref{par2}, we obtain the following result.
	\begin{lemma}\label{DtN2}
		Suppose that the assumptions of Sections \ref{sec: first order linearization} and \ref{sec: second order linearization} hold. Then for all $1\leq j\leq 2$, $1\leq k,\ell\leq m$, sufficiently small $\eps$, $\varphi=(\varphi_1,\ldots,\varphi_m)\in C_c^{\infty}((\Omega_e)_T)$ and $\psi\in C_c^{\infty}((\Omega_e)_T)$ we have
       \begin{equation}
       \label{eq: second order linearization of DN map}
           \partial_{\eps_k\eps_\ell}^2\langle \Lambda_{q_j,g_j}(\eps\cdot \varphi),\psi\rangle=\partial_{\eps_\ell}\langle \Lambda^k_{q_j,\partial_\tau g_j}\varphi,\psi\rangle=\langle \Lambda^{k,\ell}_{q_j,\partial^2_{\tau}g_j}\varphi,\psi\rangle,
       \end{equation}
       where $\Lambda^{k,\ell}_{q_j,\partial^2_\tau g_j}$ is given by
       \begin{equation}
       \label{eq: second linearization of DN map}
        \langle \Lambda^{k,\ell}_{q_j,\partial^2_{\tau}g_j}\varphi,\psi\rangle\vcentcolon =\int_{0}^T \langle b(-\Delta)^{s/2}\partial_t \partial^2_{\eps_k \eps_\ell} u_j^{\eps}+c(-\Delta)^{s/2}\partial^2_{\eps_k \eps_\ell}  u_j^{\eps},(-\Delta)^{s/2}\psi\rangle_{L^2(\R^n)} \,dt.
       \end{equation}
	\end{lemma}
	
	With the above preparations, we can uniquely determine the term $\partial_{\tau}^2g(0)$ from the (partial) DN map $\Lambda_{q_j,g_j}$. 

    In fact, Lemma \ref{DtN2} and equation \eqref{eq: equality of nonlin DN maps} demonstrate that $w_\eps=\partial_{\varepsilon_k\varepsilon_\ell}^2u_1^\varepsilon-\partial_{\varepsilon_k\varepsilon_\ell}^2u_2^\varepsilon$ satisfies 
    \[
        (b\partial_tw_\eps+cw_\eps)|_{(W_2)_T}=(-\Delta)^s(b\partial_tw_\eps+cw_\eps)|_{(W_2)_T}=0.
    \]
    So, the UCP for the fractional Laplacian guarantees that
    \[
        \begin{cases}
            \partial_t w_\eps=-\frac{c}{b}w_\eps\text{ in }\R^n_T,\\
            w_\eps(0)=0\text{ in }\R^n.
        \end{cases}
    \]
    Hence, we have $w_\eps=0$ in $\R^n_T$. Taking the limit $\eps\to 0$ of \eqref{vj2} reveals that
    \begin{equation}
    \label{eq: equal second order derv}
        (\partial_{\tau}^2g_1(0)-\partial_{\tau}^2g_2(0))(\partial_{\varepsilon_k}u^\varepsilon|_{\varepsilon=0})(\partial_{\varepsilon_{\ell}}u^\varepsilon|_{\varepsilon=0})=0
    \end{equation}
    in $L^2(\Omega_T)$. Here, we once again used Lemma \ref{lemma: Nemytskii operators} and the explanations of Step 1 in the proof of Lemma \ref{par2}. Since $\partial_{\varepsilon_k}u^\varepsilon|_{\varepsilon=0}$ is the unique solution of \eqref{vjj} and $\partial_{\tau}^2g_j(0)\in L^{\infty}(\Omega_T)$, we infer from the Runge approximation (Proposition \ref{Runge}) and $q=q_1=q_2$ that 
    \begin{equation}
    \label{eq: equality of second der at origin}
        \partial_{\tau}^2g_1(0)=\partial_{\tau}^2g_2(0)\text{ in }\Omega_T.
    \end{equation}
	
	\subsection{Higher order linearization and recovery of \texorpdfstring{$\partial_{\tau}^{N} g_j(0)$}{partial g} }
    \label{sec: higher order linearization}
	
	Throughout this section, we assume that $3\leq N\leq m$. Let us start by noting that Faà de Bruno's formula shows that $\partial^N_{\varepsilon_{k_1}\cdots\varepsilon_{k_N}}g_j(u_j^\varepsilon)$ can be computed as
	\begin{equation}
	    \label{parml}
	\begin{split}
	    \partial^N_{\varepsilon_{k_1}\cdots\varepsilon_{k_N}}g_j(u_j^\varepsilon)&=\sum_{\pi=\{\pi^1,\ldots,\pi^M\}\in\Pi_N}\partial^M_\tau g_j(u_j^{\eps})\prod_{\ell=1}^M \partial^{N_{\ell}}_{\eps_{k_{\pi^{\ell}_1}}\cdots \,\eps_{k_{\pi^{\ell}_{N_{\ell}}}}}u_j^{\eps}\\
        &=\partial_\tau g_j(u_j^\eps)\partial^N_{\varepsilon_{k_1}\cdots\varepsilon_{k_N}} u_j^{\eps}\\
        &\quad+\sum_{\pi=\{\pi^1,\ldots,\pi^M\}\in\Pi'_N}\partial^M_\tau g_j(u_j^{\eps})\prod_{{\ell}=1}^M \partial^{N_{\ell}}_{\eps_{k_{\pi^{\ell}_1}}\cdots \,\eps_{k_{\pi^{\ell}_{N_{\ell}}}}}u_j^{\eps}
	\end{split}
	\end{equation}
    (see, for example,~\cite{hardy2006combinatorics}).
    Here, $\Pi_N$ denotes the set of all partitions $\pi=\{\pi^1,\ldots,\pi^M\}$ of the set $\{1,\ldots,N\}$ and we have put $\Pi'_N=\Pi_N\setminus \{\{1,\ldots,N\}\}$. In the above identity, we decomposed the $\ell$-th block of the partition $\pi=\{\pi^1,\ldots,\pi^M\}$ as $\pi^{\ell}=\{\pi^{\ell}_1,\ldots,\pi^{\ell}_{N_{\ell}}\}$, where the natural numbers $1\leq N_1,\ldots,N_M\leq N$ satisfy $\sum_{{\ell}=1}^M N_{\ell}=N$. To shorten some expressions appearing in the text below, we introduce the following abbreviation:
    \begin{equation}
    \label{eq: def of slashed notation}
        \slashed{\partial}^{N-1}_{\eps}g_j(u_j^\varepsilon)\vcentcolon = \sum_{\pi=\{\pi^1,\ldots,\pi^M\}\in\Pi'_N}\partial^M_\tau g_j(u_j^{\eps})\prod_{{\ell}=1}^M \partial^{N_{\ell}}_{\eps_{k_{\pi^{\ell}_1}}\cdots \,\eps_{k_{\pi^{\ell}_{N_{\ell}}}}}u_j^{\eps}.
    \end{equation}
    Hence, if we formally take the $N$-th order derivative $\partial^{N}_{\eps_{k_1},\ldots,\eps_{k_N}}$ of the PDE \eqref{vj2}, then the function $v^{\eps}_{j,k_1,\ldots,k_N}\vcentcolon = \partial^{N}_{\eps_{k_1}\cdots\,\eps_{k_N}}u_j^{\eps}$ should solve
	\begin{equation}
	    \label{vjl}
	\begin{cases}
		(\partial_t^3+\alpha\partial_t^2+b(-\Delta)^s\partial_t+c(-\Delta)^s+q+\partial_\tau g_j(u_j^\eps))v =-\slashed{\partial}^{N-1}_{\eps}g_j(u_j^\varepsilon)\ & {\rm in}\ \Omega_T,\\
		v=0\ & {\rm in}\ (\Omega_e)_T, \\
		v(0)=\partial_tv(0)=\partial_t^2 v(0)=0\ & {\rm in}\ \Omega.
	\end{cases}
	\end{equation}
In a similar spirit as in the cases $N=1$ and $N=2$, we can prove the following lemma.

    \begin{lemma}
    \label{lemma: sol for higher order lin}
        Under the above assumptions, the PDE \eqref{vjl} has a unique solution $v_{j,k_1, \ldots,k_N}^\varepsilon\in X$ for any sufficiently small $\eps\in\R^m$, $1\leq k_1,\ldots,k_N\leq m$ and $j=1,2$. Furthermore, we have
             \begin{equation}
    \label{eq: derivative wrt eps N}
        v_{j,k_1, \ldots,k_N}^\varepsilon=\partial^{N}_{\eps_{k_1}\cdots\,\eps_{k_N}}u_j^{\eps} \vcentcolon =\lim\limits_{\eta\to 0}\delta^{k_N}_{\eta}\partial^{N-1}_{\eps_{k_1}\cdots\,\eps_{k_{N-1}}}u_j^{\eps}\text{ in }X.
    \end{equation}
    \end{lemma}
    \begin{proof}
        We prove the conclusion of Lemma \ref{lemma: sol for higher order lin} in two steps, both of which are based on an induction argument. In view of Lemma \ref{der1} and \ref{par2}, we may assume that the assertions of Lemma \ref{lemma: sol for higher order lin} hold for all $0\leq \ell \leq N-1$. \\

    \noindent \textit{Step 1.} First of all, let us observe that by Corollary \ref{wellu} and \eqref{eq: regularity of partial g u eps}, it suffices to show that $\slashed{\partial}^{N-1}_{\eps}g_j(u_j^\varepsilon)\in L^2(\Omega_T)$, or equivalently, 
    \begin{equation}
    \label{eq: L2 integrability of source Nth lin}
        \partial^M_\tau g_j(u_j^{\eps})\prod_{{\ell}=1}^M \partial^{N_{\ell}}_{\eps_{k_{\pi^{\ell}_1}}\cdots \,\eps_{k_{\pi^{\ell}_{N_{\ell}}}}}u_j^{\eps}\in L^2(\Omega_T)
    \end{equation}
   for all $\pi=\{\pi^1,\ldots,\pi^M\}\in\Pi'_N$, where $2\leq M\leq N$ and $1\leq N_{\ell}\leq N-1$. 
   
   Next, let us recall that our assumptions on the nonlinearity $g_j$ include the following condition: There exist $m\in \N_{\geq 2}$ and $k\geq 0$ such that
       \begin{equation}
        \label{cond higher order case}
            |\partial_{\tau}^{\ell} g_j(x,t,\tau)|\lesssim\begin{cases}
                 1+|\tau|^k, &\text{ when }\ell\leq m-1,\\
                1, &\text{ when }\ell = m
            \end{cases}
        \end{equation}
        for a.e.~$(x,t)\in\Omega_T$ and $\tau\in\mathbb R$. Moreover, if $2s<n$, then the constants $m,k$ need to fulfill the restrictions
        \begin{equation}
        \label{eq: restrictions on m and k higher order}
            m\leq \frac{n}{n-2s} \text{ and } k\leq \min\left(\frac{n}{(n-2s)m},\frac{2s}{n-2s}\right).
        \end{equation}
        Observe that in the range $2s\geq n$, the Sobolev embedding and \eqref{cond higher order case} demonstrate the inclusion \eqref{eq: L2 integrability of source Nth lin}. Hence, it remains to prove the assertion in the case $2s<n$ in which we may assume without loss of generality that $k>0$.

        Suppose for the time being that there exists $1\leq p\leq \infty$ such that
    \begin{equation}
    \label{eq: Hölder Mth lin}
       \frac{1}{p^{*}}\vcentcolon =\frac{1}{p}+M\frac{n-2s}{2n}\leq\frac{1}{2}\text{ and } \partial^M_\tau g_j(u_j^{\eps})\in L^{\infty}(0,T;L^p(\Omega)),
    \end{equation}
    then H\"older's and Sobolev's inequality ensure the estimate
    \begin{equation}
    \label{eq: Holder estimate for partial M gj}
    \begin{split}
         &\Bigg\|\partial^M_\tau g_j(u_j^{\eps})\prod_{{\ell}=1}^M \partial^{N_{\ell}}_{\eps_{k_{\pi^{\ell}_1}}\cdots \,\eps_{k_{\pi^{\ell}_{N_{\ell}}}}}u_j^{\eps}\Bigg\|_{L^2(\Omega_T)}\\
         &\lesssim \Bigg\|\partial^M_\tau g_j(u_j^{\eps})\prod_{{\ell}=1}^M \partial^{N_{\ell}}_{\eps_{k_{\pi^{\ell}_1}}\cdots \,\eps_{k_{\pi^{\ell}_{N_{\ell}}}}}u_j^{\eps}\Bigg\|_{L^{\infty}(0,T;L^{p^{*}}(\Omega))}\\
        &\quad\lesssim \big\|\partial^M_\tau g_j(u_j^{\eps})\big\|_{L^{\infty}(0,T;L^p(\Omega))}\prod_{{\ell}=1}^M\Bigg\|\partial^{N_{\ell}}_{\eps_{k_{\pi^{\ell}_1}}\cdots \,\eps_{k_{\pi^{\ell}_{N_{\ell}}}}}u_j^{\eps}\Bigg\|_{L^{\infty}(0,T;H^s(\R^n))}
    \end{split}
    \end{equation}
    for a.e.~$0<t<T$. For $M= m$, we can take $p=\infty$. In the case $M\leq m-1$ we put
    \begin{equation}
    \label{eq: choice of p higher}
        p=\frac{2n}{k(n-2s)}\geq 2
    \end{equation}
     (see~\eqref{eq: restrictions on m and k higher order}). On the one hand, this choice implies
    \[
        \frac{1}{p}+M\frac{n-2s}{2n}\leq \frac{k(n-2s)}{2n}+M\frac{n-2s}{2n}\leq \frac{1}{2m}+\frac{M}{2 m}= \frac{M+1}{2m}\leq \frac{1}{2}.
    \]
    On the other hand, by Sobolev's embedding theorem and Lemma \ref{lemma: Nemytskii operators}, the condition \eqref{eq: choice of p higher} guarantees $\partial^M_\tau g_j(u_j^{\eps})\in L^{\infty}(0,T;L^p(\Omega))$ (see~\eqref{eq: boundedness Nemytskii}). From \eqref{cond higher order case}, $2\leq M\leq m$, \eqref{eq: choice of p higher} and Sobolev's embedding theorem the estimate
    \begin{equation}
    \begin{split}
         \|\partial^M_\tau g_j(u_j^{\eps})\|_{L^{\infty}(0,T;L^p(\Omega))}\lesssim 1+\|u_\eps\|^k_{L^{\infty}(0,T;H^s(\R^n))},
    \end{split}
    \end{equation}
    where we choose $p$ as explained above. Therefore, we may deduce \eqref{eq: L2 integrability of source Nth lin}, with the explicit bound
    \begin{equation}
    \label{eq: estimate source Nth linearization}
    \begin{split}
        &\Bigg\|\partial^M_\tau g_j(u_j^{\eps})\prod_{{\ell}=1}^M \partial^{N_{\ell}}_{\eps_{k_{\pi^{\ell}_1}}\cdots \,\eps_{k_{\pi^{\ell}_{N_{\ell}}}}}u_j^{\eps}\Bigg\|_{L^2(\Omega_T)}\\
        &\lesssim (1+\|u_\eps\|^k_{L^{\infty}(0,T;H^s(\R^n))})\prod_{{\ell}=1}^M\Bigg\|\partial^{N_{\ell}}_{\eps_{k_{\pi^{\ell}_1}}\cdots \,\eps_{k_{\pi^{\ell}_{N_{\ell}}}}}u_j^{\eps}\Bigg\|_{L^{\infty}(0,T;H^s(\R^n))}<\infty.
    \end{split}
    \end{equation}
    Hence, we have shown that the problem \eqref{vjl} has a unique solution $v^{\eps}_{j,k_1,\ldots,k_N}\in X$. \\
        
 \noindent \textit{Step 2.} In this step, we establish the convergence \eqref{eq: derivative wrt eps N}.
 Next, let us note that \[
w^{\eps,\eta}_{j,k_1,\ldots,k_N}=v^{\eps}_{j,k_1,\ldots,k_N}-\delta^{k_N}_{\eta}v^{\eps}_{j,k_1,\ldots,k_{N-1}}\in X
 \] 
 solves
    \begin{equation}
         \label{eq: PDE for W j k ell Nth order}
            \begin{cases}
		(\partial_t^3+\alpha\partial_t^2+b(-\Delta)^s\partial_t+c(-\Delta)^s+q+\partial_\tau g_j(U_j^{\eps+\eta e_{k_N}}))w=-G_{j,k_1,\ldots,k_N}^{\eps,\eta}  & {\rm in}\, \Omega_T,\\
		w=0 & {\rm in}\, (\Omega_e)_T, \\
		w(0)=\partial_t w(0)=\partial_t^2w(0)=0 & {\rm in}\, \Omega,
    	\end{cases}
    \end{equation}
    where
    \begin{equation}
    \label{eq: G jkl Nth order}
    \begin{split}
        G_{j,k_1,\ldots,k_N}^{\eps,\eta} &\vcentcolon =(\partial_\tau g_j(U_j^{\eps})-\partial_\tau g_j(U_j^{\eps+\eta e_{k_N}}))v_{j,k_1,\ldots,k_N}^{\eps}\\
        &\quad + \slashed{\partial}^{N-1}_{\eps}g_j(U_j^\varepsilon)-\delta_{\eta}^{k_N}\slashed{\partial}^{N-2}_{\eps}g_j(U_j^\varepsilon)\\
        &\quad -(\delta_{\eta}^{k_N}\partial_{\tau}g_j(U_j^{\eps}))v_{j,k_1,\ldots,k_{N-1}}^{\eps}. 
    \end{split}
    \end{equation} 
    Due to the energy estimate for \eqref{eq: PDE for W j k ell Nth order}, it suffices to show that $G^{\eps,\eta}_{j,k,\ell}\to 0$ in $L^2(\Omega_T)$ as $\eta\to 0$ to conclude \eqref{eq: derivative wrt eps 2}. By Step 2 of the proof of Lemma \ref{par2}, we already know that the first line of \eqref{eq: G jkl Nth order} goes to zero in $L^2(\Omega_T)$ as $\eta\to 0$ and hence we only need to prove that the last line of \eqref{eq: G jkl Nth order} goes to zero in $L^2(\Omega_T)$.\\

    To this end, we record one possible generalization of \eqref{eq: product difference quotient}. Let $K\in\N$ and suppose that $(f_k)_{1\leq k\leq K}$ is a family of maps from an open subset $\Theta\subset \R^m$ into a Banach space $Z$, consisting of functions taking values in some (commutative) algebra. If $\theta_0,\theta_0+\eta e_{\ell}\in\Theta$ for $\eta\neq 0$ and $1\leq \ell\leq m$, then we have
    \begin{equation}
    \label{eq: K th order product rule}
        \delta^{\ell}_{\eta}\left(\prod_{k=1}^K f_k\right)(\theta_0)=\sum_{k=1}^K\left(\prod_{j=1}^{k-1}f_j(\theta_0)\right)\delta_{\eta}^{\ell}f_k(\theta_0)\left(\prod_{j=k+1}^{K}f_j(\theta_0+\eta e_{\ell})\right),
    \end{equation}
    where we use the convention that the product over an empty index set is equal to one. Now, using \eqref{eq: def of slashed notation} and the product rules \eqref{eq: product difference quotient} and \eqref{eq: K th order product rule}, we can rewrite the last line of \eqref{eq: G jkl Nth order} as follows:
    \begin{equation}
    \label{eq: new form of last line of G Nth order}
    \begin{split}
            &\slashed{\partial}^{N-1}_{\eps}g_j(U_j^\varepsilon)-\delta_{\eta}^{k_N}\slashed{\partial}^{N-2}_{\eps}g_j(U_j^\varepsilon)-(\delta_{\eta}^{k_N}\partial_{\tau}g_j(U_j^{\eps}))\partial_{\eps_{k_1}\cdots \eps_{k_{N-1}}}^{N-1}U_j^{\eps}\\
            &=\sum_{\pi=\{\pi^1,\ldots,\pi^M\}\in\Pi'_N}\partial^M_\tau g_j(U_j^{\eps})\prod_{{\ell}=1}^M \partial^{N_{\ell}}_{\eps_{k_{\pi^{\ell}_1}}\cdots \,\eps_{k_{\pi^{\ell}_{N_{\ell}}}}}U_j^{\eps}\\
            &\quad -\delta_{\eta}^{k_N}\sum_{\widetilde{\pi}=\{\widetilde{\pi}^1,\ldots,\widetilde{\pi}^{\widetilde{M}}\}\in\Pi'_{N-1}}\partial^{\widetilde{M}}_{\tau} g_j(U_j^{\eps})\prod_{{\ell}=1}^{\widetilde{M}} \partial^{\widetilde{N}_{\ell}}_{\eps_{k_{\widetilde{\pi}^{\ell}_1}}\cdots \,\eps_{k_{\widetilde{\pi}^{\ell}_{\widetilde{N}_{\ell}}}}} U_j^{\eps}\\
            &\quad -(\delta_{\eta}^{k_N}\partial_{\tau}g_j(U_j^{\eps}))\partial_{\eps_{k_1}\cdots \eps_{k_{N-1}}}^{N-1}U_j^{\eps}\\
            &=\sum_{\pi=\{\pi^1,\ldots,\pi^M\}\in\Pi'_N}\partial^M_\tau g_j(U_j^{\eps})\prod_{{\ell}=1}^M \partial^{N_{\ell}}_{\eps_{k_{\pi^{\ell}_1}}\cdots \,\eps_{k_{\pi^{\ell}_{N_{\ell}}}}}U_j^{\eps}\\
            &\quad -\sum_{\widetilde{\pi}=\{\widetilde{\pi}^1,\ldots,\widetilde{\pi}^{\widetilde{M}}\}\in\Pi'_{N-1}}\partial^{\widetilde{M}}_{\tau} g_j(U_j^{\eps+\eta e_{k_N}})\delta_{\eta}^{k_N}\prod_{{\ell}=1}^{\widetilde{M}} \partial^{\widetilde{N}_{\ell}}_{\eps_{k_{\widetilde{\pi}^{\ell}_1}}\cdots \,\eps_{k_{\widetilde{\pi}^{\ell}_{\widetilde{N}_{\ell}}}}} U_j^{\eps}\\
            &\quad -\sum_{\widetilde{\pi}=\{\widetilde{\pi}^1,\ldots,\widetilde{\pi}^{\widetilde{M}}\}\in\Pi'_{N-1}}(\delta^{k_N}_{\eta}\partial^{\widetilde{M}}_{\tau} g_j(U_j^{\eps}))\prod_{{\ell}=1}^{\widetilde{M}} \partial^{\widetilde{N}_{\ell}}_{\eps_{k_{\widetilde{\pi}^{\ell}_1}}\cdots \,\eps_{k_{\widetilde{\pi}^{\ell}_{\widetilde{N}_{\ell}}}}} U_j^{\eps}\\
            &\quad -(\delta_{\eta}^{k_N}\partial_{\tau}g_j(U_j^{\eps}))\partial_{\eps_{k_1}\cdots \eps_{k_{N-1}}}^{N-1}U_j^{\eps}\\
            &=\sum_{\pi=\{\pi^1,\ldots,\pi^M\}\in\Pi'_N}\partial^M_\tau g_j(U_j^{\eps})\prod_{{\ell}=1}^M \partial^{N_{\ell}}_{\eps_{k_{\pi^{\ell}_1}}\cdots \,\eps_{k_{\pi^{\ell}_{N_{\ell}}}}}U_j^{\eps}\\
            &\quad -\sum_{\widetilde{\pi}=\{\widetilde{\pi}^1,\ldots,\widetilde{\pi}^{\widetilde{M}}\}\in\Pi_{N-1}}(\delta^{k_N}_{\eta}\partial^{\widetilde{M}}_{\tau} g_j(U_j^{\eps}))\prod_{{\ell}=1}^{\widetilde{M}} \partial^{\widetilde{N}_{\ell}}_{\eps_{k_{\widetilde{\pi}^{\ell}_1}}\cdots \,\eps_{k_{\widetilde{\pi}^{\ell}_{\widetilde{N}_{\ell}}}}} U_j^{\eps}\\
            &\quad -\sum_{\widetilde{\pi}=\{\widetilde{\pi}^1,\ldots,\widetilde{\pi}^{\widetilde{M}}\}\in\Pi'_{N-1}}\partial^{\widetilde{M}}_{\tau} g_j(U_j^{\eps+\eta e_{k_N}})\delta_{\eta}^{k_N}\prod_{{\ell}=1}^{\widetilde{M}} \partial^{\widetilde{N}_{\ell}}_{\eps_{k_{\widetilde{\pi}^{\ell}_1}}\cdots \,\eps_{k_{\widetilde{\pi}^{\ell}_{\widetilde{N}_{\ell}}}}} U_j^{\eps}\\
            &=\sum_{\pi=\{\pi^1,\ldots,\pi^M\}\in\Pi'_N}\partial^M_\tau g_j(U_j^{\eps})\prod_{{\ell}=1}^M \partial^{N_{\ell}}_{\eps_{k_{\pi^{\ell}_1}}\cdots \,\eps_{k_{\pi^{\ell}_{N_{\ell}}}}}U_j^{\eps}\\
            &\quad -\sum_{\widetilde{\pi}=\{\widetilde{\pi}^1,\ldots,\widetilde{\pi}^{\widetilde{M}}\}\in\Pi_{N-1}}(\delta^{k_N}_{\eta}\partial^{\widetilde{M}}_{\tau} g_j(U_j^{\eps}))\prod_{{\ell}=1}^{\widetilde{M}} \partial^{\widetilde{N}_{\ell}}_{\eps_{k_{\widetilde{\pi}^{\ell}_1}}\cdots \,\eps_{k_{\widetilde{\pi}^{\ell}_{\widetilde{N}_{\ell}}}}} U_j^{\eps}\\
            &\quad -\sum_{\widetilde{\pi}=\{\widetilde{\pi}^1,\ldots,\widetilde{\pi}^{\widetilde{M}}\}\in\Pi'_{N-1}}\partial^{\widetilde{M}}_{\tau} g_j(U_j^{\eps+\eta e_{k_N}})\sum_{L=1}^{\widetilde{M}}\left(\prod_{{\ell}=1}^{L-1} \partial^{\widetilde{N}_{\ell}}_{\eps_{k_{\widetilde{\pi}^{\ell}_1}}\cdots \,\eps_{k_{\widetilde{\pi}^{\ell}_{\widetilde{N}_{\ell}}}}} U_j^{\eps}\right)\\
            &\quad \hspace{3cm}\quad \cdot\delta_{\eta}^{k_N}\partial^{\widetilde{N}_{L}}_{\eps_{k_{\widetilde{\pi}^{L}_1}}\cdots \,\eps_{k_{\widetilde{\pi}^{L}_{\widetilde{N}_{L}}}}} U_j^{\eps}\left(\prod_{{\ell}=L+1}^{\widetilde{M}} \partial^{\widetilde{N}_{\ell}}_{\eps_{k_{\widetilde{\pi}^{\ell}_1}}\cdots \,\eps_{k_{\widetilde{\pi}^{\ell}_{\widetilde{N}_{\ell}}}}} U_j^{\eps+\eta e_{k_N}}\right).
        \end{split}
    \end{equation}
    To proceed, we observe that in every partition $\pi=(\pi^1,\ldots,\pi^M)\in \Pi'_N$ either $\pi^M=\{N\}$ and $\widetilde{\pi}=\{\pi^1,\ldots,\pi^{M-1}\}$ is a partition of $\{1,\ldots,N-1\}$ or $N\in \pi^{\ell}$ for some $1\leq \ell\leq M$ and $|\pi^\ell|\geq 2$. Hence, we may write
    \begin{equation}
    \label{eq: decomp N th order deriv}
        \begin{split}
            &\sum_{\pi=\{\pi^1,\ldots,\pi^M\}\in\Pi'_N}\partial^M_\tau g_j(U_j^{\eps})\prod_{{\ell}=1}^M \partial^{N_{\ell}}_{\eps_{k_{\pi^{\ell}_1}}\cdots \,\eps_{k_{\pi^{\ell}_{N_{\ell}}}}}U_j^{\eps}\\
            &=\sum_{\widetilde{\pi}=\{\widetilde{\pi}^1,\ldots,\widetilde{\pi}^{\widetilde{M}}\}\in\Pi_{N-1}}\partial^{\widetilde{M}+1}_{\tau} g_j(U_j^{\eps})\partial_{\eps_{k_N}}U_j^{\eps}\prod_{{\ell}=1}^{\widetilde{M}} \partial^{\widetilde{N}_{\ell}}_{\eps_{k_{\widetilde{\pi}^{\ell}_1}}\cdots \,\eps_{k_{\widetilde{\pi}^{\ell}_{\widetilde{N}_{\ell}}}}}U_j^{\eps}\\
            &\quad +\sum_{\widetilde{\pi}=\{\widetilde{\pi}^1,\ldots,\widetilde{\pi}^{\widetilde{M}}\}\in\Pi'_{N-1}}\partial^{\widetilde{M}}_{\tau} g_j(U_j^{\eps})\sum_{L=1}^{\widetilde{M}}\left(\prod_{{\ell}=1}^{L-1} \partial^{\widetilde{N}_{\ell}}_{\eps_{k_{\widetilde{\pi}^{\ell}_1}}\cdots \,\eps_{k_{\widetilde{\pi}^{\ell}_{\widetilde{N}_{\ell}}}}} U_j^{\eps}\right)\partial^{\widetilde{N}_{L}+1}_{\eps_{k_{\widetilde{\pi}^{L}_1}}\cdots \,\eps_{k_{\widetilde{\pi}^{L}_{\widetilde{N}_{L}}}}\,\eps_{k_N}} U_j^{\eps}\\
            &\hspace{4cm}\quad\cdot\left(\prod_{{\ell}=L+1}^{\widetilde{M}} \partial^{\widetilde{N}_{\ell}}_{\eps_{k_{\widetilde{\pi}^{\ell}_1}}\cdots \,\eps_{k_{\widetilde{\pi}^{\ell}_{\widetilde{N}_{\ell}}}}} U_j^{\eps}\right).
        \end{split}
    \end{equation}
    Therefore, we obtain
        \begin{equation}
        \label{eq: L2 conv Nth order}
            \begin{split}
                &\slashed{\partial}^{N-1}_{\eps}g_j(U_j^\varepsilon)-\delta_{\eta}^{k_N}\slashed{\partial}^{N-2}_{\eps}g_j(U_j^\varepsilon)-(\delta_{\eta}^{k_N}\partial_{\tau}g_j(U_j^{\eps}))\partial_{\eps_{k_1}\cdots \eps_{k_{N-1}}}^{N-1}U_j^{\eps}\\
                &=\sum_{\widetilde{\pi}=\{\widetilde{\pi}^1,\ldots,\widetilde{\pi}^{\widetilde{M}}\}\in\Pi_{N-1}}\partial^{\widetilde{M}+1}_{\tau} g_j(U_j^{\eps})\partial_{\eps_{k_N}}U_j^{\eps}\prod_{{\ell}=1}^{\widetilde{M}} \partial^{\widetilde{N}_{\ell}}_{\eps_{k_{\widetilde{\pi}^{\ell}_1}}\cdots \,\eps_{k_{\widetilde{\pi}^{\ell}_{\widetilde{N}_{\ell}}}}}U_j^{\eps}\\
                &\quad -\sum_{\widetilde{\pi}=\{\widetilde{\pi}^1,\ldots,\widetilde{\pi}^{\widetilde{M}}\}\in\Pi_{N-1}}(\delta^{k_N}_{\eta}\partial^{\widetilde{M}}_{\tau} g_j(U_j^{\eps}))\prod_{{\ell}=1}^{\widetilde{M}} \partial^{\widetilde{N}_{\ell}}_{\eps_{k_{\widetilde{\pi}^{\ell}_1}}\cdots \,\eps_{k_{\widetilde{\pi}^{\ell}_{\widetilde{N}_{\ell}}}}} U_j^{\eps}\\
                &\quad +\sum_{\widetilde{\pi}=\{\widetilde{\pi}^1,\ldots,\widetilde{\pi}^{\widetilde{M}}\}\in\Pi'_{N-1}}\partial^{\widetilde{M}}_{\tau} g_j(U_j^{\eps})\sum_{L=1}^{\widetilde{M}}\left(\prod_{{\ell}=1}^{L-1} \partial^{\widetilde{N}_{\ell}}_{\eps_{k_{\widetilde{\pi}^{\ell}_1}}\cdots \,\eps_{k_{\widetilde{\pi}^{\ell}_{\widetilde{N}_{\ell}}}}} U_j^{\eps}\right)\partial^{\widetilde{N}_{L}+1}_{\eps_{k_{\widetilde{\pi}^{L}_1}}\cdots \,\eps_{k_{\widetilde{\pi}^{L}_{\widetilde{N}_{L}}}}\,\eps_{k_N}} U_j^{\eps}\\
            &\hspace{4cm}\quad\cdot\left(\prod_{{\ell}=L+1}^{\widetilde{M}} \partial^{\widetilde{N}_{\ell}}_{\eps_{k_{\widetilde{\pi}^{\ell}_1}}\cdots \,\eps_{k_{\widetilde{\pi}^{\ell}_{\widetilde{N}_{\ell}}}}} U_j^{\eps}\right)\\
            &\quad -\sum_{\widetilde{\pi}=\{\widetilde{\pi}^1,\ldots,\widetilde{\pi}^{\widetilde{M}}\}\in\Pi'_{N-1}}\partial^{\widetilde{M}}_{\tau} g_j(U_j^{\eps+\eta e_{k_N}})\sum_{L=1}^{\widetilde{M}}\left(\prod_{{\ell}=1}^{L-1} \partial^{\widetilde{N}_{\ell}}_{\eps_{k_{\widetilde{\pi}^{\ell}_1}}\cdots \,\eps_{k_{\widetilde{\pi}^{\ell}_{\widetilde{N}_{\ell}}}}} U_j^{\eps}\right)\\
            &\quad \hspace{3cm}\quad \cdot\delta_{\eta}^{k_N}\partial^{\widetilde{N}_{L}}_{\eps_{k_{\widetilde{\pi}^{L}_1}}\cdots \,\eps_{k_{\widetilde{\pi}^{L}_{\widetilde{N}_{L}}}}} U_j^{\eps}\left(\prod_{{\ell}=L+1}^{\widetilde{M}} \partial^{\widetilde{N}_{\ell}}_{\eps_{k_{\widetilde{\pi}^{\ell}_1}}\cdots \,\eps_{k_{\widetilde{\pi}^{\ell}_{\widetilde{N}_{\ell}}}}} U_j^{\eps+\eta e_{k_N}}\right).
            \end{split}
        \end{equation}
        Now, one can follow exactly the same arguments as in the $N=2$ case to demonstrate the convergence 
        \[
            \slashed{\partial}^{N-1}_{\eps}g_j(U_j^\varepsilon)-\delta_{\eta}^{k_N}\slashed{\partial}^{N-2}_{\eps}g_j(U_j^\varepsilon)-(\delta_{\eta}^{k_N}\partial_{\tau}g_j(U_j^{\eps}))\partial_{\eps_{k_1}\cdots \eps_{k_{N-1}}}^{N-1}U_j^{\eps}\to 0\text{ in }L^2(\Omega_T)
        \]
        as $\eta\to 0$. To see that the difference of the terms in the second and third line of \eqref{eq: L2 conv Nth order} vanish in the limit $\eta\to 0$, one can use the same computation as in \eqref{eq: G jkl} and \eqref{eq: Gjkl second term goes to zero}. The difference of the fourth and fifth lines of \eqref{eq: L2 conv Nth order} vanishes in the limit $\eta\to 0$, because the derivatives appearing in these terms are at most of order $N-2$ and hence the result follows by the induction hypothesis. This completes the proof.
    \end{proof}

    Using Lemma \ref{DtN2} and \ref{lemma: sol for higher order lin}, we may inductively prove the following result: 
	
	\begin{lemma}
    \label{DtN3}
		Suppose that the assumptions of Sections \ref{sec: first order linearization} - \ref{sec: second order linearization} hold and let $3\leq N\leq m$. Then for all $1\leq j\leq 2$, $1\leq k_1,\ldots,k_N\leq m$, sufficiently small $\eps\in\R^m$, $\varphi=(\varphi_1,\ldots,\varphi_m)\in C_c^{\infty}((\Omega_e)_T)$ and $\psi\in C_c^{\infty}((\Omega_e)_T)$ we have
       \begin{equation}
       \label{eq: Nth order linearization of DN map}
           \partial_{\eps_{k_1}\cdots \eps_{k_N}}^N\langle \Lambda_{q_j,g_j}(\eps\cdot \varphi),\psi\rangle=\partial_{\eps_{k_N}}\langle \Lambda^{k_1,\ldots,k_{N-1}}_{q_j,\partial^{N-1}_\tau g_j}\varphi,\psi\rangle=\langle \Lambda^{k_1,\ldots,k_N}_{q_j,\partial^N_{\tau}g_j}\varphi,\psi\rangle,
       \end{equation}
       where, for all $2\leq M\leq N$, we use the notation
       \begin{equation}
       \label{eq: Nth order linearized DN map}
       \begin{split}
           \langle \Lambda^{k_1,\ldots,k_M}_{q_j,\partial^M_{\tau}g_j}\varphi,\psi\rangle&\vcentcolon = b\int_0^T \langle (-\Delta)^{s/2}\partial_t \partial^M_{\eps_{k_1}\cdots\,\eps_{k_M}} u_j^{\eps},(-\Delta)^{s/2}\psi\rangle_{L^2(\R^n)} \,dt\\
           &\quad +c\int_0^T \langle (-\Delta)^{s/2}\partial^M_{\eps_{k_1}\cdots\,\eps_{k_M}} u_j^{\eps},(-\Delta)^{s/2}\psi\rangle_{L^2(\R^n)} \,dt.
       \end{split}
       \end{equation}
	\end{lemma}
	
	We have now established all the necessary results to uniquely recover the higher order derivatives $\partial_{\tau}^N g_j(0)$, $3\leq N\leq m$. We use once more an induction argument, and hence we may assume that 
	\begin{equation}
	    \label{induc}
	\partial_{\tau}^M g_1(0)=\partial_{\tau}^M g_2(0)
	\end{equation}
	for all $1\leq M \leq N-1$. Now, using condition \eqref{eq: equality of nonlin DN maps}, Lemma \ref{DtN3} and the UCP, we can demonstrate that for all $1\leq M\leq N$ and indices $1
    \leq k_1,\ldots,k_M\leq m$ we have
    \begin{equation}
    \label{eq: equal Mth order derivatives}
        \partial^M_{\eps_{k_1}\cdots \,\eps_{k_M}}u^{\eps}_1=\partial^M_{\eps_{k_1}\cdots \,\eps_{k_M}}u^{\eps}_2
    \end{equation}
    Here, $u_j^{\eps}$ is supposed to have exterior condition $\eps\cdot\varphi$, whereupon $\eps\in\R^m$ has a sufficiently small modulus and $\varphi=(\varphi_1,\ldots,\varphi_m)\in C^{\infty}((W_1)_T)$. We refer the interested reader to Section \ref{sec: second order linearization} for the same argument as $N=2$. Now, we combine \eqref{eq: equal Mth order derivatives} with Lemma \ref{lemma: Nemytskii operators}, Step 1 of the proof of Lemma \ref{lemma: sol for higher order lin} and \eqref{eq: def of slashed notation} to deduce that in the limit $\eps\to 0$ we get
    \begin{equation}
    \label{eq: equal slashed der}
        \slashed{\partial}^{N-1}g_1(0)=\lim_{\eps\to 0}\slashed{\partial}^{N-1}g_1(u_1^{\eps})=\lim_{\eps\to 0}\slashed{\partial}^{N-1}g_2(u_2^{\eps})=\slashed{\partial}^{N-1}g_2(0)\text{ in }L^2(\Omega_T),
    \end{equation}
    where 
    \[
        \slashed{\partial}^{N-1}g_j(0)=\sum_{\pi=\{\pi^1,\ldots,\pi^M\}\in\Pi'_N}\partial^M_\tau g_j(0)\prod_{{\ell}=1}^M \partial^{N_{\ell}}_{\eps_{k_{\pi^{\ell}_1}}\cdots \,\eps_{k_{\pi^{\ell}_{N_{\ell}}}}}u_j^{\eps}\Big|_{\eps=0}
    \]
    for $j=1,2$. Taking into account \eqref{eq: equal Mth order derivatives} and \eqref{induc}, we obtain
    \begin{equation}
    \label{eq: final identity Nth order linearization}
        (\partial^N_\tau g_1(0)-\partial^N_\tau g_2(0))\prod_{\ell=1}^N \partial_{\eps_{k_{\ell}}}u_1^{\eps}\big|_{\eps=0}=0\text{ in }L^2(\Omega_T).
    \end{equation}
    Next, we recall several useful facts: 
    \begin{enumerate}[(A)]
        \item For any $1\leq \ell\leq N$, the function
    \[
    v_{k_{\ell}}\vcentcolon =\partial_{\eps_{k_{\ell}}}u_1^{\eps}\big|_{\eps=0}
    \]
    is the unique solution to the linear PDE
    \begin{equation}
    \label{eq: PDE for Runge argument final}
    \begin{cases}
			(\partial_t^3+\alpha\partial_t^2+b(-\Delta)^s\partial_t+c(-\Delta)^s+q)v=0\ & {\rm in}\ \Omega_T,\\
			v=\varphi_{k_{\ell}}\ & {\rm in}\ (\Omega_e)_T, \\
			v(0)=\partial_tv(0)=\partial_t^2v(0)=0\ & {\rm in}\ \Omega.
		\end{cases}
    \end{equation}
    \item The constants $m\in \N_{\geq 2}$ and $k\geq 0$ are chosen in such a way that 
       \begin{equation}
        \label{cond higher order case final proof}
            |\partial_{\tau}^{M} g_j(x,t,\tau)|\lesssim\begin{cases}
                 1+|\tau|^k, &\text{ when }M\leq m-1,\\
                1, &\text{ when }M = m
            \end{cases}
        \end{equation}
        for a.e.~$(x,t)\in\Omega_T$, $\tau\in\mathbb R$ and, in the range $2s<n$, they obey the constraints
        \begin{equation}
        \label{eq: restrictions on m and k higher order final proof}
            m\leq \frac{n}{n-2s} \text{ and } k\leq \min\left(\frac{n}{(n-2s)m},\frac{2s}{n-2s}\right).
        \end{equation}
    \end{enumerate}
    Next, we want to demonstrate that \eqref{eq: final identity Nth order linearization} implies 
    \begin{equation}
        \label{eq: equal Nth order derivative}
	   \partial_{\tau}^N g_1(0)=\partial_{\tau}^N g_2(0).
    \end{equation}
    To this end, we distinguish several cases:\\

    \noindent\textit{Supercritical range: $2s> n$.} We start by integrating the identity \eqref{eq: final identity Nth order linearization} over $\Omega_T$, which yields
    \begin{equation}
    \label{eq: integral identity Nth term}
         \int_{\Omega_T}(\partial^N_\tau g_1(0)-\partial^N_\tau g_2(0))\prod_{\ell=1}^N v_{k_{\ell}}\,dxdt=0.
    \end{equation}
    Next, assume that $\psi_1,\ldots,\psi_m\in C_c^{\infty}(\Omega_T)$ and let us choose according to the Runge approximation result (Proposition~\ref{Runge}) for each $1\leq k \leq m$ a sequence of exterior values $(\varphi_{k}^{n})_{n\in \N}\in C_c^{\infty}((W_1)_T)$ such that the corresponding solutions $v_k^n$, $n\in\N$, of \eqref{eq: PDE for Runge argument final} satisfy
    \begin{equation}
    \label{eq: Runge approximation final}
        v_k^n\to \psi_k\text{ in }L^2(\Omega_T)
    \end{equation}
    as $n\to\infty$. Thus, we may replace $v_{k_{\ell}}$ in \eqref{eq: integral identity Nth term} by $v_{k_{\ell}}^{n_{\ell}}$, $n_{\ell}\in\N$, to obtain
     \begin{equation}
    \label{eq: second integral identity Nth term}
         \int_{\Omega_T}(\partial^N_\tau g_1(0)-\partial^N_\tau g_2(0))\prod_{\ell=1}^N v^{n_{\ell}}_{k_{\ell}}\,dxdt=0.
    \end{equation} 
    Next, we notice that Sobolev's lemma guarantees $H^s(\R^n)\hookrightarrow L^{\infty}(\R^n)$ for $2s>n$ and \eqref{cond higher order case final proof} implies $\partial^N_\tau g_j(0)\in L^{\infty}(\Omega_T)$ for all $1\leq N\leq m$. Therefore, H\"older's inequality and \eqref{eq: Runge approximation final} show that we can pass to limit $n_N\to \infty$ in \eqref{eq: second integral identity Nth term}, yielding
    \[
    \int_{\Omega_T}(\partial^N_\tau g_1(0)-\partial^N_\tau g_2(0))\left(\prod_{\ell=1}^{N-1} v^{n_{\ell}}_{k_{\ell}}\right)\psi_{k_N}\,dxdt=0.
    \]
    Hence, through a simple iteration we obtain
    \[
        \int_{\Omega_T}(\partial^N_\tau g_1(0)-\partial^N_\tau g_2(0))\prod_{\ell=1}^{N} \psi_{k_{\ell}}\,dxdt=0.
    \]
    By standard arguments this identity can be used to derive the desired conclusion \eqref{eq: equal Nth order derivative}.
    \\

    \noindent\textit{Critical range: $2s=n$.} In the critical range $2s=n$, we can argue exactly in the same way as in the supercritical case, but this time we just use the embedding $H^s(\R^n)\hookrightarrow L^r(\R^n)$ for any $2\leq r<\infty$ (see Proposition \ref{critical}). \\

    \noindent\textit{Subcritical range: $2s< n$.} Let us start by noting that for all $1\leq M\leq m$ there holds
    \[
        \frac{n-2s}{2n}\leq \frac{1}{p_M}\vcentcolon = M\frac{n-2s}{2n}\leq \frac{1}{2}.
    \]
    Thus, if we use the Sobolev embedding $H^s(\R^n)\hookrightarrow L^{\frac{2n}{n-2s}}(\R^n)$, then we can still repeat the arguments of the supercritical case and again conclude the validity of \eqref{eq: equal Nth order derivative}.
    
\section{Proof of Theorem~\ref{thm: polyhom nonlinearities}: Polyhomogeneous nonlinearities} 
\label{sec: polyhom}

In this section, we prove Theorem~\ref{thm: polyhom nonlinearities} that investigates the unique determination of polyhomogeneous nonlinearities of finite order from the related DN map. Hence, for $j=1,2$, we assume that the nonlinearity $g_j$ is given by \eqref{power} with coefficients $(\alpha^{(j)}_k)_{1\leq k\leq L}\subset  L^{\infty}(\Omega_T)$ and exponents $(r_k)_{1\leq k\leq L}\subset (0,\infty)$ satisfying \ref{A4}. The proof, as those presented in \cite[Section 4.2]{LTZ2} and \cite[Section 5.2]{LTZ1}, relies on the first order linearization of the PDE. We also mention the works \cite{griesmaier2022inverse,kow2025inverse}, which apply implicit linearization techniques to study inverse problems for nonlinear PDEs. In contrast to \cite[Section 5.2]{LTZ1}, we deal in this manuscript with a finite sum of homogeneous functions in $\tau$, which allows us to iteratively lower the order of the polyhomogeneous nonlinearity. In this way, we can reduce the number of unknown quantities from $L$ to one, namely $\alpha_L$. The determination of $\alpha_L$ is established through techniques that were already introduced in \cite[Section 4.2]{LTZ2}. 

The proof proceeds as follows. We first analyze the asymptotic expansion of solutions to the semilinear nonlocal MGT equation with respect to a small parameter $\varepsilon$ appearing in the exterior data $\varepsilon \varphi$ (see Claim~6.1). This allows us to apply a first-order linearization argument and to establish the unique determination of the linear potential $q$.

Next, exploiting the structure of the polyhomogeneous nonlinearity together with the Runge approximation property (Proposition~4.3), we recover $\alpha_1$ in $\Omega_T$, and then determine $\alpha_k$ for $k=2,\dots,L-1$ inductively. Finally, using the homogeneity of the last term $\alpha_L |\tau|^{r_L}\tau$ and the argument developed in \cite[Section~4.2]{LTZ2}, we uniquely determine $\alpha_L$.

Before turning our attention to the proof of Theorem~\ref{thm: polyhom nonlinearities}, let us show that any polyhomogeneous nonlinearity $g$ of order $L$, with coefficients $(\alpha_k)_{1\le k\leq L}\subset  L^{\infty}(\Omega_T)$ and exponents $(r_k)_{1\le k\leq L}\subset (0,\infty)$, satisfies the properties \ref{prop f}--\ref{prop f growth} of Theorem \ref{wellnonlinear}. First of all, note that for any $r> 0$, we have
\[
    \frac{d}{d\tau} |\tau|^r \tau=(r+1)|\tau|^r,
\]
which implies
\begin{equation}
\label{eq: growth cond poly}
    |\partial_\tau g|=\left|\sum_{k=1}^L (r_k+1)\alpha_k|\tau|^{r_k}\right|\lesssim \sum_{k=1}^L |\tau|^{r_k}\lesssim |\tau|^{r_1}+|\tau|^{r_L}
\end{equation}
for a.e.~$(x,t)\in \Omega_T$, $\tau\in \R$ and $j=1,2$. This immediately shows that $g_j$ satisfies the conditions \ref{prop f}--\ref{prop f growth} of Theorem \ref{wellnonlinear}. Therefore, we can apply all the results of Section \ref{sec: Well-posedness semilinear problem} to polyhomogeneous nonlinearities of finite order.

\begin{proof}[Proof of Theorem~\ref{thm: polyhom nonlinearities}]

    Here we only consider the case $2s<n$, as the proof can be adapted to the cases $2s\geq n$ without further complications. Let us denote by $u_j^{\eps}$ the unique solution to 
    \begin{equation}
	    \label{eq: PDE poly}
	\begin{cases}
		(\partial_t^3+\alpha\partial_t^2+b(-\Delta)^s\partial_t+c(-\Delta)^s+q_j)u_j^\varepsilon+g_j(u_j^\varepsilon)=0\ & {\rm in}\ \Omega_T,\\
		u^\varepsilon_j=\eps\varphi\ & {\rm in}\ (\Omega_e)_T, \\
		u^\varepsilon_j(0)=\partial_tu^\varepsilon_j(0)=\partial_t^2u^\varepsilon_j(0)=0\ & {\rm in}\ \Omega,
    	\end{cases}
	\end{equation}
    where $\varphi\in C_c^{\infty}((W_1)_T)$, $\eps>0$ is a small positive constant and $j=1,2$. The exterior data $\varphi$ will be fixed later. Using the same arguments as at the end of Section~\ref{sec: second order linearization}, we can use \eqref{eq: equal DN maps poly} and the UCP for the fractional Laplacian to obtain
    \[
        u_\eps\vcentcolon = u_1^{\eps}=u_2^{\eps}.
    \]
    Hence, \eqref{eq: PDE poly} implies that 
    \begin{equation}
    \label{eq: equal potential plus DN map}
        q_1 u_\eps+g_1(u_\eps)=q_2 u_\eps+ g_2(u_\eps)\text{ in }\Omega_T.
    \end{equation}
    Let $v_j$, $j=1,2$, be the unique solution to the problem 
\begin{equation}
\label{eq: linear PDE polyhom}
    \begin{cases}
		(\partial_t^3+\alpha\partial_t^2+b(-\Delta)^s\partial_t+c(-\Delta)^s+q_j)v=0\ & {\rm in}\ \Omega_T,\\
		v=\varphi\ & {\rm in}\ (\Omega_e)_T, \\
		v(0)=\partial_t v(0)=\partial_t^2 v(0)=0\ & {\rm in}\ \Omega.
	\end{cases}
\end{equation}
Next, we establish the following.
\begin{claim}
\label{first linearization poly}
    For any $\varphi\in C_c^{\infty}((W_1)_T)$ and all sufficiently small $0<\eps\leq 1$, we have
    \begin{equation}
    \label{eq: decay estiamte remainder poly}
       \|\partial_t^2(v_j-\eps^{-1}u_\eps)\|_{L^{\infty}(0,T;L^2(\Omega))}+\|v_j-\eps^{-1}u_\eps\|_{W^{1,\infty}(0,T;\widetilde{H}^s(\Omega))}\lesssim \eps^{r_1}.
    \end{equation}
\end{claim}
\begin{proof}
    First, note that
    \[
        R^j_{\eps}\vcentcolon = \eps^{-1}u_\eps -v_j
    \]
    is the unique solution of
    \begin{equation}
        \label{eq: remainder PDE polyhom}
    \begin{cases}
		(\partial_t^3+\alpha\partial_t^2+b(-\Delta)^s\partial_t+c(-\Delta)^s+q_j)R=-\eps^{-1}g_j(u_\eps)\ & {\rm in}\ \Omega_T,\\
		R=0 \ & {\rm in}\ (\Omega_e)_T, \\
		R(0)=\partial_t R(0)=\partial_t^2 R(0)=0\ & {\rm in}\ \Omega
	\end{cases}
    \end{equation}
    for $j=1,2$. From Theorem~\ref{wellposednessv}, we deduce the following energy estimate
    \begin{equation}
    \label{eq: energy estimate R}
    \begin{split}
        &\|\partial_t^2 R^j_{\eps} \|_{L^{\infty}(0,T;L^2(\Omega))}+\|R^j_{\eps}\|_{W^{1,\infty}(0,T;\widetilde H^s(\Omega))} \lesssim  
        \eps^{-1}\|g_j(u_\eps)\|_{L^2(\Omega_T)}.
    \end{split}
    \end{equation}
    Furthermore, we infer from Theorem \ref{wellnonlinear} that $u_\eps$ obeys
    \begin{equation}
    \label{eq: bound for u eps poly}
        \|u_\eps-\eps\varphi\|_X\lesssim \eps
    \end{equation}
    (cf.,~e.g.,~proof of Lemma \ref{der1}). By the fundamental theorem of calculus and \eqref{eq: growth cond poly}, we get
    \[
        |g_j(x,t,\tau)|=\left|\int_0^{\tau} \partial_\tau g_j(x,t,s)\,ds\right|\lesssim |\tau|^{r_1+1}+|\tau|^{r_L+1}
    \]
    for $(x,t,\tau) \in\Omega_T\times \R$ and $j=1,2$. Combining this with Sobolev's embedding $\widetilde{H}^s(\Omega)\hookrightarrow L^{\frac{2n}{n-2s}}(\Omega)$, \eqref{eq: bound for u eps poly} and using the fact that condition \ref{A4} implies
\begin{equation}
\label{eq: restriction on exp poly}
    1\leq  \frac{r_k+1}{r_1+1}\leq \frac{r_{k'}+1}{r_1+1}\leq r_{k'}+1\leq \frac{n}{n-2s}
\end{equation}
for any $1\leq k\leq k'\leq L$, we deduce that
    \[
    \begin{split}
        \|g_j(u_\eps)\|_{L^2(\Omega_T)}&=\|g_j(u_\eps-\eps\varphi)\|_{L^2(\Omega_T)}\\
        &\lesssim \|u_\eps-\eps\varphi\|_{L^{\infty}(0,T;\widetilde{H}^s(\Omega))}^{r_1+1}+\|u_\eps-\eps\varphi\|_{L^{\infty}(0,T;\widetilde{H}^s(\Omega))}^{r_L+1}\\
        &\lesssim \eps^{r_1+1}+\eps^{r_L+1}\lesssim \eps^{r_1+1}
    \end{split}
    \]
    for sufficiently small $0<\eps\leq 1$. Inserting this into \eqref{eq: energy estimate R}, we get the desired decay estimate \eqref{eq: decay estiamte remainder poly}.
\end{proof}
    Hence, through Claim~\ref{first linearization poly} and \ref{linear DN map}--\ref{nonlinear DN map 1} we get the convergence
    \begin{equation}
        \label{eq: equal limits for DNs}
        \langle\Lambda_{q_j} \varphi,\psi\rangle=\lim_{\eps\to 0}\eps^{-1}\langle \Lambda_{q_j,g_j} \eps \varphi,\psi\rangle
    \end{equation}
    for any $\psi\in C_c^{\infty}((W_2)_T)$ and $j=1,2$. So, with the help of \eqref{eq: equal DN maps poly} and Proposition~\ref{prop: uniqueness of linear perturbations}, we get
    \begin{equation}
    \label{eq: equal potentials and linear functions}
        q\vcentcolon = q_1=q_2\text{ in }\Omega_T
    \end{equation}
    and
    \[
        v\vcentcolon = v_1=v_2\text{ and }R_\eps\vcentcolon = R^1_{\eps}=R^2_{\eps}.
    \]
    Therefore, by \eqref{eq: equal potential plus DN map} we get
    \begin{equation}
    \label{eq: equality poly nonlin}
        g_1(u_\eps)=g_2(u_\eps)\text{ in }\Omega_T.
    \end{equation}
Using our assumptions on $g_j$, we deduce that
\begin{equation}
\label{eq: equality poly nonlin 2}
    (\alpha^{(1)}_{1}-\alpha^{(2)}_{1})|u_\eps|^{r_1}u_\eps=P^{(2)}_{1}(u_\eps)-P^{(1)}_{1}(u_\eps)\text{ in }\Omega_T.
\end{equation}
The functions $P^{(j)}_\ell$, which appear on the right hand side of \eqref{eq: equality poly nonlin 2} for $\ell=1$, are given by
\[
    P^{(j)}_{\ell}(x,t,\tau)\vcentcolon =\sum_{k=\ell+1}^{L}\alpha^{(j)}_{k}(x,t)|\tau|^{r_k}\tau
\]
for all $1\leq \ell \leq L-1$, $j=1,2$ and, as usual, we suppress the explicit dependence of $P^{(j)}_\ell$, $j=1,2$, $1\leq \ell\leq L-1$, in \eqref{eq: equality poly nonlin 2} on $(x,t)\in\Omega_T$, because they are regarded as Nemytskii operators. Suppose for the moment that $L\geq 2$. Due to \eqref{eq: equality poly nonlin 2}, we have 
\[
\begin{split}
    |\alpha^{(1)}_{1}-\alpha^{(2)}_{1}||u_{\eps}|^{r_1+1}&=|P^{(1)}_{1}(u_\eps)-P^{(2)}_{1}(u_\eps)|\leq \sum_{k=2}^{L}|\alpha^{(1)}_k-\alpha^{(2)}_k||u_\eps|^{r_k+1}.
\end{split}
\]
This in turn guarantees that
\[
    |\alpha^{(1)}_{1}-\alpha^{(2)}_{1}|^{\frac{1}{r_1+1}}|u_{\eps}|\leq \sum_{k=2}^{L}|\alpha^{(1)}_k-\alpha^{(2)}_k|^{\frac{1}{r_1+1}}|u_\eps|^{\frac{r_k+1}{r_1+1}},
\]
which yields
\[
\begin{split}
  |\alpha^{(1)}_1-\alpha^{(2)}_1|^{\frac{1}{r_1+1}} &=|\alpha^{(1)}_1-\alpha^{(2)}_1|^{\frac{1}{r_1+1}}|1-\eps^{-1}u_\eps+\eps^{-1}u_\eps| \\
 &\leq |\alpha^{(1)}_1-\alpha^{(2)}_1|^{\frac{1}{r_1+1}}|1-\eps^{-1}u_\eps|+\eps^{-1}|\alpha^{(1)}_1-\alpha^{(2)}_1|^{\frac{1}{r_1+1}}|u_\eps| \\
 &\leq |\alpha^{(1)}_1-\alpha^{(2)}_1|^{\frac{1}{r_1+1}}|1-\eps^{-1}u_\eps|+\eps^{-1}\sum_{k=2}^L|\alpha^{(1)}_k-\alpha^{(2)}_k|^{\frac{1}{r_1+1}}|u_\eps|^{\frac{r_k+1}{r_1+1}}.   
\end{split}
\]
Using Sobolev's embedding $\widetilde{H}^s(\Omega)\hookrightarrow L^{\frac{2n}{n-2s}}(\Omega)$, H\"older's inequality and \eqref{eq: restriction on exp poly}, we can conclude that
\begin{equation}
\label{alpha}
    \begin{split}
        \int_{\Omega_T}|\alpha^{(1)}_{1}-\alpha^{(2)}_{1}|^{\frac{1}{r_1+1}}dxdt &\leq \||\alpha^{(1)}_{1}-\alpha^{(2)}_{1}|^{\frac{1}{r_1+1}}\|_{L^2(\Omega_T)}\|\chi_{\Omega}-\eps^{-1}u_\eps\|_{L^2(\Omega_T)} \\
 &\quad+ \eps^{-1}\sum_{k=2}^L\||\alpha^{(1)}_{k}-\alpha^{(2)}_{k}|^{\frac{1}{r_1+1}}\|_{L^2(\Omega_T)}\|u_\eps\|^{\frac{r_k+1}{r_1+1}}_{L^{2\frac{r_k+1}{r_1+1}}(\Omega_T)} \\
 &\leq \||\alpha^{(1)}_{1}-\alpha^{(2)}_{1}|^{\frac{1}{r_1+1}}\|_{L^2(\Omega_T)}\|\chi_{\Omega}-\eps^{-1}u_\eps\|_{L^2(\Omega_T)} \\
 &\quad+C\eps^{-1}\sum_{k=2}^L\|\alpha^{(1)}_{k}-\alpha^{(2)}_{k}|^{\frac{1}{r_1+1}}\|_{L^2(\Omega_T)}\|u_\eps-\eps\varphi\|^{\frac{r_k+1}{r_1+1}}_{L^\infty(0,T;\widetilde H^s(\Omega))}\\
 &\lesssim \|\chi_{\Omega}-\eps^{-1}u_\eps\|_{L^2(\Omega_T)}+\eps^{-1}\sum_{k=2}^L \|u_\eps-\eps\varphi\|^{\frac{r_k+1}{r_1+1}}_{L^\infty(0,T;\widetilde H^s(\Omega))},
    \end{split}
\end{equation}
where the constant appearing in the last estimate only depends on $\Omega,T,n,s,(r_k)_{1\leq k\leq L}$ and
\[
    \mathfrak{a}\vcentcolon = \max_{j=1,2}\max_{1\leq k\leq L}\|\alpha_k^{(j)}\|_{L^{\infty}(\Omega_T)}.
\]
For any given $\delta>0$, we can choose by the Runge approximation (Proposition \ref{Runge}) an exterior condition $\varphi\in C_c^\infty((W_1)_T)$ such that the associated solution $v$ of the linear problem \eqref{eq: linear PDE polyhom} satisfies
\[
    \|\chi_{\Omega}-v\|_{L^2(\Omega_T)}\leq \delta/2.
\]
Thus, Claim \ref{first linearization poly} yields
\[
    \|\chi_{\Omega}-\eps^{-1}u_{\eps}\|_{L^2(\Omega_T)}\leq \|\chi_{\Omega}-v\|_{L^2(\Omega_T)}+\|v-\eps^{-1}u_{\eps}\|_{L^2(\Omega_T)} \leq \delta
\]
for all sufficiently small $0<\eps\leq 1$. Therefore, it follows from \eqref{alpha} and \eqref{eq: bound for u eps poly} that
\begin{equation}
\label{eq: equation for det first coeff}
\begin{split}
    \int_{\Omega_T}|\alpha^{(1)}_1-\alpha^{(2)}_1|^{\frac{1}{r_1+1}}dxdt &\lesssim \delta+\sum_{k=2}^{L}\eps^{\frac{r_k-r_1}{r_1+1}}\lesssim \delta +\eps^{\frac{r_2-r_1}{r_1+1}}
\end{split}
\end{equation}
for any sufficiently small $0<\eps\leq 1$, where the constant in \eqref{eq: equation for det first coeff} depends only on $\Omega,T,n,s,(r_k)_{1\leq k\leq L}$ and $\mathfrak{a}$. The previous estimate ensures that
\[
    \alpha_1\vcentcolon =\alpha^{(1)}_1=\alpha^{(2)}_1\text{ in }\Omega_T,
\]
since the right hand side of \eqref{eq: equation for det first coeff} can be made arbitrarily small.\\

Now, we may deduce from \eqref{eq: equality poly nonlin 2} that 
\begin{equation}
\label{eq: equality poly nonlin 3}
     (\alpha^{(1)}_{2}-\alpha^{(2)}_{2})|u_\eps|^{r_2}u_\eps=P^{(2)}_{2}(u_\eps)-P^{(1)}_{2}(u_\eps)\text{ in }\Omega_T.
\end{equation}
Arguing as in the case $k=1$, we get
\[
\begin{split}
    \int_{\Omega_T}|\alpha^{(1)}_{2}-\alpha^{(2)}_{2}|^{\frac{1}{r_2+1}}dxdt&\lesssim \|\chi_{\Omega}-\eps^{-1}u_\eps\|_{L^2(\Omega_T)}+\eps^{-1}\sum_{k=3}^L \|u_\eps-\eps\varphi\|^{\frac{r_k+1}{r_2+1}}_{L^\infty(0,T;\widetilde H^s(\Omega))}\\
    &\lesssim \delta+\eps^{\frac{r_3-r_2}{r_2+1}}
\end{split}
\]
for any $\delta>0$ and all sufficiently small $\eps>0$. Therefore, we can again conclude that
\begin{equation}
\label{eq: equality poly nonlin 4}
     \alpha_2\vcentcolon =\alpha^{(1)}_{2}=\alpha^{(2)}_{2}\text{ in }\Omega_T.
\end{equation}
Thus, we obtain inductively that
\begin{equation}
\label{eq: equality poly nonlin 5}
     \alpha_k\vcentcolon = \alpha^{(1)}_{k}=\alpha^{(2)}_{k}\text{ in }\Omega_T
\end{equation}
for all $3\leq k\leq L-1$. In the last step, we get the identity
\begin{equation}
\label{eq: last term}
    g^{(L)}(u_\eps)\vcentcolon = (\alpha_L^{(1)}-\alpha_L^{(2)})|u_\eps|^{r_L}u_\eps=0\text{ in }\Omega_T,
\end{equation}
where $g^{(L)}$ is the following $(r_L+1)$-homogeneous function: 
\[
    g^{(L)}(x,t,\tau)\vcentcolon = (\alpha_L^{(1)}-\alpha_L^{(2)})|\tau|^{r_L}\tau.
\]
So, it remains to show that this ensures
\begin{equation}
\label{eq: equality of last terms}
    \alpha_L^{(1)}=\alpha_L^{(2)}\text{ in }\Omega_T,
\end{equation}
which is also the implication we need to establish in the case $L=1$. To this end, we multiply \eqref{eq: equality poly nonlin 5} by $\eps^{-(r_L+1)}$ to get
\[
    g^{(L)}(\eps^{-1}u_{\eps})=0\text{ in }\Omega_T.
\] 
By Lemma \ref{lemma: Nemytskii operators} the function $g^{(L)}$ is continuous as a map from $L^q(0,T; L^p(\Omega))$ to $L^\frac{q}{r_L+1}(0,T; L^\frac{p}{r_L+1}(\Omega))$, where $p=\frac{2n}{n-2s}>\frac{n}{n-2s}\geq r_L+1$ and $q> 1$ is assumed to satisfy $q> r_L+1$. Hence, Claim~\ref{first linearization poly} implies
\[
     g^{(L)}(v) = (\alpha_L^{(1)}-\alpha_L^{(2)})|v|^{r_L}v=0\text{ in }\Omega_T.
\]
Now, we may follow the argument in \cite[Section 4.2]{LTZ2} to conclude that \eqref{eq: equality poly nonlin 5} is true.

This finishes the proof of Theorem~\ref{thm: polyhom nonlinearities}. 
	
\end{proof}

	\section{Proof of Theorem~\ref{thmIP2}: Westervelt-type nonlinearities}
    \label{sec: recovery beta and kappa}
  
In this section, our objective is to uniquely recover the Westervelt parameters $\beta$ and $\kappa$ from the DN maps $\Lambda_\beta$ and $\Lambda_\kappa$, respectively. We are not as detailed as in the previous sections as the strategy is very similar to the case of homogeneous nonlinearities (see~Section~\ref{sec: polyhom} and  \cite[Section 4.2]{LTZ2}). More specifically, we first prove the asymptotics of solutions to the nonlocal JMGT equation with respect to the small parameter $\varepsilon$. Then we can apply the first order linearization to deduce a (integral) identity related to the nonlinear coefficient $\beta$ (resp. $\kappa$). Lastly, we apply the Runge approximation (Proposition \ref{Runge}) to prove the unique determination of $\beta$ (resp. $\kappa$).  

Before proving this theorem,  we highlight that for nonlocal JMGT equations with Westervelt-type nonlinearities, we always assume that $2s>n$ and the Westervelt parameters $(\beta,\kappa)$ satisfy
    \begin{equation}
        \label{eq: Westervelt regularity inverse}
        (\beta,\kappa)\in W^{2,\infty}(0,T;L^\infty(\Omega))\times W^{1,\infty}(0,T;L^\infty(\Omega)).
    \end{equation}

\begin{proof}[Proof of Theorem~\ref{thmIP2}]
	
	\noindent \ref{Westervelt}: For any sufficiently small $\eps>0$, $\varphi\in C_c^{\infty}((W_1)_T)$ and $j=1,2$, we denote by $u_{j,\varphi}^{\eps}$ the unique solution to
	\begin{equation}
	    \label{sysbeta}
	\begin{cases}
		(\partial_t^3+\alpha\partial_t^2+b(-\Delta)^s\partial_t+c(-\Delta)^s)u=\partial_t^2(\beta_j u^2) & \text{ in } \Omega_T,\\
		u=\eps\varphi & \text{ in }(\Omega_e)_T, \\
		u(0)=\partial_t u(0)=\partial_t^2 u(0)=0  & \text{ in }\Omega
	\end{cases}
	\end{equation}
    (see Theorem~\ref{well-Wel}). As in the case of polyhomogeneous nonlinearities, \eqref{eq: equal DN Westervelt 1} and the UCP demonstrate that $u^{\eps}_{\varphi} \vcentcolon = u^{\eps}_{1,\varphi}=u^{\eps}_{2,\varphi}$. Therefore, as $u^{\eps}_{\varphi}$ solves \eqref{sysbeta}, we have
    \[
        \partial_t^2((\beta_1-\beta_2) (u^{\eps}_{\varphi})^2)=0\text{ in }\Omega_T.
    \]
    Since $u^{\eps}_{\varphi}\in W^{2,\infty}(0,T;L^2(\R^n))$ satisfies $\partial^{\ell}_t u^{\eps}_{\varphi}(0)=0$ in $\Omega$ for all $0\leq \ell\leq 2$, we obtain 
    \begin{equation}
    \label{eq: beta identity}
        (\beta_1-\beta_2) (u^{\eps}_{\varphi})^2=0\text{ in }\Omega_T.
    \end{equation}
    Furthermore, observe that we have the asymptotics
    \begin{equation}
    \label{eq: linearization}
        u^{\eps}_{\varphi}-\eps v=o(\eps)\text{ in } X=W^{1,\infty}(0,T;\widetilde{H}^s(\Omega))\cap W^{2,\infty}(0,T;\widetilde{L}^2(\Omega))
    \end{equation} 
    as $\eps\to 0$, where $v$ denotes the unique solution to
	\begin{equation}
     \label{eq: 1st lin Westervelt}
	\begin{cases}
		(\partial_t^3+\alpha\partial_t^2+b(-\Delta)^s\partial_t+c(-\Delta)^s)v= 0& \text{ in } \Omega_T,\\
		v=\varphi  & \text{ in } (\Omega_e)_T, \\
		v(0)=\partial_t v(0)=\partial_t^2 v(0)=0  & \text{ in } \Omega
	\end{cases}
	\end{equation}
    (see, for example, Claim~\ref{first linearization poly}). Now, \eqref{eq: beta identity} and \eqref{eq: linearization} yield
    \begin{equation}
        (\beta_1-\beta_2) v^2=0\text{ in }\Omega_T
    \end{equation}
    and hence by polarization we get
    \[
        (\beta_1-\beta_2)v_1 v_2=0\text{ in }\Omega_T,
    \]
    where $v_j$ solves \eqref{eq: 1st lin Westervelt} and has exterior condition $\varphi_j$ for $j=1,2$. Thus, the Runge approximation property (Proposition~\ref{Runge}) shows that 
    \[
        \beta_1=\beta_2\text{ in }\Omega_T.
    \]

\noindent \ref{Westervelt 2}: For nonlocal JMGT equations with Westervelt-type nonlinearities of the form $\partial_t(\kappa(\partial_tu)^2)$, we start similarly as for the Westervelt-type nonlinearities having the structure $\partial_t^2(\beta u^2)$. More precisely, we use first the equality of DN maps \eqref{eq: equal DN Westervelt 2}, the UCP and the JMGT equations to deduce that
\begin{equation}
\label{eq: first identity westervelt 2}
\partial_t[(\kappa_1-\kappa_2)(\partial_t u^{\eps}_{\varphi})^2]=0\text{ in }\Omega_T.
\end{equation}
Here, $\eps>0$ is a small parameter, $\varphi\in C_c^{\infty}((W_1)_T)$ is a given exterior condition and $u^{\eps}_{\varphi}\vcentcolon = u^{\eps}_{1,\varphi}=u^{\eps}_{2,\varphi}$, in which $u^{\eps}_{j,\varphi}$ denotes the unique solution to
\begin{equation}
    \label{eq: westervelt inverse prob}
	\begin{cases}
		(\partial_t^3+\alpha\partial_t^2+b(-\Delta)^s\partial_t+c(-\Delta)^s)u=\partial_t(\kappa_j(\partial_tu)^2) & \text{ in } \Omega_T,\\
		u=\eps\varphi & \text{ in }(\Omega_e)_T, \\
		u(0)=\partial_t u(0)=\partial_t^2 u(0)=0  & \text{ in }\Omega
	\end{cases}
	\end{equation}
that is provided by Theorem~\ref{well-Wel}. Once again, if we integrate \eqref{eq: first identity westervelt 2}, then we get
\begin{equation}
(\kappa_1-\kappa_2)(\partial_t u^{\eps}_{\varphi})^2=0\text{ in }\Omega_T.
\end{equation}
As in the previous case, we have the asymptotics \eqref{eq: linearization}, which yields
\[
     (\kappa_1-\kappa_2)(\partial_t v)^2=0\text{ in }\Omega_T
\]
and thus by polarization we obtain
\begin{equation}
\label{eq: almost final id for kappa}
    (\kappa_1-\kappa_2)\partial_t v_1\partial_t v_2=0,
\end{equation}
where $v$ and $v_j$, $j=1,2$, solve \eqref{eq: 1st lin Westervelt} with exterior conditions $\varphi$  and $\varphi_j$, respectively. Next, let $\psi_1,\psi_2\in C_c^{\infty}(\Omega\times (0,T])$. According to the Runge approximation (Proposition~\ref{Runge}), we may choose sequences $(\varphi_j^k)_{k\in\N}\subset C_c^{\infty}((W_1)_T)$, $j=1,2$, such that the related solutions $v_j^k$, $k\in\N$, $j=1,2$, of \eqref{eq: 1st lin Westervelt} satisfy $v_j^k|_{\Omega_T}\to \psi_j$ in $L^2(\Omega_T)$ for $j=1,2$. By an integration by parts, we immediately see that 
\begin{equation}
\label{eq: runge time der}
\begin{split}
    \lim_{k\to\infty}\int_{\Omega_T}(\partial_t v_j^k)w\,dxdt&=- \lim_{k\to\infty}\int_{\Omega_T} v_j^k\partial_t w\,dxdt=-\int_{\Omega_T}\psi_j\partial_t w\,dxdt\\
    &=\int_{\Omega_T}(\partial_t \psi_j)w\,dxdt,
\end{split}
\end{equation}
for any $w\in H^1(0,T;\widetilde{L}^2(\Omega))$ and $w(T)=0$. Note that the conditions $v_j^k(0)=\psi_j(0)=w(T)=0$ ensure that we do not get boundary terms in the integration by parts formulas. Thus, from \eqref{eq: almost final id for kappa} and \eqref{eq: runge time der}, we may deduce that
\[
    \begin{split}
        0&=\lim_{\ell\to\infty}\lim_{k\to\infty}\int_{\Omega_T}(\kappa_1-\kappa_2)(\partial_t v_1^k)(\partial_t v^{\ell}_2)\eta\,dxdt\\
        &=\lim_{\ell\to\infty}\int_{\Omega_T}(\kappa_1-\kappa_2)(\partial_t \psi_1)(\partial_t v^{\ell}_2)\eta\,dxdt\\
        &=\int_{\Omega_T}(\kappa_1-\kappa_2)(\partial_t \psi_1)(\partial_t \psi_2)\eta\,dxdt
    \end{split}
\]
for any $\eta\in C_c^{\infty}((0,T))$. Next, by inserting
\[
    \psi_j(x,t)=\int_0^t \Psi_j(x,\sigma)\,d\sigma\in C_c^{\infty}(\Omega\times (0,T]),
\]
for some functions $\Psi_1,\Psi_2\in C_c^{\infty}(\Omega_T)$, into the above formula, we get
\[
    0=\int_{\Omega_T}(\kappa_1-\kappa_2)\Psi_1\Psi_2\eta\,dxdt.
\]
This clearly implies that $\kappa_1=\kappa_2$ in $\Omega_T$ and hence we can conclude the proof.
\end{proof}	
	
\section{Conclusion and future work}	
\label{sec: concluding remarks}

In this paper, we have investigated the Calderón problem for semilinear nonlocal MGT equations as well as for nonlocal JMGT equations of Westervelt type. For the semilinear MGT case, our main tool was a higher-order linearization method, whereas for polyhomogeneous nonlinearities and JMGT equations such an approach was not required. We emphasize that our results for semilinear MGT equations remain valid for second-order in time nonlocal wave equations arising from the same class of nonlinearities. Moreover, recalling that in the vanishing relaxation-time limit $\tau \to 0$ the (J)MGT equation converges formally to a second-order nonlocal wave equation (see Section~\ref{sec: intro}), the present work may be viewed as a natural extension of the results in \cite{KMSK,PZ1,LTZ2,LTZ1}.

As observed in the recovery of polynomial-type nonlinearities, the main analytical difficulties arise in higher space dimensions. This is largely due to the apparent inaccessibility of uniform $L^\infty_{t,x}$-bounds for solutions of \eqref{sysmgt} in high dimensions. It is therefore natural to ask whether the methods of Section~\ref{sec: proof of theorem 1.5} can be refined so as to extract additional information about \( g \) from the DN map beyond the derivatives $(\partial_\tau^{\ell} g(x,t,0))_{2 \le \ell \le m}$. One possible direction is to linearize the PDE around nontrivial background solutions rather than around the zero solution. A strategy of this type was employed in \cite{Johansson} to recover low-regularity coefficients in semilinear Schrödinger equations. For reasons of conciseness, we do not pursue this here, but this remains an intriguing direction for future research.

Other natural directions for future research include identifying the optimal class of recoverable polynomial-type nonlinearities $g$, and establishing stability estimates for the inverse problem within suitable subclasses of potentials and nonlinearities (see, for example, \cite{KLW,MN,ARMS}).
In particular, our forthcoming work will develop stable recovery results for low-regularity potentials in nonlocal wave equations.

Advances on these questions would substantially deepen the current understanding of inverse problems for nonlinear nonlocal wave models, and would further clarify the interplay between nonlinear wave dynamics and spatial nonlocality.

Finally, the next two subsections address the following topics:
\begin{enumerate}[(i)]
    \item the one-dimensional measurement problem for nonlocal MGT equations, and
    \item nonlocal generalizations of the JMGT equation of Kuznetsov type (see \eqref{eq: JMGT eq}). 
\end{enumerate}

\subsection{Unique determination results by one-dimensional measurements}
\label{IP-one dimensional measurement}

In this section, we address the one-dimensional measurement problem for the semilinear nonlocal MGT equation.

\medskip
\noindent
\textit{Polynomial-type nonlinearities.}
Assume that the linear potential $q \in L^p(\Omega)\cap C(\Omega)$, with $p$ obeying \eqref{conp}, and that the nonlinearity $g\colon \Omega\times \mathbb{R} \to \mathbb{R}$ is time-independent and satisfies assumptions \ref{A1}--\ref{A3}, as well as 
$\partial_\tau^k g(x,0)\in C(\overline{\Omega})$ for all $1\le k \le m$.
Our goal is to show that $q$ and the derivatives $\{\partial_\tau^k g(x,0)\}_{k=1}^m$ can be uniquely determined from one-dimensional measurements of the form
\[
    \left.\Lambda_{q,g}(\varepsilon\varphi)\right|_{(W_2)_T},
\]
for all sufficiently small $\varepsilon$ and a fixed nonzero $\varphi \in C_c^\infty((W_1)_T)$.

Accordingly, we assume that for $j=1,2$ the potentials $q_j$ and nonlinearities $g_j$ satisfy the above conditions and that the identity
\begin{equation}
\label{eq:DN-qg-one-dim}
    \left.\Lambda_{q_1,g_1}(\varepsilon\varphi)\right|_{(W_2)_T}
    = \left.\Lambda_{q_2,g_2}(\varepsilon\varphi)\right|_{(W_2)_T}
    \quad\text{in } (W_2)_T
\end{equation}
holds for all sufficiently small $\varepsilon$. Arguing as in Section~\ref{sec: first order linearization}, Lemma~\ref{linear1DtN} yields
\begin{equation}
\label{eq: equality linear DN maps}
    \left.\Lambda_{q_1}\right|_{(W_2)_T}\varphi = \left.\Lambda_{q_2} \varphi\right|_{(W_2)_T}
    \quad\text{in } (W_2)_T,
\end{equation}
where $\Lambda_{q_j}$ denotes the DN map associated with the linear problem
\begin{equation}
\label{eq: linear problem sec 8}
\begin{cases}
(\partial_t^3 + \alpha \partial_t^2 
    + b(-\Delta)^s \partial_t 
    + c(-\Delta)^s + q_j)v = 0 & \text{in } \Omega_T,\\[1mm]
v = \varphi & \text{in } (\Omega_e)_T,\\
v(0)=\partial_t v(0)=\partial_t^2 v(0)=0 & \text{in } \Omega.
\end{cases}
\end{equation}
By \eqref{eq: equality linear DN maps}, the equivalence of DN maps \eqref{eq: equivalence DN maps}, and the UCP, we obtain
\[
    v \vcentcolon = v_1 = v_2.
\]
Subtracting the equations for $v_1$ and $v_2$ gives
\begin{equation}
\label{eq:q1-q2 v}
    (q_1 - q_2) v = 0
    \quad\text{in } \Omega_T.
\end{equation}
If $q_1(x_0)\neq q_2(x_0)$ for some $x_0\in\Omega$, then continuity of $q_j$ yields the existence of a ball $B_\delta(x_0)$ where $q_1\neq q_2$, and hence $v=0$ in $(B_\delta(x_0))_T$.  
By \eqref{eq: linear problem sec 8} and the UCP, we deduce that $v\equiv 0$ in $\mathbb{R}^n_T$, contradicting the fact that $\varphi$ is nonzero.  
Thus, $q_1(x_0) = q_2(x_0)$.
Since $x_0\in\Omega$ is arbitrary, we conclude that $q_1=q_2$ in $\Omega$.
The same argument shows that if \eqref{eq:DN-qg-one-dim} (or \eqref{eq: equality linear DN maps}) holds for two time-dependent potentials $q_1,q_2 \in L^\infty(0,T;L^p(\Omega))\cap C(\Omega_T)$, then
\[
    q_1 = q_2 \quad \text{in }\operatorname{supp}\varphi.
\]

To recover the higher derivatives $\partial_\tau^k g(x,0)$ for $k=2,\dots,m$, we apply the higher-order linearization method.  
For instance, using the argument of Section~\ref{sec: second order linearization} with $\varphi_1=\varphi_2=\varphi$, Lemma~\ref{DtN2} and \eqref{eq: equal second order derv} give
\[
    (\partial_{\tau}^2 g_1(0) - \partial_{\tau}^2 g_2(0))\, v^2 = 0 
    \quad\text{in }\Omega_T,
\]
where $v$ solves \eqref{eq: linear problem sec 8}.  
The higher derivatives follow similarly (see Section~\ref{sec: higher order linearization}).

\medskip
\noindent
\textit{Polyhomogeneous nonlinearities.}
We fix again a nonzero exterior condition $\varphi \in C_c^\infty((W_1)_T)$ and assume that \eqref{eq:DN-qg-one-dim} holds for all sufficiently small $\varepsilon$.  
Here, we suppose that $q_1, q_2 \in L^p(\Omega)\cap C(\Omega)$, with $p$ obeying \eqref{conp}, and let the two polyhomogeneous nonlinearities be given by  
\[
    g_j(x,\tau)
    = \sum_{k=1}^L \alpha_k^{(j)}(x)\, |\tau|^{r_k}\tau,
    \qquad j=1,2,
\]
where $\alpha_k^{(j)} \in C(\overline{\Omega})$, $0 < r_1 < \cdots < r_L \le 1$, and we suppose that the exponents $r_k$ satisfy the condition~\eqref{conr}.

By the proof of Theorem~\ref{thm: polyhom nonlinearities}--in particular, by \eqref{eq: equal limits for DNs}--we already know that
\[
    \langle \Lambda_{q_j}\varphi, \psi \rangle
    = \lim_{\varepsilon\to 0} 
      \varepsilon^{-1}
      \langle \Lambda_{q_j,g_j} (\varepsilon\varphi), \psi \rangle
\]
for all $\psi \in C_c^\infty((W_2)_T)$ and $j=1,2$.  
Combining this with \eqref{eq:DN-qg-one-dim} and the argument above yields
\[
    q \vcentcolon= q_1 = q_2
    \quad\text{in }\Omega.
\]
Thus, \eqref{eq: equal potential plus DN map} implies that
\[
    g_1(u_\varepsilon) = g_2(u_\varepsilon)
    \quad\text{in }\Omega_T,
\]
where $u_\varepsilon = u_\varepsilon^{(1)} = u_\varepsilon^{(2)}$, and $u_{\eps}^{(j)}$ denotes the unique solution to
\[
\begin{cases}
(\partial_t^3+\alpha\partial_t^2+b(-\Delta)^s\partial_t 
 + c(-\Delta)^s+q)u + g_j(u)=0 & \text{in }\Omega_T,\\
u = \varepsilon\varphi & \text{in }(\Omega_e)_T,\\
u(0)=\partial_t u(0)=\partial_t^2 u(0)=0 & \text{in }\Omega
\end{cases}
\]
for $j=1,2$. Note that the equality $u_\varepsilon^{(1)} = u_\varepsilon^{(2)}$ follows again from \eqref{eq:DN-qg-one-dim} together with the UCP.

Combining Claim~\ref{first linearization poly} with \cite[Eq.~(5.21)]{LTZ1}, we obtain
\[
    (\alpha_1^{(1)} - \alpha_1^{(2)})\, |v|^{r_1+1} v = 0
    \quad\text{in }\Omega_T,
\]
where $v$ solves \eqref{eq: linear problem sec 8}.  
Following the strategy above, we obtain that $\alpha_1^{(1)}=\alpha_1^{(2)}$ in $\Omega$.  
An induction argument identical to that in \cite[Proof of Theorem~1.5]{LTZ1} then shows that
\[
    \alpha_k^{(1)} = \alpha_k^{(2)}
    \qquad\text{for all } 2\le k\le L.
\]
Hence, both the potential $q_j$ and the full finite-order polyhomogeneous nonlinearity $g_j$ are uniquely determined by one-dimensional measurements.

\subsection{Nonlocal JMGT equations of Kuznetsov-type}
\label{subsec:nonlocal-gradient}

The aim of this section is to discuss nonlocal generalizations of the JMGT equation of Kuznetsov-type, that is, of the evolution equation
\begin{equation}
\label{eq:non-local-JMGT}
\left(\tau\partial_t^3+\partial_t^2-b\Delta\partial_t-c^2\Delta\right)u
    =\partial_t\left(\frac{1}{c^2}\frac{B}{2A}(\partial_tu)^2+|\nabla u|^2\right)
\end{equation}
(see~Eq.~\eqref{eq: JMGT eq}). As explained in Section~\ref{sec: intro}, neglecting local nonlinear effects in \eqref{eq:non-local-JMGT} leads to the Westervelt-type equation
\begin{equation}
\label{eq: Westervelt-type sec 8}
    \left(\tau\partial_t^3+\partial_t^2-b\Delta\partial_t-c^2\Delta\right)u
    =\partial_t\bigl(\kappa (\partial_tu)^2\bigr)
\end{equation}
(in the velocity potential formulation), and further dropping all quadratic terms yields the MGT equation
\begin{equation}
\label{eq: MGT eq sec 8}
    \left(\tau\partial_t^3+\partial_t^2-b\Delta\partial_t-c^2\Delta\right)u=0
\end{equation}
(see \eqref{eq: westervelt JMGT eq} and \eqref{eq: MGT eq}). In the present article we have studied inverse problems for nonlocal analogues of \eqref{eq: Westervelt-type sec 8} and for semilinear perturbations of \eqref{eq: MGT eq sec 8}. To fully extend the Kuznetsov-type JMGT model to a nonlocal setting, it remains to replace the local gradient $\nabla u$ in \eqref{eq:non-local-JMGT} by a suitable nonlocal gradient. We therefore introduce two natural, but non-equivalent, notions of $s$-fractional gradients. For simplicity, we here restrict our discussions to the cases $0<s<1$.
\begin{definition}
\label{def:fra-gradient-1}
Let $0<s<1$. For $u\in H^s(\mathbb{R}^n)$, the \emph{fractional gradient of order $s$} of $u$ is the map
\[
\nabla^s u\colon \mathbb{R}^{2n}\to\mathbb{R}^n
\]
defined by
\begin{equation}
\nabla^s u(x,y)\vcentcolon=
    \sqrt{\frac{C_{n,s}}{2}}\,
    \frac{u(x)-u(y)}{|x-y|^{\frac n2+s+1}}(x-y),
\end{equation}
where $C_{n,s}=\frac{4^s\Gamma(\frac n2+s)}{\pi^{\frac n2}|\Gamma(-s)|}>0$.
\end{definition}
By \cite[Proposition~3.4]{EPV} and the continuity of the fractional Laplacian, we obtain
\begin{equation}
\|\nabla^s u\|_{L^2(\mathbb{R}^{2n};\mathbb{R}^n)}
    =\|(-\Delta)^{s/2} u\|_{L^2(\mathbb{R}^n)}
    \le \|u\|_{H^s(\mathbb{R}^n)}
\end{equation}
for all $u\in H^s(\mathbb{R}^n)$. Hence
\[
    \nabla^s\colon H^s(\mathbb{R}^n)\to L^{2}(\mathbb{R}^{2n};\mathbb{R}^n)
\]
is a bounded linear operator, and we can define the \emph{fractional divergence}
\[
    \Div_s\colon L^2(\mathbb{R}^{2n};\mathbb{R}^n)\to H^{-s}(\mathbb{R}^n)
\]
as its formal adjoint, that is,
\[
    \langle \Div_s u,v\rangle_{H^{-s}(
    \mathbb{R}^n)\times H^s(\mathbb{R}^n)}
    =\langle u,\nabla^s v\rangle_{L^2(\mathbb{R}^{2n};\mathbb{R}^n)}
\]
for all $u\in L^{2}(\mathbb{R}^{2n};\mathbb{R}^n)$ and $v\in H^s(\mathbb{R}^n)$. Combining \cite[Proposition~3.4]{EPV} with the polarization identity yields
\begin{equation}
\label{eq: div grad equal frac lap}
    \begin{split}
        \langle \Div_s(\nabla^s u),v\rangle_{H^{-s}(\mathbb{R}^n)\times H^s(\mathbb{R}^n)}
            &=\langle \nabla^s u,\nabla^s v\rangle_{L^{2}(\mathbb{R}^{2n};\mathbb{R}^n)}\\
            &=\langle (-\Delta)^{s/2}u,(-\Delta)^{s/2}v\rangle_{L^2(\mathbb{R}^n)}\\
            &=\langle (-\Delta)^s u,v\rangle_{H^{-s}(\mathbb{R}^n)\times H^s(\mathbb{R}^n)}
    \end{split}
\end{equation}
for all $u,v\in H^s(\mathbb{R}^n)$. Moreover, $\nabla^s$ satisfies the product rule
\begin{equation}
\label{eq: product rule frac grad}
    \nabla^s(uv)(x,y)
        =v(y)\,\nabla^s u(x,y)+u(x)\,\nabla^s v(x,y)
\end{equation}
for all $u,v\in H^s(\mathbb{R}^n)$ and $x,y\in\mathbb{R}^n$.
For further properties of $\nabla^s$ and $\Div_s$ we refer to \cite{du2012analysis,JRPZ,RZ-low-reg,CRTZ,KLZ}. 
The works \cite{JRPZ,RZ-low-reg,CRTZ} analyze the \emph{inverse fractional conductivity problem}, which consists in recovering a uniformly elliptic conductivity $\gamma\colon\mathbb{R}^n\to\mathbb{R}$ in
\[
    \Div_s(\Theta_{\gamma}\nabla^s u)=0\quad\text{in }\Omega,
\]
where $\Theta_{\gamma}(x,y)=\gamma^{1/2}(x)\gamma^{1/2}(y)$, from the associated DN map $\Lambda_\gamma$. 
The article \cite{KLZ} studies the Calderón problem for a weighted fractional $p$-Laplacian and can be viewed as a nonlinear extension of the inverse fractional conductivity problem.

We next introduce another notion of fractional gradient, the \emph{Riesz $s$-fractional gradient} $\text{D}^s$, which--unlike $\nabla^s$--associates to each $u\colon\mathbb{R}^n\to\mathbb{R}$ a vector field $\text{D}^s u\colon\mathbb{R}^n\to\mathbb{R}^n$ (see, e.g., \cite[p.~246]{ponce2016elliptic}, \cite{almi2025rieszfractionalgradientfunctionals,du2012analysis}).
\begin{definition}
\label{def:fra-gradient-2}
Let $0<s<1$. For $u\in C^1_c(\mathbb{R}^n)$, the \emph{Riesz fractional gradient of order $s$} of $u$ is the vector field
\[
    \text{D}^s u\colon \mathbb{R}^n\to \mathbb{R}^n
\]
defined by
\begin{equation}
\text{D}^s u(x)\vcentcolon= \mu_s\int_{\mathbb{R}^n}\frac{u(x)-u(y)}{|x-y|^{n+s+1}}(x-y)\,dy,
\end{equation}
where $\mu_s= \frac{2^s \Gamma\left(\frac{n+s+1}{2}\right)}{\pi^{\frac{n}{2}}\Gamma\left(\frac{1-s}{2}\right)}>0$. 
For $v\in C^1_c(\mathbb{R}^n;\mathbb{R}^n)$, the \emph{Riesz fractional divergence of order $s$} of $v$ is the scalar field
\[
\text{Div}_s v\colon\mathbb{R}^n\to\mathbb{R}
\]
given by
\[
\text{Div}_s v(x)\vcentcolon = \mu_s\int_{\mathbb{R}^n}\frac{v(x)-v(y)}{|x-y|^{n+s+1}}\cdot (x-y)\,dy.
\]
\end{definition}
These operators satisfy the integration-by-parts identity
\begin{equation}
\label{eq: integration by parts formula}
\int_{\mathbb{R}^n}u\,\text{Div}_s v\,dx
    =-\int_{\mathbb{R}^n}\text{D}^s u\cdot v\,dx
\end{equation}
for all $u\in C^1_c(\mathbb{R}^n)$ and $v\in C^1_c(\mathbb{R}^n;\mathbb{R}^n)$ (see \cite[Proposition~2.3]{almi2025rieszfractionalgradientfunctionals}). Furthermore, one has
\begin{equation}
\label{eq: rewriting fractional operators}
    \begin{split}
        \text{D}^s u(x)&=\nabla (-\Delta)^{-(1-s)/2}u (x)
                  =(-\Delta)^{-(1-s)/2}\nabla u(x),\\
        \text{Div}_s v(x)&=\Div(-\Delta)^{-(1-s)/2}v(x)
                  =(-\Delta)^{-(1-s)/2}\Div v(x)
    \end{split}
\end{equation}
for $u\in C^1_c(\mathbb{R}^n)$ and $v\in C^1_c(\mathbb{R}^n;\mathbb{R}^n)$; see \cite[Proposition~2.9]{almi2025rieszfractionalgradientfunctionals}. Here $(-\Delta)^{-\alpha/2}$, $0<\alpha<n$, denotes the \emph{Riesz potential},
\begin{equation}
\label{eq: Riesz potential}
    (-\Delta)^{-\alpha/2}f(x)\vcentcolon 
    = \frac{1}{c_{n,\alpha}}\int_{\mathbb{R}^n}\frac{f(y)}{|x-y|^{n-\alpha}}\,dy,
\end{equation}
defined for $f\in L^p(\mathbb{R}^n)$ with $1\leq p<n/\alpha$ and
$c_{n,\alpha}=2^{\alpha}\pi^{\frac{n}{2}}\frac{\Gamma\left(\frac{\alpha}{2}\right)}{\Gamma\left(\frac{n-\alpha}{2}\right)}$.
Our notation emphasizes that $(-\Delta)^{-\alpha/2}$ is the inverse of $(-\Delta)^{\alpha/2}$ in the sense of Fourier multipliers, with symbol $|\xi|^{-\alpha}$.

The representation \eqref{eq: rewriting fractional operators} has several consequences. For instance, the Hardy--Littlewood--Sobolev inequality for the Riesz potential yields $L^{p^\ast}$–$L^p$ estimates for $\text{D}^s$ and $\text{Div}_s$ (see \cite[Proposition~2.2]{almi2025rieszfractionalgradientfunctionals} for interpolation-type $L^p$ estimates). Moreover, for $f\in C_c^\infty(\mathbb{R}^n)$ we have
\[
    \lim_{\alpha\to 0}(-\Delta)^{-\alpha/2}f(x)=f(x),
\]
which implies
\[
    \text{D}^s u(x)\to \nabla u(x)\quad\text{and}\quad \text{Div}_s v(x)\to \Div v(x)
\]
as $s\to 1$. In analogy with \eqref{eq: div grad equal frac lap} for $(\nabla^s,\Div_s)$, one also has
\[
    \text{Div}_s(\text{D}^s u)(x)=(-\Delta)^s u(x)
\]
for all $u\in C^{\infty}_c(\mathbb{R}^n)$.

Unlike $\nabla^s$, the Riesz fractional gradient $\text{D}^s$ does not satisfy an exact product rule. Nevertheless, for all $u,v\in C^1_c(\mathbb{R}^n)$ and $w\in C^1_c(\mathbb{R}^n;\mathbb{R}^n)$ one has
\[
   \begin{split}
       \text{D}^s(uv)(x)
           &=u(x)\text{D}^s v(x)+v(x)\text{D}^s u(x)+\text{D}^s_{\mathrm{NL}}(u,v)(x),\\
       \text{Div}_s(uw)(x)
           &=w(x)\cdot \text{D}^s u(x)+u(x)\text{Div}_s w(x)
             +\text{Div}_{s,\mathrm{NL}}(w,u)(x)
   \end{split}
\]
for all $x\in\mathbb{R}^n$, where
\[
\begin{split}
    \text{D}^s_{\mathrm{NL}}(u,v)(x)
        &\vcentcolon = -\mu_s\int_{\mathbb{R}^n}
            \frac{(u(x)-u(y))(v(x)-v(y))}
                 {|x-y|^{n+s+1}}(x-y)\,dy,\\[2mm]
    \text{Div}_{s,\mathrm{NL}}(w,u)(x)
        &\vcentcolon = -\mu_s\int_{\mathbb{R}^n}
            \frac{(u(x)-u(y))\,(w(x)-w(y))}
                 {|x-y|^{n+s+1}}\cdot (x-y)\,dy
\end{split}
\]
(see \cite[Proposition~2.6]{almi2025rieszfractionalgradientfunctionals}).

We conclude this section with several natural nonlocal Kuznetsov-type JMGT models that we plan to investigate in future work. As in the present article, the principal part is given by
\[
L^s_{\tau,\alpha,b,c}u\vcentcolon=\tau\partial_t^3 u+\alpha\partial_t^2 u
+b(-\Delta)^s\partial_t u+c\,(-\Delta)^s u,
\]
while the nonlinearity is taken in the form
\[
    \mathcal{G}(u,\partial_t u, \partial_t^2 u)
        =\partial_t\bigl(\mathcal{F}(\partial_t u)+\mathcal{N}(u)\bigr),
    \qquad \mathcal{F}(\partial_t u)=\kappa(\partial_t u)^2
\]
(see \eqref{sysmgtWtype}). For the spatial part $\mathcal{N}(u)$ we may consider, for instance,
\begin{enumerate}[{(N}1)]
    \item\label{N1} $\displaystyle \mathcal{N}(u)(x)
        =\int_{\mathbb{R}^n}\gamma(x,y)\,|\nabla^s u(x,y)|^2\,dy,$
    \item\label{N2} $\displaystyle \mathcal{N}(u)(x)
        =\theta(x)\,|\text{D}^s u(x)|^2,$
\end{enumerate}
for suitable coefficients $\gamma$ and $\theta$. Note that, upon integration in $x\in\mathbb{R}^n$ and taking the square root, the nonlinearity in \ref{N1} yields a weighted variant of the Gagliardo–Slobodeckij seminorm
\[
    [u]_{W^{s,2}(\mathbb{R}^n)}
        \vcentcolon =\left(\int_{\mathbb{R}^{2n}}
        \frac{|u(x)-u(y)|^2}{|x-y|^{n+2s}}\,dx\,dy\right)^{1/2}.
\]
Moreover, by \eqref{eq: equivalent norm Hs} and \cite[Proposition~3.4]{EPV}, the quantity
\[
    \|u\|_{W^{s,2}(\mathbb{R}^n)}
        \vcentcolon =
        \Bigl(\|u\|_{L^2(\mathbb{R}^n)}^2
              +[u]_{W^{s,2}(\mathbb{R}^n)}^2\Bigr)^{1/2}
\]
defines an equivalent norm on $H^s(\mathbb{R}^n)$. The inverse problem of determining the nonlinear coefficients $\gamma$ and $\theta$ from the associated DN data is highly nontrivial and, in our view, very natural. We refer the interested reader also to the forthcoming work \cite{PZ-KK} for related results.

\medskip

\appendix

\section{Method of parabolic regularization and its applications}
\label{appendix: parabolic regularization}

In this section, we carry out the parabolic regularization of third order nonlocal wave equations and show that in the vanishing viscosity limit $(\eps\to 0)$ the solutions $u_\eps$ converge to the unique solution of the original problem. Furthermore, we establish an integration by parts formula for third order derivatives.
    
    \subsection{Parabolic regularization}
    
    The main result of this section is the following, which will be established similarly as the corresponding assertions for second order nonlocal wave equations and to that effect we closely follow the strategy presented in \cite{PZ-KK}.

    \begin{theorem}[Parabolic regularization]
    \label{thm: parabolic regularization}
        Let $\Omega\subset\R^n$ be a bounded Lipschitz domain, $T>0$ and $s\in\R_+$. Suppose that $q\in L^{\infty}(0,T;L^p(\Omega))$ is a real-valued potential with $2\leq p\leq \infty$ satisfying \eqref{conp} and 
        \[
            (u_0,u_1,u_2,F)\in \widetilde{H}^s(\Omega)\times \widetilde{H}^s(\Omega)\times \widetilde{L}^2(\Omega)\times L^2(0,T;\widetilde{L}^2(\Omega)).
        \]
        Under the above assumptions, for any $\eps>0$ there exists a unique solution
         \begin{equation}
        \label{eq: regularity weak sol parabolic reg}
            u_\eps\in W^{1,\infty}(0,T;\widetilde{H}^s(\Omega))\cap W^{2,\infty}(0,T;\widetilde{L}^2(\Omega))\cap H^2(0,T;\widetilde{H}^s(\Omega))
        \end{equation}
        of the parabolically regularized problem
        \begin{equation}
    \label{parabolically regularized linear MGT without ext cord}
	\begin{cases}
		(\partial_t^3+\eps (-\Delta)^s\partial_t^2+\alpha\partial_t^2+b(-\Delta)^s\partial_t+c(-\Delta)^s+q)u=F\ & {\rm in}\ \Omega_T,\\
		u=0\ & {\rm in}\ (\Omega_e)_T, \\
		u(0)=u_0, \partial_t u(0)=u_1, \partial_t^2 u(0)=u_2\ & {\rm in}\ \Omega.
	\end{cases}
	\end{equation}
     If we denote by $u$ the unique solution of \eqref{linear MGT without ext cord} (Theorem \ref{wellposednessv}), then we see that the family $(u_\eps)_{\eps>0}$ obeys the following convergence properties:
        \begin{enumerate}[{(C}1)]
              \item\label{conv 1: par} $u_\eps\weakstar u$ in $L^{\infty}(0,T;\widetilde{H}^s(\Omega))$,
            \item\label{conv 2: par}  $\partial_t u_\eps\weakstar \partial_t u$ in $L^{\infty}(0,T;\widetilde{H}^s(\Omega))$,
            \item\label{conv 3: par}  $\partial_t^2 u_\eps\weakstar \partial_t^2 u$ in $L^{\infty}(0,T;\widetilde{L}^2(\Omega))$
            \item\label{conv 4: par}  $\partial_t^3 u_\eps\weakstar \partial_t^3 u$ in $L^{2}(0,T;H^{-s}(\Omega))$
        \end{enumerate}
        as $\eps\to 0$.
    \end{theorem}

    \begin{proof} For the sake of readability, we divide the proof of Theorem \ref{thm: parabolic regularization} into several steps. As some parts of the proof are very similar to the one of Theorem \ref{wellposednessv}, we are sometimes rather brief.\\
    
        \noindent\textbf{Step 1.} \textit{Existence of $u_\eps$.} We again use the Galerkin method to show the existence of a solution $u_\eps$ to the problem \eqref{parabolically regularized linear MGT without ext cord} and hence we use the same notation as in the proof of Theorem \ref{wellposednessv}. To shorten the notation, we write in this step $u$ for the solution $u_\eps$. From Step 1 of the proof of Theorem \ref{wellposednessv}, we know that there exist sequences $(u^m_{j})_{m\in\N}\subset\widetilde{H}^s(\Omega)$, $0\leq j\leq 2$, such that
        \begin{equation}
        \label{eq: convergence intial cond reg}
            u^m_{j}\in\widetilde{H}^s_m \text{ for }0\leq j\leq 2\text{ and }\begin{cases}
                u^m_0\to u_0\text{ in }\widetilde{H}^s(\Omega),\\
                 u^m_1\to u_1\text{ in }\widetilde{H}^s(\Omega),\\
                  u^m_2\to u_2\text{ in }\widetilde{L}^2(\Omega)
            \end{cases}
        \end{equation}
        as $m\to\infty$. Next, we set
        \[
            u_m=\sum_{j=1}^m c_m^jw_j\in\widetilde{H}^s_m,
        \]
        and look for coefficients $(c_m^j)_{1\leq j\leq m}$ such that $u_m\in \widetilde{H}^s_m$ solves
        \begin{equation}
        \label{eq: approximate problem par reg}
        \begin{cases}
            \langle \partial_t^3 u_m(t),w_j\rangle_{L^2(\Omega)}+\eps\langle (-\Delta)^s \partial_t^2 u_m(t),w_j\rangle+\langle L_{\alpha,b,c,q}u_m(t),w_j\rangle=\langle F(t),w_j\rangle_{L^2(\Omega)},\\
            u_m(0)=u^m_0,\,\partial_t u_m(0)=u^m_1,\,\partial_t^2 u_m(0)=u^m_2
        \end{cases}
        \end{equation}
        for all $1\leq j\leq m$ and $0<t<T$. The above ODE problem can be solved exactly in the same way as problem \eqref{eq: approximate problem}. Hence, we may deduce the existence of a solution $u_m\in C^2([0,T];\widetilde{H}^s_m)\cap H^3(0,T;\widetilde{H}^s_m)$ to \eqref{eq: approximate problem par reg}.
        
        If we multiply \eqref{eq: approximate problem par reg} by  $\partial_t^2 c^j_m$ and sum $j$ over $\{1,\ldots,m\}$, then we arrive at the identity
        \[
        \begin{split}
            &\partial_t\frac{\|\partial_t^2 u_m(\tau)\|_{L^2(\Omega)}^2}{2}+\eps\| (-\Delta)^{s/2}\partial_t^2 u_m(\tau)\|_{L^2(\R^n)}^2+\langle L_{\alpha,b,c,q}u_m(\tau),\partial_t^2 u_m(\tau)\rangle\\
            &\quad =\langle F(\tau),\partial_t^2 u_m(\tau)\rangle_{L^2(\Omega)}
        \end{split}
        \]
        for all $0<\tau<T$. By integrating the previous identity over $[0,t]$, we get the estimate
        \begin{equation}
        \label{eq: prelim estimate for Gronwall}
                \begin{split}
             &\frac{1}{2}\|\partial_t^2 u_m(t)\|_{L^2(\Omega)}^2+\frac{b}{4}\left(\|(-\Delta)^{s/2}\partial_t u_m(t)\|_{L^2(\R^n)}^2+\|(-\Delta)^{s/2} u_m(t)\|_{L^2(\R^n)}^2\right)\\
             &+\eps\| (-\Delta)^{s/2}\partial_t^2 u_m\|_{L^2(\R^n_t)}^2\\
             &\leq C(\|(-\Delta)^{s/2}u_0\|_{L^2(\R^n)}^2+\|(-\Delta)^{s/2}u_1\|_{L^2(\R^n)}^2+\|u_2\|_{L^2(\Omega)}^2+\|F\|_{L^2(0,T;L^2(\Omega))}^2)\\
             &\quad +C\int_0^t(\|(-\Delta)^{s/2}u_m\|_{L^2(\R^n)}^2+\|(-\Delta)^{s/2}\partial_t u_m\|_{L^2(\R^n)}^2+\|\partial_t^2 u_m\|_{L^2(\Omega)}^2)\,d\tau
        \end{split}
        \end{equation}
        for any $0<t<T$ and some $C>0$ independent of $m\in\N$. For more details on the derivation of the above estimate, we refer to Step 2 in the proof of Theorem \ref{wellposednessv}. Hence, Gronwall's inequality ensures that we have
        \begin{equation}
        \label{eq: Gronwall estimate}
         \begin{split}
            &\|\partial_t^2 u_m(t)\|_{L^2(\Omega)}^2+\|(-\Delta)^{s/2}\partial_t u_m(t)\|_{L^2(\R^n)}^2+\|(-\Delta)^{s/2} u_m(t)\|_{L^2(\R^n)}^2\\
            &\leq C(\|(-\Delta)^{s/2}u_0\|_{L^2(\R^n)}^2+\|(-\Delta)^{s/2}u_1\|_{L^2(\R^n)}^2+\|u_2\|_{L^2(\Omega)}^2+\|F\|_{L^2(0,T;L^2(\Omega))}^2)
        \end{split}
        \end{equation}
        for all $0\leq t\leq T$ and large $m\in\N$. Combining \eqref{eq: Gronwall estimate} with \eqref{eq: prelim estimate for Gronwall}, we deduce that
        \begin{equation}
        \label{eq: uniform estimate in eps and m}
            \eps \| (-\Delta)^{s/2}\partial_t^2 u_m\|_{L^2(\R^n_T)}^2\leq C
        \end{equation}
        for some $C>0$ independent of $m$. Therefore, we have shown that the solution $u_m$ of problem \eqref{eq: approximate problem par reg} shares the following properties:
        \begin{enumerate}[{(B}1)]
            \item $u_m$ is uniformly bounded in $W^{1,\infty}(0,T;\widetilde{H}^s(\Omega))$,
            \item $\partial_t^2 u_m$ is uniformly bounded in $L^{\infty}(0,T;\widetilde{L}^2(\Omega))$ and $L^2(0,T;\widetilde{H}^s(\Omega))$. 
        \end{enumerate}
        Thus, we may extract a subsequence, still denoted by $(u_m)$, and a function
        \[
          u\in W^{1,\infty}(0,T;\widetilde{H}^s(\Omega))\cap W^{2,\infty}(0,T;\widetilde{L}^2(\Omega))\cap H^2(0,T;\widetilde{H}^s(\Omega))
        \]
        such that
        \begin{enumerate}[{(c}1)]
            \item $u_m\weakstar u$ in $L^{\infty}(0,T;\widetilde{H}^s(\Omega))$,
            \item $\partial_t u_m\weakstar \partial_t u$ in $L^{\infty}(0,T;\widetilde{H}^s(\Omega))$,
            \item $\partial_t^2 u_m\weakstar \partial_t^2 u$ in $L^{\infty}(0,T;\widetilde{L}^2(\Omega))$
            \item and $\partial_t^2 u_m \weak \partial_t^2 u$ in $L^2(0,T;\widetilde{H}^s(\Omega))$
        \end{enumerate}
        as $m\to\infty$. Using these convergence properties, it can easily be deduced from \eqref{eq: approximate problem par reg} that the found function $u$ solves \eqref{parabolically regularized linear MGT without ext cord}, i.e. for any $v\in \widetilde{H}^s(\Omega)$ there holds
        \begin{equation}
        \label{eq: weak for par reg problem}
            \begin{split}
                &\partial_t \langle \partial_t^2 u,v\rangle_{L^2(\Omega)}+\eps \langle (-\Delta)^{s/2}\partial_t^2 u,(-\Delta)^{s/2}v\rangle_{L^2(\R^n)}+\alpha\langle \partial_t^2 u,v\rangle_{L^2(\Omega)}\\
                &+b\langle (-\Delta)^{s/2} \partial_t u,(-\Delta)^{s/2}v\rangle_{L^2(\R^n)}+c\langle (-\Delta)^{s/2} u,(-\Delta)^{s/2}v\rangle_{L^2(\R^n)}\\
                &+\langle q u,v\rangle_{L^2(\Omega)}=\langle F,v\rangle_{L^2(\Omega)}
            \end{split}
        \end{equation}
        in the sense of distributions on $(0,T)$ and one has
        \[
            u(0)=u_0\text{ in }\widetilde{H}^s(\Omega),\,\partial_t u(0)=u_1\text{ in }\widetilde{H}^s(\Omega)\text{ and }\partial_t^2 u(0)=u_2\text{ in }\widetilde{L}^2(\Omega)
        \]
        (cf.~e.g.~\cite[p.~564-566]{DautrayLionsVol5}). Furthermore, as shown in the next step, the function $u$ is unique and so the whole sequence $(u_m)_{m\in\N}$ converges to $u$ and not only a particular subsequence of $(u_m)_{m\in\N}$. \\

        \noindent\textbf{Step 2.} \textit{Uniqueness of $u_\eps$.} Note that by linearity of the problem it suffices to show that whenever $u=u_\eps$, satisfying \eqref{eq: regularity weak sol parabolic reg}, solves \eqref{parabolically regularized linear MGT without ext cord} for $u_0=u_1=u_2=F=0$, then $u=0$. Moreover, observe that \eqref{parabolically regularized linear MGT without ext cord}, \eqref{eq: regularity weak sol parabolic reg} and an approximation argument ensure that $\partial_t^3 u\in L^2(0,T;H^{-s}(\Omega))$ and \eqref{parabolically regularized linear MGT without ext cord} holds in the sense of $L^2(0,T;H^{-s}(\Omega))$. Taking into account $\partial_t^2 u\in L^2(0,T;\widetilde{H}^s(\Omega))$, $\partial_t^3u\in L^2(0,T;H^{-s}(\Omega))$ and the integration by parts formula, then we get from \eqref{parabolically regularized linear MGT without ext cord} the identity
        \[
        \begin{split}
            \frac{\|\partial_t^2 u(t)\|_{L^2(\Omega)}^2}{2}&=\int_0^t\langle \partial_t^3 u,\partial_t^2 u\rangle d\tau\\
            &=-\eps \|(-\Delta)^{s/2}\partial_t^2 u\|_{L^2(\R^n_t)}^2-\alpha \|\partial_t^2 u\|_{L^2(\Omega_t)}^2-\frac{b}{2}\|(-\Delta)^{s/2}\partial_t u(t)\|_{L^2(\R^n)}^2\\
            &\quad -c\langle (-\Delta)^{s/2} u,(-\Delta)^{s/2}\partial_t^2 u\rangle_{L^2(\R^n_t)}+\langle qu,\partial_t^2 u\rangle_{L^2(\Omega_t)}.
        \end{split}
        \]
        Hence, we get 
        \begin{equation}
        \label{eq: energy estimate for uniqueness}
        \begin{split}
            &\frac{\|\partial_t^2 u(t)\|_{L^2(\Omega)}^2}{2}+\frac{b}{2}(\|(-\Delta)^{s/2}\partial_t u(t)\|_{L^2(\R^n)}^2+\|(-\Delta)^{s/2}u(t)\|_{L^2(\R^n)}^2)\\
            &+\frac{\eps}{2} \|(-\Delta)^{s/2}\partial_t^2 u\|_{L^2(\R^n_t)}^2\lesssim \int_0^t(\|\partial_t^2 u\|_{L^2(\Omega)}^2+\|(-\Delta)^{s/2}u\|_{L^2(\R^n)}^2)d\tau.
        \end{split}
        \end{equation}
        Here, we used Lemma \ref{lemma: potential} and the bound
        \[
          \begin{split}
              -c\langle (-\Delta)^{s/2} u,(-\Delta)^{s/2}\partial_t^2 u\rangle_{L^2(\R^n_t)}&\leq C_\eps \|(-\Delta)^{s/2} u\|_{L^2(\R^n_t)}^2+\frac{\eps}{2}\|(-\Delta)^{s/2}\partial_t^2 u\|_{L^2(\R^n_t)}^2.
          \end{split} 
        \]
        Thanks to \eqref{eq: energy estimate for uniqueness} and Gronwall's inequality we may conclude that $u=0$. Hence, the solutions $u_\eps$, constructed in Step 1, are the unique solutions of \eqref{parabolically regularized linear MGT without ext cord}.\\
        
        \noindent\textbf{Step 3.} \textit{Convergence properties of $u_\eps$.} First of all, by repeating the arguments of Step 2, we notice that the solution $u_\eps$ of \eqref{parabolically regularized linear MGT without ext cord} satisfies
        \begin{equation}
        \label{eq: almost final energy estimate for convergence of u eps}
        \begin{split}
            &\frac{\|\partial_t^2 u_\eps(t)\|_{L^2(\Omega)}^2}{2}+\frac{b}{2}\|(-\Delta)^{s/2}\partial_t u_\eps(t)\|_{L^2(\R^n)}^2+\eps \|(-\Delta)^{s/2}\partial_t^2 u_\eps\|_{L^2(\R^n_t)}^2\\
            &= \frac{\|\partial_t^2 u_2\|_{L^2(\Omega)}^2}{2}+\frac{b}{2}\|(-\Delta)^{s/2}u_1\|_{L^2(\R^n)}^2-c\langle (-\Delta)^{s/2} u_\eps,(-\Delta)^{s/2}\partial_t^2 u_\eps\rangle_{L^2(\R^n_t)}\\
            &\quad -\alpha \|\partial_t^2 u_\eps\|_{L^2(\Omega_t)}^2+\langle qu_\eps,\partial_t^2 u_\eps\rangle_{L^2(\Omega_t)}+\langle F,\partial_t^2 u_\eps\rangle_{L^2(\Omega_t)}\\
            &\leq \frac{\|\partial_t^2 u_2\|_{L^2(\Omega)}^2}{2}+\frac{b}{2}\|(-\Delta)^{s/2}u_1\|_{L^2(\R^n)}^2+\|F\|_{L^2(\Omega_T)}^2\\
            &\quad +C\int_0^t (\|(-\Delta)^{s/2}u_\eps\|_{L^2(\R^n)}^2+\|\partial_t^2 u_\eps\|_{L^2(\Omega)}^2)\,d\tau\\
            &\quad -c\langle (-\Delta)^{s/2} u_\eps,(-\Delta)^{s/2}\partial_t^2 u_\eps\rangle_{L^2(\R^n_t)}.
        \end{split}
        \end{equation}
        As $u_\eps\in H^2(0,T;\widetilde{H}^s(\Omega))$, we may apply the integration by parts formula to write
        \[
        \begin{split}
            -c\langle (-\Delta)^{s/2} u_\eps,(-\Delta)^{s/2}\partial_t^2 u_\eps\rangle_{L^2(\R^n_t)}&=-c\langle (-\Delta)^{s/2} u_\eps(t),(-\Delta)^{s/2}\partial_t u_\eps(t)\rangle_{L^2(\R^n)}\\
            &\quad +c\langle (-\Delta)^{s/2} u_0,(-\Delta)^{s/2} u_1\rangle_{L^2(\R^n)}\\
            &\quad +c\int_0^t\|(-\Delta)^{s/2}\partial_t u_\eps\|_{L^2(\R^n)}^2\,d\tau.
        \end{split}
        \]
        Using $u_\eps \in C^1([0,T];\widetilde{H}^s(\Omega))$, we can perform the computations in \eqref{eq: estimate product} for $u_\eps$ and get
        \[
        \begin{split}
         -c\langle (-\Delta)^{s/2} u_\eps,(-\Delta)^{s/2}\partial_t^2 u_\eps\rangle_{L^2(\R^n_t)}&\leq C(\|(-\Delta)^{s/2}u_0\|_{L^2(\R^n)}^2+ \|(-\Delta)^{s/2}u_1\|_{L^2(\R^n)}^2)\\
         &\quad +C\int_0^t\|(-\Delta)^{s/2}\partial_t u_\eps\|_{L^2(\R^n)}^2\,d\tau\\
         &\quad +\eta \|(-\Delta)^{s/2}\partial_t u_\eps(t)\|_{L^2(\R^n)}^2
         \end{split}
        \]
        for any $\eta>0$. If we select $\eta=b/4$, then we get from \eqref{eq: almost final energy estimate for convergence of u eps} the bound
        \begin{equation}
        \label{eq: uniform bound for u eps}
        \begin{split}
             &\frac{\|\partial_t^2 u_\eps(t)\|_{L^2(\Omega)}^2}{2}+\frac{b}{4}\|(-\Delta)^{s/2}\partial_t u_\eps(t)\|_{L^2(\R^n)}^2+\eps \|(-\Delta)^{s/2}\partial_t^2 u_\eps\|_{L^2(\R^n_t)}^2\\
             &\lesssim \|(-\Delta)^{s/2}u_0\|_{L^2(\R^n)}^2+ \|(-\Delta)^{s/2}u_1\|_{L^2(\R^n)}^2+\|\partial_t^2 u_2\|_{L^2(\Omega)}^2+\|F\|_{L^2(\Omega_T)}^2\\
             &\quad +\int_0^t (\|(-\Delta)^{s/2}u_\eps\|_{L^2(\R^n)}^2+\|\partial_t^2 u_\eps\|_{L^2(\Omega)}^2)\,d\tau.
        \end{split}
        \end{equation}
        Thus, Gronwall's inequality and the usual compactness arguments demonstrate the existence of a function 
        \[
            u\in W^{1,\infty}(0,T;\widetilde{H}^s(\Omega))\cap W^{2,\infty}(0,T;\widetilde{L}^2(\Omega))  
        \]
        such that
        \begin{enumerate}[(i)]
            \item\label{i u eps conv} $u_\eps\weakstar u$ in $L^{\infty}(0,T;\widetilde{H}^s(\Omega))$,
            \item\label{ii u eps conv} $\partial_t u_\eps\weakstar \partial_t u$ in $L^{\infty}(0,T;\widetilde{H}^s(\Omega))$,
            \item\label{iii u eps conv} $\partial_t^2 u_\eps\weakstar \partial_t^2 u$ in $L^{\infty}(0,T;\widetilde{L}^2(\Omega))$,
            \item\label{iv u eps conv} and $\eps^{1/2}\partial_t^2 u_\eps$ is uniformly bounded in $L^2(0,T;\widetilde{H}^s(\Omega))$.
        \end{enumerate}
        Here the convergence results a priori only hold for a suitable subsequence, but which we still denote by $u_\eps$. Using these convergence results, it is straightforward to deduce that $u$ solves \eqref{linear MGT without ext cord} (cf.~\eqref{eq: weak for par reg problem} and Definition~\ref{def: weak sol}). As the limit $u$ is the unique solution of \eqref{linear MGT without ext cord}, it is independent of the subsequence and hence the whole sequence $(u_\eps)_{\eps>0}$ satisfies the above convergence properties. The last convergence assertion \ref{conv 4: par} follows from the parabolically regularized problem and \ref{i u eps conv}--\ref{iv u eps conv}. This completes the proof.
    \end{proof}

    \subsection{Integration by parts formula for third order derivatives}

   Using Theorem~\ref{thm: parabolic regularization}, we can now demonstrate the integration by parts formula that we used in the proof of the Runge approximation theorem (Proposition~\ref{Runge}).

    \begin{proposition}
    \label{prop: integration by parts formula}
         Let $\Omega\subset\R^n$ be a bounded Lipschitz domain, $T>0$ and $s\in\R_+$. Suppose that $q_1,q_2\in L^{\infty}(0,T;L^p(\Omega))$ are real-valued potentials with $2\leq p\leq \infty$ satisfying \eqref{conp} and $F,G\in L^2(\Omega_T)$. Furthermore, assume that $q_2$ is time-reversal invariant and that
         \[
            u,v\in  W^{1,\infty}(0,T;\widetilde{H}^s(\Omega))\cap W^{2,\infty}(0,T;\widetilde{L}^2(\Omega))
         \]
         are the unique solutions of
         \[
         \begin{cases}
		(\partial_t^3+\alpha\partial_t^2+b(-\Delta)^s\partial_t+c(-\Delta)^s+q_1)u=F\ & {\rm in}\ \Omega_T,\\
		u=0\ & {\rm in}\ (\Omega_e)_T, \\
		u(0)=\partial_t u(0)=\partial_t^2 u(0)=0\ & {\rm in}\ \Omega.
	\end{cases}
         \]
         and
         \[
         \begin{cases}
		(\partial_t^3-\alpha\partial_t^2+b(-\Delta)^s\partial_t-c(-\Delta)^s-q_2)v=G\ & {\rm in}\ \Omega_T,\\
		v=0\ & {\rm in}\ (\Omega_e)_T, \\
		v(T)=\partial_t v(T)=\partial_t^2 v(T)=0\ & {\rm in}\ \Omega
	\end{cases}
         \]
         (see~Theorem~\ref{wellposednessv}). Then we have the following integration by parts formula:
        \begin{equation}
        \label{eq: integration by parts appendix}
            \int_0^T \langle \partial_t^3 u,v\rangle\,dt=-\int_0^T \langle u,\partial_t^3 v\rangle \,dt.
        \end{equation}
    \end{proposition}

    \begin{proof}
        We start by recalling that we have $v=w^\star$, where $w$ is the unique solution of 
        \[
        \begin{cases}
		(\partial_t^3+\alpha\partial_t^2+b(-\Delta)^s\partial_t+c(-\Delta)^s+q_2)w=-G^\star\ & {\rm in}\ \Omega_T,\\
		w=0\ & {\rm in}\ (\Omega_e)_T, \\
		w(0)=\partial_t w(0)=\partial_t^2 w(0)=0\ & {\rm in}\ \Omega.
	\end{cases}
        \]
        By Theorem \ref{thm: parabolic regularization} for any $\eps>0$ there exist 
        \[
            u_\eps,w_\eps\in W^{1,\infty}(0,T;\widetilde{H}^s(\Omega))\cap W^{2,\infty}(0,T;\widetilde{L}^2(\Omega))\cap H^2(0,T;\widetilde{H}^s(\Omega))
        \]
        solving the associated parabolically regularized problems and having the convergence properties \ref{conv 1: par}--\ref{conv 4: par}. Moreover, let us set $v_\eps=w^\star_\eps$. Next, notice that the integration by parts formula yields
        \[
             \int_0^T \langle \partial_t^3 u_\eps,v_\eps\rangle\,dt=-\int_0^T \langle u_\eps,\partial_t^3 v_\eps\rangle \,dt
        \]
        for any $\eps>0$. Thanks to the convergence properties of $u_\eps$ and $v_\eps$, we may pass to the limit $\eps\to 0$ in the previous identity and obtain \eqref{eq: integration by parts appendix}.
    \end{proof}

\medskip

\noindent{\bf Acknowledgments}  
\begin{itemize}
    \item S. Fu is supported by the  Fundamental Research Funds for the Central Universities, NPU, under grant number D5000250416, and the key program of the National Natural Science Foundation of China under grant number 62433020.
    \item Y. Yu is supported by the National Natural Science Foundation of China under grant number 12401579.

    \item P. Zimmermann is supported by the Swiss National Science Foundation (SNSF), under the grant number 214500.
\end{itemize}

  \section*{Statements and Declarations}
	
	\subsection*{Data availability statement}
	No datasets were generated or analyzed during the current study.
	
	\subsection*{Conflict of Interests} Hereby, we declare that there is no conflict of interest.

\bibliography{refs} 
	
	\bibliographystyle{alpha}
	
\end{document}